\documentclass[11pt]{article}

\usepackage{amsmath}
\usepackage{amssymb}
\usepackage{amsthm}
\usepackage{latexsym}
\usepackage{color}
\usepackage{mathtools}
\usepackage{dsfont}
\usepackage{graphicx}
\usepackage{appendix}
\usepackage{color} 
\usepackage{enumerate}
\usepackage{diagbox}
\usepackage[T1]{fontenc}
\usepackage[english]{babel}
\usepackage[table]{xcolor}
\usepackage{calrsfs}
\usepackage{geometry}
\usepackage{url}
\DeclareSymbolFont{calletters}{OMS}{cmsy}{m}{n}
\DeclareSymbolFontAlphabet{\mathcal}{calletters}
\DeclareMathAlphabet{\mathpzc}{OT1}{pzc}{m}{it}

\def\be{\begin{eqnarray}}
\def\ee{\end{eqnarray}}

\def\b*{\begin{eqnarray*}}
\def\e*{\end{eqnarray*}}

\newtheorem{Theorem}{Theorem}[part]
\newtheorem{Definition}{Definition}[part]
\newtheorem{Proposition}{Proposition}[part]

\newtheorem{Assumption}{Assumption}[part]
\newtheorem*{Assumptionn}{Assumption}
\newtheorem{Lemma}{Lemma}[part]
\newtheorem{Corollary}{Corollary}[part]
\newtheorem{Remark}{Remark}[part]
\newtheorem{Method}{Method}
\newtheorem{Example}{Example}[part]

\makeatletter \@addtoreset{equation}{section}

\@addtoreset{Definition}{section}

\@addtoreset{Theorem}{section}

\@addtoreset{Proposition}{section}

\@addtoreset{Property}{section}

\@addtoreset{Assumption}{section}

\@addtoreset{Corollary}{section}

\@addtoreset{Lemma}{section}

\@addtoreset{Remark}{section}

\@addtoreset{Example}{section}

\@addtoreset{Condition}{section}

\newcommand{\No}[1]{\left\|#1\right\|}     
\newcommand{\abs}[1]{\left|#1\right|}     

\def \E{\mathbb{E}}
\def \F{\mathbb{F}}

\def \P{\mathbb{P}}
\def \Q{\mathbb{Q}}
\def \R{\mathbb{R}}

\def \G{\mathbb{G}}
\def \w{\mathsf{w}}

\def\Ac{{\cal A}}

\def\Cc{{\cal C}}

\def\Ec{{\cal E}}
\def\Fc{{\cal F}}

\def\Kc{{\cal K}}

\def\Pc{{\cal P}}
\def\Qc{{\cal Q}}

\def\Yc{{\cal Y}}
\def\Zc{{\cal Z}}

\def\x{\times}

\def\={\;=\;}
\def\.{\;.}

\def\reff#1{{\rm(\ref{#1})}}

\def\1{{\bf 1}}

\def\b*{\begin{eqnarray*}}
\def\e*{\end{eqnarray*}}

 \def\normeL2#1{\left\|{#1}\right\|_{L^2}}
 
 \title{Moral Hazard under Ambiguity}

 \author{Thibaut {\sc Mastrolia} \footnote{CMAP, \'Ecole Polytechnique, thibaut.mastrolia@polytechnique.edu. Part of this work was carried out while the author was doing his PhD at Universit\'e Paris Dauphine, whose support is kindly acknowledged.}\and Dylan {\sc Possama\"{i}} \footnote{Universit\'e Paris--Dauphine, PSL Research University, CNRS, CEREMADE, 75016 Paris, France, possamai@ceremade.dau-phine.fr.}
}
 
             \date{\today}

\begin{document}

 \maketitle

\begin{abstract}
In this paper, we extend the Holmstr\"om and Milgrom problem \cite{holmstrom1987aggregation} by adding uncertainty about the volatility of the output for both the Agent and the Principal. We study more precisely the impact of the "Nature" playing against the Agent and the Principal by choosing the worst possible volatility of the output. We solve the first--best and the second-best problems associated with this framework and we show that optimal contracts are in a class of contracts similar to \cite{cvitanic2014moral,cvitanic2015dynamic}, linear with respect to the output \textit{and} its quadratic variation. We compare our results with the classical problem in \cite{holmstrom1987aggregation}.

\vspace{5mm}

\vspace{5mm}


\vspace{0.5em}
\end{abstract}

\section{Introduction}
By and large, it has now become common knowledge among the economists, that almost everything in economics was to a certain degree a matter of incentives: incentives to work hard, to produce, to study, to invest, to consume reasonably... At the heart of the importance of incentives, lies the fact that, to quote B. Salani\'e \cite{salanie2005economics} "asymmetries of information are pervasive in economic relationships, that is to say, customers know more about their tastes than firms, firms know more about their costs than the government, and all Agents take actions that are at least partly unobservable''. Starting from the 70s, the theory of contracts evolved from this acknowledgment and the fact that such situations could not be reproduced using the general equilibrium theory. In the corresponding typical situation, a Principal (who takes the initiative of the contract) is (potentially) imperfectly informed about the actions of an Agent (who accepts or rejects the contract). The goal is to design a contract that maximises the utility of the Principal while that of the Agent is held to a given level. Of course, the form of the optimal contracts typically depends on whether these actions are observable/contractible or not, and on whether there are characteristics of the Agent that are unknown to the Principal. There are three main types of such problems: the first best case, or risk sharing, in which both parties have the same information; the second best case, or moral hazard, in which the action of the Agent is hidden or not contractible; the third best case or adverse selection, in which the type of the Agent is hidden. We will not study this last problem, and refer the interested reader to, among others, \cite{carlier2007optimal, cvitanic2013dynamics,cvitanic2007optimal,sannikov2007agency,sung2005optimal}. These problems are fundamentally linked to designing optimal incentives, and are therefore present in a very large number of situations. Beyond the obvious application to the optimal remuneration of an employee, one can for instance think on how regulators with imperfect information and limited policy instruments can motivate banks to operate entirely in the social interest, on how a company can optimally compensate its executives, on how banks achieve optimal securitisation of mortgage loans or on how investors should pay their portfolio managers (see Bolton and Dewatripont \cite{bolton2005contract} or Laffont and Martimort \cite{laffont2009theory} for many more examples). 

\vspace{0.5em}
Early studies of the risk-sharing problem can be found, among others, in Borch \cite{borch1992equilibrium}, Wilson \cite{wilson1968theory} or Ross \cite{ross1973economic}. Since then, a large literature has emerged, solving very general risk-sharing problems, for instance in a framework with several Agents and recursive utilities (see Duffie et al. \cite{duffie1994efficient} or Dumas et al. \cite{dumas2000efficient}, or for studying optimal compensation of portfolio managers (see Ou-Yang \cite{ou2003optimal} or Cadenillas et al. \cite{cadenillas2007optimal}). From the mathematical point of view, these problems can usually be tackled using either their dual formulation or the so--called stochastic maximum principle, which can characterize the optimal choices of the Principal and the Agent through coupled systems of Forward Backward Stochastic Differential Equations (FBSDEs in the sequel) (see the very nice monograph \cite{cvitanic2012contract} by Cvitani\'c and Zhang for a systematic presentation). One of the main findings in almost all of these works, is that one can find an optimal contract which is linear in the terminal value of the output managed by the Agent (a result already obtained in \cite{ross1973economic}) and possibly some benchmark to which his performance is compared. In specific cases, one can even have Markovian optimal contracts which are given as a (possibly linear) functional of the terminal value of the output (see in particular \cite{cadenillas2007optimal} for details).

\vspace{0.5em}
Concerning the so--called moral hazard problem, the first paper on continuous-time Principal--Agent problems is the seminal paper by Holmstr\"om and Milgrom \cite{holmstrom1987aggregation}. They consider a Principal and
an Agent with exponential utility functions and find that the optimal contract is
linear. Their work  was generalized by Sch\"attler and
Sung \cite{schattler1993first,schattler1997optimal}, Sung \cite{sung1995linearity, sung1997corporate},
M\"uller \cite{Muller1998276,muller2000asymptotic},
and Hellwig and Schmidt \cite{ECTA:ECTA375}, using a dynamic programming and martingales approach, which is classical in stochastic control theory (see also the survey paper by Sung \cite{sung2001lectures} for more references). The papers by Williams \cite{williams2009dynamic} and
Cvitani\'c, Wan and Zhang \cite{cvitanic2006optimal, cvitanic2009optimal} use the
stochastic maximum principle and FBSDEs to characterise the optimal
compensation for more general utility functions. More recently, Djehiche and Hegelsson \cite{djehiche2014principal,djehiche2015principal} have also used this approach. A more recent seminal paper in moral hazard setting is
Sannikov \cite{sannikov2008continuous}, who
finds a tractable model for solving the problem with a random time of
retiring the Agent and with continuous payments, rather than a lump-sum payment at the terminal time. Since then, a growing literature extending the above models has emerged, be it to include output processes with jumps \cite{biais2010large,capponi2015dynamic,pages2012bank,pages2014mathematical,zhang2009dynamic}, imperfect information and learning \cite{adrian2009disagreement,demarzo2011learning,giat2010investment,giat2013dynamic,he2010permanent,prat2014dynamic}, asset pricing \cite{ou2005equilibrium}, executive compensation \cite{he2009optimal}, or mortgage contracts \cite{piskorski2010optimal} (see also the illuminating survey paper \cite{sannikov2012contracts} for more references).

\vspace{0.5em}
Compared to the first--best problem, the moral hazard case corresponds to a Stackelberg--like game between the Principal and the Agent, in the sense that the Principal will start by trying to compute the reaction function of the Agent to a given contract (that is to say the optimal action chosen by the Agent given the contract), and use this action to maximise his utility over all admissible contracts\footnote{For a recent different approach, see Miller and Yang \cite{miller2015optimal}. For each possible Agent's control process, they characterise contracts that are incentive compatible for it.}. This approach does not always work, because
it may be hard to solve the Agent's stochastic control problem given an arbitrary payoff, possibly non-Markovian, and it may also be hard for the Principal
 to maximise over all such contracts. 
 Furthermore, the Agent's
 optimal control, if it even exists, depends on the given contract in a highly non--linear manner, rendering  the Principal's optimisation problem even harder and obscure.
For these reasons, and as mentioned above, in its most general form the problem was also approached in the literature by the stochastic version of the Pontryagin maximal principle. Nonetheless, none of these standard approaches can solve the problem when the Agent also controls the diffusion coefficient of the output, and not just the drift. Building upon this gap in the literature, Cvitani\'c, Possama\"{i} and Touzi \cite{cvitanic2014moral,cvitanic2015dynamic} have very recently developed a general approach of the problem through dynamic programming and so-called BSDEs and 2BSDEs, showing that under mild conditions, the problem of the Principal could always be rewritten in an equivalent way as a standard stochastic control problem involving two state variables, namely the output itself but also the continuation utility (or value function) of the Agent, a property which was pointed out by Sannikov in the specific setting of \cite{sannikov2008continuous}, and which was already well--known by the economists, even in discrete--time models, see for instance Spear and Srivastrava \cite{spear1987repeated}. An important finding of \cite{cvitanic2014moral}, in the context of a delegated portfolio management problem which generalizes Holmstr\"om and Milgrom problem \cite{holmstrom1987aggregation} to a context where the Agent can control the volatility of the (multidimensional) output process, is that in both the first--best and moral hazard problems, the optimal contracts become path--dependent, as they crucially use the quadratic variation of the output process (see also \cite{leung2014continuous} for a related problem).

\vspace{0.5em}
Our goal in this paper is to study yet another generalisation of the Holmstr\"om and Milgrom problem \cite{holmstrom1987aggregation}, to a setting where the Agent only controls the drift of the output, but where the twist is that both the Principal and the Agent may have some uncertainty about the volatility of the output, and only believe that it lies in some given interval of $\R_+$. This is the so--called situation of volatility ambiguity which has received a lot of attention recently, both in the mathematical finance community, since the seminal paper by Denis and Martini \cite{denis2006theoretical}, and in the economics literature, see for instance \cite{epstein2013ambiguous,epstein2014ambiguous}. From the mathematical point of view, everything happens as if both the Principal and the Agent have a "worst--case" approach to the contracting problem, in the sense that they act as if "Nature" was playing against them by choosing the worst possible volatility of the output. Mathematically, this means that the Principal and the Agent utility criterion incorporates the fact that they are playing a zero--sum game against "Nature". Furthermore, we put no restrictions on the beliefs that the Agent and the Principal have with respect to the likely volatility scenario, in the sense that their volatility intervals may or may not have non--empty intersections

\vspace{0.5em}
Although the impact of volatility ambiguity on optimal contracting has not been considered before in the literature\footnote{\label{foot:sng}After the completion of this paper, we have been made aware of a paper in preparation by Sung, where the author studies a problem similar to ours. Since a preprint version of this paper \cite{sung2015optimal} has appeared during the revision of the present manuscript, we will explain and detail the main differences between the two of them in Section \ref{sec:comp} below.}, the impact of drift ambiguity has been studied by Weinschenk \cite{weinschenk2010moral} for linear contracts in discrete--time, by Szydlowski \cite{szydlowski2012ambiguity} and Miao and Rivera \cite{miao2013robust} who consider an extension of Sannikov's model \cite{sannikov2008continuous} including ambiguity about the Agent's effort cost. Our paper also belongs to the literature on optimal contracting with learning, for which we can refer to the seminal papers of Williams \cite{williams2011persistent}, Prat and Jovanovic \cite{prat2014dynamic} and He, Wei and Yu \cite{he2014optimal}, or to Golosov, Troshkin and Tsyvinski \cite{golosov2016redistribution} and Farhi and Werning \cite{farhi2013insurance} for models addressing learning in the context of optimal dynamic taxation, Pavan, Segal and Toikka \cite{pavan2009dynamic} and Garrett and Pavan \cite{garrett2009dynamic} for models with transferrable utility, or DeMarzo and Sannikov \cite{demarzo2011learning} for a setting in which both the Principal and the Agent learn about future profitability from output. The learning feature of our model however concerns the risk--aversion of the Principal and the Agent themselves.

\vspace{0.5em}
Our fist task in this paper is to solve the risk--sharing problem. Surprisingly, this problem is much more involved than in the classical case, since it takes a very unusual form, as a supremum of a sum of two infimum over different sets. Nonetheless, we provide a generic method to solve it, which first focus on a sub--class of contracts similar to the ones obtained in \cite{cvitanic2014moral,cvitanic2015dynamic}, and then uses calculus of variations and convex analysis to argue that the optimal contracts in the sub--class should be optimal in the class of all admissible contracts. We use it successfully in what we coin a "non--learning" model, where both the Principal and the Agent do not update their beliefs with regards to volatility as time passes. Despite being restrictive, this benchmark model has the nice property that everything becomes completely explicit, and illustrates how our method can be applied in practice. We also highlight a surprising effect, which we interpret as a kind of arbitrage--like situation\footnote{We would like to thank one of the referees for suggesting this interpretation.}, corresponding to the situation where the volatility intervals of the Principal and the Agent are completely disjoint. In this case the problem degenerates and the Principal can actually reach utility $0$ using an appropriate sequence of contracts (we remind the reader that the exponential utility is $-\exp(-R_p x)$, so that it is bounded from above by $0$).

\vspace{0.5em}
Next, we concentrate on the second--best problem. Our first contribution is to use the theory of second--order BSDEs developed by Soner, Touzi and Zhang \cite{soner2012wellposedness}, and more precisely the recent wellposedness results obtained by Possama\"i, Tan and Zhou \cite{possamai2015stochastic}, to obtain a probabilistic representation of the value function of the Agent, for any sufficiently integrable contract. In particular, this representation gives an easy access to the optimal action chosen by the Agent. Then, following the ideas of \cite{cvitanic2014moral,cvitanic2015dynamic}, we concentrate our attention on a sub--class of contracts, for which the Principal problem can be solved using classical dynamic programming type arguments. The main problem is then to prove that the restriction is actually without loss of generality. We emphasise that in spite of the fact that  this approach is similar in spirit to the one used in \cite{cvitanic2015dynamic}, we cannot use their method of proof. Indeed, our problem is of a fundamentally different nature, because the Agent himself does not control the volatility of the output, but rather endures it. We therefore have to proceed completely differently and provide a general argument which shows with PDE technics, that the value function of the original and the sub--optimal problem actually solve the same PDE, which implies that they are equal by uniqueness of the solution to this PDE. We believe that this general approach can actually be applied to many situations and, constitutes one of our main contributions. In addition, we obtain an extremely general result stating that if the beliefs of the Principal and the Agent are completely different, then the problem always degenerates, and the Principal can reach utility $0$, making the second--best and first--best problems identical. Once more, this result highlights the necessity to get rid of these arbitrage-like situations.

\vspace{0.5em}
For simplicity and clarity, we also solve by a direct method the problem in the "non--learning" model mentioned above, where the identification of the two value functions can actually be obtained by simple (but tedious) algebra, constructing appropriate tight upper and lower bounds?

\vspace{0.5em}
The rest of the paper is organised as follows. We introduce the model and the contracting problem in Section \ref{Model sec}. Then Section \ref{sec:first} is devoted to the risk--sharing problem, while Section \ref{sec:2} treats the moral hazard case. We finally present some possible extensions in Section \ref{sec:comp}. The Appendix regroups some technical proofs.

\section{The model\label{Model sec}}

\subsection{The stochastic basis}
We start by giving all the necessary notations and definitions allowing us to consider the so--called "weak" formulation of the problem. 

\vspace{0.5em}
In this paper, we will denote by $\mathbb R^*_+$ the set of positive reals. Let $ \Omega \mathrel{\mathop:}= \lbrace \omega \in C\left( \left[ 0,T \right], \mathbb{R} \right),\; \omega_0 = 0 \rbrace $ be the canonical space equipped with the uniform norm $ ||\omega||_{\infty}^T \mathrel{\mathop:}= \sup_{0\le t \le T} |\omega_t| $. We then denote by $ B $ the canonical process, $ \mathbb{P}_0 $ the Wiener measure, $ \mathbb{F} \mathrel{\mathop:}= \lbrace \mathcal{F}_t \rbrace_{0\le t \le T} $ the filtration generated by $ B$ and $ \mathbb{F}^+ \mathrel{\mathop:}= \lbrace \mathcal{F}^+_t , 0 \le t \le T \rbrace $, the right limit of  $ \mathbb{F} $ where $ \mathcal{F}_t^+ \mathrel{\mathop:}= \cap_{ s> t} \mathcal{F}_s $. We will denote by $\mathbf{M}(\Omega)$ the set of all probability measures on $(\Omega,\Fc_T)$. We also recall the so--called universal filtration $\mathbb{F}^\star :=\lbrace \mathcal{F}^\star _t \rbrace_{0\le t \le T} $ defined as follows
$$
\mathcal{F}^\star _t
:=
\underset{\mathbb{P}\in\mathbf{M}(\Omega)}
{\bigcap}\mathcal{F}_t^{\mathbb{P}},
$$ 
where $\mathcal{F}_t^{\mathbb{P}}$ is the usual augmentation under $\mathbb{P}$. 
 
 \vspace{0.5em}
For any normed vector space $(E,\No{\cdot}_E)$ of a finite dimensional space and any filtration $\mathbb{X}$ on $(\Omega,\mathcal{F}_T)$, we denote by $\mathbb{H}^0(E,\mathbb{X})$ the set of all $\mathbb{X}-$progressively measurable processes with values in $E$. Moreover for all $p>0$ and for all $\mathbb{P}\in\mathbf{M}(\Omega)$, we denote by $\mathbb{H}^p(\mathbb{P},E,\mathbb{X})$ the subset of $\mathbb{H}^0(E,\mathbb{X})$ whose elements $H$ satisfy $\mathbb{E}^\mathbb{P}\left[\int_0^T\No{H_t}_E^p dt\right]<+\infty.$ The localised versions of these spaces are denoted by $\mathbb{H}^p_{\text{loc}}(\mathbb{P},E,\mathbb{X})$. 

\vspace{0.5em}
For any subset $\mathcal{P}\subset\mathbf{M}(\Omega)$, a $\mathcal{P}-$polar set is a $\mathbb{P}-$negligible set for all $\mathbb{P}\in \mathcal{P}$, and we say that a property holds $\mathcal{P}-$quasi--surely if it holds outside some $\mathcal{P}-$polar set. We also denote by $\mathbb H_{\mathcal P}^p(E,\mathbb X):= \bigcap_{\mathbb P\in \mathcal P} \mathbb H^p_{\text{loc}}(\mathbb P,E,\mathbb X)$. Finally, we introduce the following filtration $\mathbb{G}^{\mathcal{P}}:=\lbrace \mathcal{G}^\mathcal{P}_t \rbrace_{0\le t \le T} $ which will be useful in the sequel
$$\mathcal{G}^{\mathcal{P}}_t:=\mathcal{F}_{t}^\star \vee \mathcal{N}^\mathcal{P},\ t\leq T,$$
where $\mathcal{N}^\mathcal{P}$ is the collection of $\mathcal{P}-$polar sets, and its right--continuous limit, denoted $\G^{\Pc,+}$.

\vspace{0.3em}

Let us use the notation $\mathbb R_+^\star:=(0,+\infty)$. For all  $\alpha\in\mathbb{H}^1_{\text{loc}}(\mathbb{P}_0,\mathbb R^\star _+,\mathbb{F}) $, we define the following probability measure on $ (\Omega, \mathcal{F} )$
\begin{equation}\label{eq:recall}
\mathbb{P}^{\alpha}:= \mathbb{P}_0 \circ (X^{\alpha}_.)^{-1} \text{ where } X^{\alpha}_t := \int_0^t \alpha_s^{1/2} dB_s, \ t \in [0,T], \ \mathbb{P}_0-a.s.
\end{equation}
We denote by $ \mathcal{P}_S $ the collection of all such probability measures on $ (\Omega,\mathcal{F}_T)$. We recall from \cite{karandikar1995pathwise} that the quadratic variation process $ \left \langle B \right \rangle $ is universally defined under any $ \mathbb{P} \in \mathcal{P}_S $, and takes values in the set of all non--decreasing continuous functions from $ \mathbb{R}_+ $ to $ \R^\star _+$. We denote for any $p>0$
$$ \widehat{\mathbb H}_{\mathcal P}^p(E,\mathbb X):=   \left\{\gamma\in\mathbb H^0(E, \mathbb X),\; \sup_{\mathbb P\in \mathcal P} \mathbb E^\mathbb P\left[\int_0^T \No{\gamma_t}_E^p d\langle B\rangle_t \right] <+\infty\right\}.$$
 
We will denote the path--wise density of $\langle B\rangle$ with respect to the Lebesgue measure by $\widehat \alpha$. Finally we recall from \cite{soner2011quasi} that every $ \mathbb{P} \in \mathcal{P}_S $ satisfies the Blumenthal zero--one law and the martingale representation property. By definition, for any $\P\in\mathcal P_S$
$$W^\P_t:=\int_0^t\widehat\alpha_s^{-1/2}dB_s,\ \P-a.s.,$$
is a $(\mathbb P,\F)-$Brownian motion. Notice that the probability measures in $\P\in\mathcal P_S$ verify that the two following completed filtrations are equal
\begin{equation}\label{eq:filt}
\F^\P=(\F^{W^\P})^\P,
\end{equation}
where $\F^{W^\P}$ is the natural (raw) filtration of the process $W^\P$.

\vspace{0.5em}
The dependence of $W^{\P}$ on the underlying probability measure is mainly due to the fact that the construction of the stochastic integral is generically done only in an almost sure sense. For want of cosmetically nicer results, we would like to be able to find a universal aggregator of this family. Using the result of \cite{nutz2012pathwise}, and for instance assuming that we work under the usual ZFC framework, and in addition the continuum hypothesis\footnote{We insist on the fact that if one does not want to assume such an axiom, this is not a problem for this part of our work, and one just has to keep working with the family $(W^\P)_{\P\in\Pc_S}$. However, when defining the set of admissible contracts $\mathfrak C^{\rm SB}$ later in the paper, we will need it in order to define aggregated versions of stochastic integrals. If one does not want to use it then, it means that we have to restrict the control processes $Z$ in $\mathfrak C^{\rm SB}$ to ones having sufficiently regular trajectories to apply the pathwise integration theory of Karandikar \cite{karandikar1995pathwise} for instance. By standard density results, it should however not change the value function of the Principal.}, there actually exists an aggregated version of this family, which we denote by $W$, which is $\mathbb F^\star-$adapted and a $(\mathbb P,\F^\P)-$Brownian motion for every $\mathbb P\in\mathcal P_S$.

\vspace{0.5em}

Our focus in this paper will be on the following subset of $\mathcal{P}_S$.

\begin{Definition}
$\mathcal{P}_m$ is the sub--class of $ \mathcal{P}_S$ consisting of all $ \mathbb{P} \in \mathcal{P}_S $ such that the canonical process $B$ is a $(\mathbb{P},\F)-$uniformly integrable martingale.
\end{Definition}

The actions of the Agent will be considered as $\F-$predictable processes $a$ taking values in the compact set $[0,a_{\max}]$ (for every $\omega$). This upper bound corresponds to a maximal effort for the Agent, and we assume that it is known by the Principal. We believe that such an assumption is reasonable, since we assume here that the Principal knows the key characteristics of the Agent, and that the latter cannot exercise arbitrarily large effort\footnote{Obviously, an extension of the present framework to model incorporating adverse selection, that is to say that the Principal does not actually know perfectly all the characteristics of the Agent, is not only interesting mathematically, but also from the point of view of applications. However, we believe that this would lead to a much more difficult problem and leave it for future research.}. We denote this set by $\mathcal A$. Next, for any subset $\mathcal P\subset\mathcal P_S$ and any $a\in\mathcal A$, we define
$$\mathcal P^{a}:=\left\{\mathbb Q,\text{ s.t.}Ê\frac{d\mathbb Q}{d\mathbb P}
=\mathcal E\left(\int_0^T a_s\widehat{\alpha}_s^{-1/2}dW_s\right),\ \mathbb P-a.s., \text{ for some }\mathbb P\in\mathcal P
\right\}.$$

We also denote $\mathcal P^{\mathcal{A}}:=\cup_{a\in\mathcal A}\mathcal P^a$. In particular, for every $\P\in\mathcal P^{\mathcal A}$ there exists a unique pair $(\alpha^\P,a^\P)\in\mathbb{H}^1_{\text{loc}}(\mathbb{P}_0,\mathbb R^\star _+,\mathbb{F})\times\mathcal A$ such that
\begin{equation}\label{eq:dyn}
B_t=\int_0^ta^\P_sds+\int_0^t(\alpha_s^\P)^{1/2} dW^{a^\P}_s,\ \mathbb P-a.s.,
\end{equation}
where $dW^{a^\P}_s:=dW_s-(\alpha_s^\P)^{-1/2}a^\P_sds,\ \mathbb P-a.s.$ is a $(\P^a,\F^{\P^a})-$Brownian motion by Girsanov's theorem. More precisely, for any $\P\in\mathcal P^\mathcal A$, we must have $$\frac{d\P}{d\P^\alpha}=\mathcal E\left(\int_0^T a_s\widehat{\alpha}_s^{-1/2}dB_s\right),$$ for some $(\alpha,a)\in\mathbb{H}^1_{\text{loc}}(\mathbb{P}_0,\mathbb R^\star _+,\mathbb{F})\times\mathcal A$ and the following equalities hold
$$a^\P(B_\cdot)=a(B_\cdot)\ \text{and}\ \alpha^\P(B_\cdot)=\alpha(W_\cdot),\ dt\times\mathbb P-a.e.$$

For simplicity, we will therefore sometimes denote a probability measure $\P\in\mathcal P_S^{\mathcal A}$ by $\P^\alpha_a$. For any subset of $\Pc\subset\Pc_m$, we also denote for any $(t,\P)\in[0,T]\times\Pc$
$$\mathcal P(\P,t^+) :=\left\{\P^{'}\in\Pc, \ \P^{'}=\P,\ \text{on }\Fc_{t}^+\right\}.$$

We also recall that for every probability measure $\P$ on $\Omega$ and $\F-$stopping time $\tau$ taking value in $[0,T]$, there exists a family of regular conditional probability distribution (r.c.p.d. for short) $(\P^{\tau}_{\omega})_{\omega \in \Omega}$
	 (see e.g. Stroock and Varadhan \cite{stroock2007multidimensional}), satisfying

\begin{itemize}
	\item[(i)] For every $\omega \in \Omega$, $\mathbb{P}^{\tau}_{\omega}$ is a probability measure on $(\Omega,  \Fc_T)$.

	\vspace{0.5em}
	\item[(ii)] For every $ E \in \mathcal{F}_T $, the mapping $ \omega \longmapsto \mathbb{P}^{\tau}_{\omega}(E) $ is 
	$\mathcal{F}_\tau-$measurable.

	\vspace{0.5em}
	\item[(iii)] The family $ (\mathbb{P}^{\tau}_{\omega})_{\omega \in \Omega} $ is a version of the conditional probability measure of $ \mathbb{P} $ on $ \mathcal{F}_{\tau}$, {\it i.e.}, for every integrable $ \mathcal{F}_T-$measurable random variable $ \xi $ we have
	$ \mathbb{E}^{\mathbb{P} } [ \xi | \mathcal{F}_{\tau}](\omega)=\mathbb{E}^{ \mathbb{P}^{\tau}_{\omega}} \big[\xi \big],$ for $\P-a.e.$ $\omega\in\Omega$.

	\vspace{0.5em}
	\item[(iv)] For every $\omega \in \Omega$, $ \mathbb{P}^{\tau}_{\omega} (\Omega_{\tau}^{\omega})=1$, 
	where $\Omega_{\tau}^{\omega}\mathrel{\mathop:}= \big \lbrace \overline\omega \in \Omega, \ \overline\omega(s)=\omega(s), \ 0\le s \le \tau(\omega) \big \rbrace.$
\end{itemize}

	\vspace{0.5em}
	Furthermore, given some $\P$ and a family $(\Q_{\omega})_{\omega \in \Omega}$ such that 
	$\omega \longmapsto \Q_{\omega}$ is $\Fc_{\tau}-$measurable and $\Q_{\omega} (\Omega_{\tau}^{\omega}) = 1$ for all $\omega \in \Omega$, 
	one can then define a concatenated probability measure $\P \otimes_{\tau} \Q_{\cdot}$ by
	$$
		\P \otimes_{\tau} \Q_{\cdot} \big[ A \big]
		:=
		\int_{\Omega} \Q_{\omega} \big[A \big] ~\P(d \omega),
		\ \forall A \in \Fc_T.
	$$
We conclude this introductory section by noticing that since any $a\in \mathcal A$ impacts only the drift in the decomposition \eqref{eq:dyn} of $B$, we directly have that for any $\mathcal P_1,\mathcal P_2$ subsets of $\mathcal P_m$
\begin{equation}\label{equiv:intersection}\exists \, a \in \mathcal A,\, \mathcal P_1 \cap \mathcal P_2\neq\emptyset \Longleftrightarrow \forall \, a \in \mathcal A,\,  \mathcal P^a_1 \cap \mathcal P^a_2\neq\emptyset.\end{equation}

\subsection{The contracting problem in finite horizon}\label{section:contractpb}
\subsubsection{The ambiguity sets}
We consider a generalisation of the classical problem of Holmstr\"om and Milgrom \cite{holmstrom1987aggregation} and fix a given time horizon $T>0$. Here the Agent and the Principal both observe the outcome process $B$, but the Principal may not observe the action chosen by the Agent\footnote{He observes it in the risk--sharing problem of Section \ref{sec:first}, but not in the moral hazard case of Section \ref{sec:2}.}, and both of them have a "worst--case" approach to the contract, in the sense that they act as if "Nature" was playing against them by choosing the worst possible volatility of the output. More precisely, a contract will be a $\mathcal F_T-$measurable random variable, corresponding to the salary received by the Agent at time $T$ only. The Agent has then some beliefs about the volatility of the project, which are summed up in a family $(\mathcal P_A(t,\omega))_{(t,\omega)\in[0,T]\times\Omega}$, such that for any $(t,\omega)\in[0,T]\times\Omega$, $\Pc_A(t,\omega)\subset \Pc_m$. 

\vspace{0.5em}
The dependence in $(t,\omega)$ allows the beliefs of the Agent with regards to the volatility to change with both time and the observed randomness, that is to say with the evolution and history of the output process $B$. Similarly, we introduce a family $(\mathcal P_P(t,\omega))_{(t,\omega)\in[0,T]\times\Omega}$ associated to the Principal's beliefs. Notice that since $\omega_0=0$ for any $\omega\in\Omega$, these sets at $t=0$ do not depend on $\Omega$, so that we will use the simplified notations $\mathcal P_A:=\mathcal P_A(0,\omega)$ and $\mathcal P_P:=\mathcal P_P(0,\omega)$, for any $\omega\in\Omega$. We emphasise that these families cannot be chosen completely arbitrarily, and have to satisfy a certain number of stability and measurability properties, which are classical in stochastic control theory. Most of these are quite technical, so that we will refrain from commenting them, and refer instead the interested reader to \cite{possamai2015stochastic} for instance. We will hence assume throughout the paper the following
\begin{Assumption}\label{assump:prob}
For $\Psi=A,P$, we have
 
 \vspace{0.5em}
 \hspace{2em}$\mathrm{(i)}$ For every $(t,\omega) \in [0,T] \x \Omega$, one has $\Pc_\Psi(t, \omega) = \Pc_\Psi(t, \omega_{\cdot \wedge t})$ and
		$\P(\Omega_t^{\omega}) = 1$ whenever $\P \in \Pc_\Psi(t, \omega)$.
		The graph $[[ \Pc_\Psi ]]$ of $\Pc_\Psi$, defined by $[[ \Pc_\Psi ]] := \{ (t, \omega, \P): \P \in \Pc_\Psi(t,\omega) \}$, is upper semi--analytic in $[0,T] \times \Omega \times \mathbf{M}(\Omega)$.

\vspace{0.5em}	
 \hspace{2em}	$\mathrm{(ii)}$ $\Pc_\Psi$ is stable under conditioning, {\it i.e.} for every $(t,\omega) \in [0,T] \times \Omega$ and every $\P \in \Pc_\Psi(t,\omega)$ together with an $\F-$stopping time $\tau$ taking values in $[t, T]$,
		there is a family of r.c.p.d. $(\P_{\w})_{\w \in \Omega}$ such that $\P_{\w} \in \Pc_\Psi(\tau(\w), \w)$, for $\P-a.e.$ $\w \in \Omega$.

\vspace{0.5em}	
 \hspace{2em}	$\mathrm{(iii)}$ $\Pc_\Psi$ is stable under concatenation, {\it i.e.} for every $(t,\omega) \in [0,T] \x \Omega$ and $\P \in \Pc_\Psi(t,\omega)$ together with a $\F-$stopping time $\tau$ taking values in $[t, T]$,
		if $(\Q_{\w})_{\w \in \Omega}$ is a family of probability measures such that 
		$\Q_{\w} \in \Pc_\Psi(\tau(\w), \w)$ for all $\w \in \Omega$ and $\w \longmapsto \Q_{\w}$ is $\Fc_{\tau}-$measurable,
		then the concatenated probability measure $\P \otimes_{\tau} \Q_{\cdot} \in \Pc_\Psi(t,\omega)$.
\end{Assumption}
We will also need to consider the support of $\widehat \alpha$ induced by these families.
\begin{Definition}
For $\Psi=A,P$, we denote, for any $(t,\omega)\in[0,T]\times\Omega$, by $\mathcal D_\Psi(t,\omega)$ the smallest closed subset of $\mathbb R_+^\star$ such that 
$$\P\left(\left\{\w\in\Omega,\ \widehat\alpha_s(\w)\in\mathcal D_\Psi(t+s,\omega\otimes_t\w),\; \text{for $a.e.$ $s\in[0,T-t]$}\right\}\right)=1,$$
where the concatenated path $\omega\otimes_t\w\in\Omega$ is defined by
$$(\omega\otimes_t\w)(s):=\omega(s){\mathds 1}_{s\leq t}+(\w(s)-\omega(t)){\mathds 1}_{s\in(t,T]},\; s\in[0,T].$$
\end{Definition}
Let us conclude this section with our motivating examples of ambiguity sets.
\begin{Example}[Learning and non--learning models]\label{exemple}
{\rm The main example we have in mind here is the one corresponding to the so--called random $G-$expectations, introduced by Nutz in \cite{nutz2013random}. The idea is to specify directly the support of $\widehat\alpha$ and to consider, for $\Psi=A,P$, set--valued  processes $\mathbf{D}_\Psi:[0,T]\times\Omega\longmapsto 2^{\R_+^\star}$ which are progressively measurable in the sense of graph--measurability, that is to say that for all $t\in[0,T]$, we have
$$\left\{(s,\omega,A)\in[0,t]\times\Omega\times\R_+^\star,\ A\in \mathbf{D}_\Psi(s,\omega)\right\}\in\mathcal B([0,t])\otimes\Fc_t\otimes\mathcal B(\R_+^\star),$$
where $\mathcal B([0,t])$ and $\mathcal B(\R_+^\star)$ are the Borel $\sigma-$algebrae of $[0,t]$ and $\R_+^\star$. In this case, the sets $\Pc_\Psi(t,\omega)$ are defined for any $(t,\omega)\in[0,T]\times\Omega$ as being the collection of probability measures $\P\in\mathbf{M}(\Omega)$ such that
$$\widehat \alpha_s(\w)\in\mathbf{D}_\Psi(s+t,\omega\otimes_t\w),\; \text{for $ds\otimes d\P-a.e.$ $(s,\w)\in[0,T-t]\times\Omega$}.$$
It has been shown by Nutz and van Handel in \cite{nutz2013constructing} that the sets $\Pc_\Psi(t,\omega)$ indeed satisfy Assumption \ref{assump:prob}. 

\vspace{0.5em}
We can for instance assume that there exist processes $(\underline \alpha^P,\underline\alpha^A,\overline\alpha^P,\overline\alpha^A)\in\left(\mathbb H^0(\mathbb R_+^*,\mathbb F )\right)^4$ such that for any $(t,\omega)\in[0,T]\times\Omega$
$$\mathbf{D}_A(t,\omega)=[\underline\alpha^A_t(\omega),\overline\alpha^A_t(\omega)],\; \mathbf{D}_P(t,\omega)=[\underline\alpha^P_t(\omega),\overline\alpha^P_t(\omega)].$$
This typically leads to a model in which the Principal and the Agent estimate that the volatility of the output will live in intervals, whose bounds may vary with respect to the path of the output process. An interpretation of this specification is that both the Principal and the Agent update their beliefs by observing the past realisations of the output process $B$, and therefore that there is some kind of \textit{learning} effect. 

\vspace{0.5em}
In this paper, we will explain how to solve in general the moral hazard problems associated to such frameworks, by relating it to some Hamilton--Jacobi--Bellman--Isaacs partial differential equation. However, we will illustrate further our results in the only case which, to the best of our knowledge, leads to completely explicit computations, corresponding to assuming that $\underline \alpha^P,\underline\alpha^A,\overline\alpha^P,\overline\alpha^A$ are constant processes. In this case, the Agent and the Principal are not learning with time. Obviously, this is an unrealistic situation, and will lead in some cases to arbitrage--like results, that will be commented upon. Nonetheless, we believe that with appropriate caution, even this case leads to interesting qualitative results, and deserves to be treated thoroughly.}
\end{Example}

\subsubsection{Utilities of the Principal and the Agent}
We have now all the necessary tools to specify the utility obtained by the Agent, given a contract $\xi$, a recommended level of effort $a\in\mathcal A$ and an ambiguity set $(\Pc_A(t,\omega))_{(t,\omega)\in[0,T]\times\Omega}$
$$u^A_0(\xi,a):=\underset{\P\in\mathcal P_A^a}{\inf}\E^\P\left[\mathcal U_A\left(\xi-\int_0^Tk(a_s)ds\right)\right],$$
where $\mathcal U_A(x):=-\exp\left(-R_Ax\right)$ is the utility function of the Agent, for some $R_A>0$, and $k(x)$ is his cost function, which, as usual is assumed to be increasing, strictly convex and super--linear. 

\vspace{0.5em}
The value function of the Agent at time $0$ is therefore
$$U^A_0(\xi):=\underset{a\in\mathcal A}{\sup}\ \underset{\P\in\mathcal P_A^a}{\inf}\E^\P\left[\mathcal U_A\left(\xi-\int_0^Tk(a_s)ds\right)\right].$$

Similarly, the utility of the Principal, having an ambiguity set $(\Pc_P(t,\omega))_{(t,\omega)\in[0,T]\times\Omega}$, when offering a contract $\xi$ and a recommended level of effort $a\in\Ac$ is
\begin{equation}\label{eq:princ}
u^P_0(\xi,a):=\underset{\P\in\mathcal P_P^a}{\inf}\E^\P\left[\mathcal U_P\left(B_T-\xi\right)\right],
\end{equation}
where $\mathcal U_P(x):=-\exp\left(-R_Px\right)$ is the utility function of the Principal.

\vspace{0.5em}
Let $R<0$ denote the reservation utility of the Agent. The problem of the Principal is then to offer a contract $\xi$ as well as a recommended level of effort $a$ so as to maximize his utility \reff{eq:princ}, subject to the constraints
\begin{align}\label{const:rat}
&u^A_0(\xi,a)\geq R,\\
&u^A_0(\xi,a)=U^A_0(\xi).
\label{const:inc}
\end{align}

The first constraint is the so--called participation constraint, while the second one is the usual incentive compatibility condition, stating that the recommended level of effort $a$ should be the optimal response of the Agent, given the contract $\xi$.

\vspace{0.5em}
Furthermore, we will denote by $\mathcal C$ the set of admissible contracts, that is to say the set of $\mathcal F_T-$measurable random variables such that
\begin{equation}\label{eq:moexp}
\underset{\mathbb P\in\mathcal P_A^\mathcal A\cup\mathcal P_P^\mathcal A}{\sup}\mathbb E^\mathbb P\left[\exp\left(p\abs{\xi}\right)\right]<+\infty,\ \text{for any $p\geq 0$,}
\end{equation}
and we emphasise immediately that we will have to restrict a bit more the admissible contracts when solving the second--best problem, for technical reasons linked to integrability assumptions. However, we postpone the exact statement to Section \ref{sec:princ}, since it requires quite an important number of preliminaries.

%

\section{The first--best: a problem of calculus of variations}\label{sec:first}
In this section, we start by studying the first--best problem for the Principal, since it will serve as our main benchmark and has not been considered, as far as we know, in the pre--existing literature\footnote{See however Footnote \ref{foot:sng} above.}. Moreover, we will see that the derivation is a lot more complicated than in the classical setting. So much so that, quite surprisingly compared with the classical Holmstr\"om and Milgrom \cite{holmstrom1987aggregation} problem, the optimal contracts are in general not linear with respect to the final value of the output $B_T$, and are even path--dependent.

\vspace{0.5em}
Recall that for any contract $\xi\in\Cc$ and for any recommended effort level $a\in\Ac$
$$u_0^{P}(\xi,a)=\underset{\P\in\mathcal P^a_P}{\inf}\mathbb E^\P\left[\mathcal U_P\left(B_T-\xi\right)\right] .$$The value function of the Principal is then
\begin{equation}
\label{firstbest}U^{P,{\rm FB}}_0:=\underset{\xi\in\mathcal C}{\sup}\ \underset{a\in\mathcal A}{\sup}\left\{u_0^P(\xi,a)\right\},
\end{equation}
where the following participation constraint has to be satisfied
\begin{equation}\label{eq:constraint}\underset{\P\in\mathcal P^a_A}{\inf}\mathbb E^\P\left[\mathcal U_A\left(\xi-\int_0^Tk(a_s)ds\right)\right]\geq R.\end{equation}

The value function of the Principal defined by \eqref{firstbest} can be then rewritten 
\begin{equation}
\label{firstbest2}U^{P,{\rm FB}}_0:=\underset{\xi\in\mathcal C}{\sup}\ \underset{a\in\mathcal A}{\sup}\left\{\underset{\P\in\mathcal P^a_P}{\inf}\mathbb E^\P\left[\mathcal U_P\left(B_T-\xi\right)\right]+\rho\underset{\P\in\mathcal P^a_A}{\inf}\mathbb E^\P\left[\mathcal U_A\left(\xi-\int_0^Tk(a_s)ds\right)\right]\right\},
\end{equation}
where the Lagrange multiplier $\rho>0$ is here to ensure that the participation constraint \eqref{eq:constraint} holds.\vspace{0.5em}

\subsection{G\^ateaux differentiability and optimality}\label{section:methodeoptimalityFB}
Once again, the main difficulty is that the sets $\mathcal P_A$ and $\mathcal P_P$ are too abstract to solve generally the problem \eqref{firstbest2} directly, especially since we do not know if the two infima are attained or not. In order to overcome this major difficulty, we will restrict the set of admissible contracts to the ones for which both the Principal and the Agent have indeed a worst case measure. In order to do so, let us first introduce the following sets of worst probabilities for $\Psi=A,P$, any contract $\xi\in \mathcal C$, and any effort $a\in\mathcal A$
$$\mathcal P_\Psi^{\star,a}(\xi):=\left\{  \P^\star\in\mathcal P^a_\Psi,\, \underset{\P\in\mathcal P^a_\Psi}{\inf}\mathbb E^\P\big[\mathcal U_\Psi(B_T-\xi)\big]=\mathbb E^{\P^\star}\big[\mathcal U_\Psi(B_T-\xi)\big] \right\}.$$
We then define
$$\widetilde{\mathcal C}:=\left\{ \xi\in\mathcal C,\; \mathcal P_\Psi^{\star,a}(\xi)\neq \emptyset,\, \Psi=\{A,P \},\ \forall a\in\Ac \right\}.$$
Thus, the problem \eqref{firstbest2} restricted to contracts in $\widetilde{\mathcal C}$ becomes
\begin{equation}
\label{firstbest3}\widetilde U^{P,{\rm FB}}_0:=\underset{\xi\in\widetilde{\mathcal C}}{\sup}\ \underset{a\in\mathcal A}{\sup}\left\{\mathbb E^{\P^{\star,a,\xi}_P}\left[\mathcal U_P\left(B_T-\xi\right)\right]+\rho\mathbb E^{\P_A^{\star,a,\xi}}\left[\mathcal U_A\left(\xi-\int_0^Tk(a_s)ds\right)\right]\right\},
\end{equation}
where $\P_A^{\star,a,\xi}$ and $\P_P^{\star,a,\xi}$ are generic elements of $\Pc_A^{\star,a}(\xi)$ and $\Pc_P^{\star,a}(\xi)$ respectively.

\vspace{0.5em}
Next, it is not extremely convenient that the two expectations above are written under different probability measures. We will therefore use their definition to bring back all the computations under $\P_0$. 

\vspace{0.5em}
Let us start by considering the so--called Morse--Transue space on $(\Omega,\mathcal F_T,\P_0)$ (we refer the reader to the monographs \cite{rao1991theory,rao2002applications} for more details), defined by
$$M^\phi:=\left\{\xi:=\Omega\longrightarrow \R,\ \text{measurable},\ \mathbb E^{\mathbb P_0}\left[\phi(a\xi)\right]<+\infty,\ \text{for any $a\geq 0$}\right\},$$
where $\phi$ is the Young function
$$\phi(x):=\exp(\abs{x})-1.$$
Then, if $M^\phi$ is endowed with the norm
$$\No{\xi}_{\phi}:=\sup\left\{\mathbb E^{\mathbb P_0}[\xi g],\text{ with } \mathbb E^{\mathbb P_0}[\phi(g)]\leq 1\right\},$$
it becomes a (non--reflexive) Banach space.

\vspace{0.5em}
Then, for any $a\in \mathcal A$ and $(\alpha^P,\alpha^A)\in\mathbb{H}^1_{\text{loc}}(\mathbb{P}_0,\mathbb R^\star _+,\mathbb{F})\times \mathbb{H}^1_{\text{loc}}(\mathbb{P}_0,\mathbb R^\star _+,\mathbb{F})  $, we consider the map $\Xi^{\alpha^P,\alpha^A}_a:M^\phi\longrightarrow \mathbb R$ defined by
$$\Xi^{\alpha^P,\alpha^A}_a(\xi):= \mathbb E^{\P_0}\Bigg[e^{-R_P\left(\int_0^T a_s(X^{a,\alpha^P}_\cdot) ds +\int_0^T(\alpha^P_s)^\frac12 dB_s-\xi(X^{a,P}_\cdot)\right)}+\rho e^{-R_A\left(\xi(X_\cdot^{a,\alpha^A})-\int_0^Tk(a_s(X_\cdot^{a,\alpha^A}))ds\right)}\Bigg],$$
with
$$X^{a,\alpha^P}_\cdot(B_\cdot):= \int_0^\cdot a_s(B_\cdot) ds + \int_0^\cdot (\alpha^P_s(B_\cdot))^{\frac12} dB_s,\  X^{a,\alpha^A}_\cdot(B_\cdot):= \int_0^\cdot a_s(B_\cdot) ds +\int_0^\cdot (\alpha^A_s(B_\cdot))^{\frac12} dB_s.$$

Let now $(\xi,a)\in \widetilde{\mathcal C}\times\Ac$. For $\Psi=\left\{A,P \right\}$ and each $\mathbb P^{\star}_\Psi\in \mathcal P^{\star,a}_\Psi(\xi)$ we associate the corresponding $\alpha^{\Psi,\star}$ (recall \eqref{eq:recall}). We then have
$$\widetilde U^{P,{\rm FB}}_0=\underset{\xi\in\widetilde{\mathcal C}}{\sup}\ \underset{a\in\mathcal A}{\sup}\;\left\{ -\Xi_a^{\alpha^{P,\star},\alpha^{A,\star}}\right\}.$$

We will first interest ourselves to the maximisation with respect to $\xi$. It can be readily checked that $\Xi_a^{\alpha^{P,\star},\alpha^{A,\star}}$ is a strictly convex mapping in $\xi$, which is in addition proper and continuous. However, since $M^\phi$ is not reflexive, we cannot claim that its minimum is attained. Nonetheless, we can still use the characterisation of a minimiser in terms of G\^ateaux derivatives. Indeed, a random variable $\xi\in M^\phi$ which minimises $\Xi_a^{\alpha^P,\alpha^A}$ necessarily satisfies the following property
\begin{equation}\label{propertyopt} \widetilde{D}\Xi_a^{\alpha^{P,\star},\alpha^{A,\star}}(\xi)[h-\xi]\geq 0,\; \text{for any $h\in M^\phi$,}
\end{equation} where $\widetilde D \Xi_a^{\alpha
^{P,\star},\alpha^{A,\star}}$ denotes the G\^ateaux derivative of $\Xi_a^{\alpha^{P,\star},\alpha^{A,\star}}$ given by
\begin{align*}
 \widetilde{D}\Xi_a^{\alpha^{P,\star},\alpha^{A,\star}}(\xi)[h]&= \mathbb E^{\P_0}\Bigg[R_P h(X^{a,\alpha^{P,\star}}_\cdot) e^{-R_P\left(\int_0^T a_s(X^{a,\alpha^{P,\star}}_\cdot) ds +\int_0^T(\alpha^{P,\star}_s)^\frac12 dB_s-\xi(X^{a,\alpha^{P,\star}}_\cdot)\right)}\\
 &\hspace{3.2em}-R_A h(X^{a,\alpha^{A,\star}}_\cdot)\rho e^{-R_A\left(\xi(X_\cdot^{a,A})-\int_0^Tk(a_s(X_\cdot^{a,\alpha^{A,\star}}))ds\right)}\Bigg].
 \end{align*}
 Thus, if a contract $\xi^\star\in \widetilde{\mathcal C}$ satisfies the property \eqref{propertyopt}, it is optimal for the problem \eqref{firstbest3}.\vspace{0.5em}

We would like to insist on the fact that in this section, our main purpose is to propose a general method to investigate a risk sharing problem with uncertainty on the volatility. From our understanding of the problem, we believe that it would be extremely hard, for very general sets of ambiguity $\mathcal P_A$ and $\mathcal P_P$, to not start by studying the suboptimal problem \eqref{firstbest3}, since the latter can be put in the much more convenient form above. Nonetheless, we will show below that in the model that we coined "non--learning", the restriction is actually without loss of generality, which gives hope to be able to treat the completely general case in future works. This however goes beyond the scope of the current paper.

\vspace{0.5em}
That being said, let us describe more precisely the {\it modus operandi} that we propose for solving the risk--sharing problem.

\begin{Method}\label{methodFB}\text{ } The method is divided in two steps.

\begin{itemize}
\item[{\rm (1)}] We restrict the study to a particular set of contracts included in $\widetilde{\mathcal C} $, which we justify intuitively. If this problem can be solved, we can then check whether the corresponding optimal contracts satisfy the first--order optimality conditions in \reff{propertyopt}. 

 \item[{\rm (2)}] With the solution of the suboptimal problem \eqref{firstbest3} in hand, we can then try to show that it actually coincides with the value function of the problem \eqref{firstbest2}.\end{itemize}
\end{Method}
 
We need of course to say something about the choice of a pertinent subset of contracts in the first step described above. As explained in \cite{cvitanic2014moral,cvitanic2015dynamic}, when one deals with a problem in which the volatility of the output is controlled by the Agent, contracts which are linear (in the sense of integration) with respect to the output $B$ and its quadratic variation $\langle B\rangle$ play a fundamental role. We thus hope (and expect) to have optimal contracts in the following set 
$$\widetilde{\mathcal Q}:= \left\{\xi \in \widetilde{\mathcal C},\; \xi=\int_0^Tz_tdB_t+\frac12\int_0^T\gamma_t  d\langle B\rangle_t+\int_0^T \delta_t dt , \; (z,\gamma,\delta)\in  \mathbb H^2_{\mathcal P}(\R,\mathbb F)\times \widehat{ \mathbb H}^1_{\mathcal P}(\R,\mathbb F) \times\mathbb H^1(\R,\mathbb F) \right\}.$$
In this case, the contract $\xi$ appears as the terminal value of a controlled diffusion process, and we expect that the risk--sharing problem \reff{firstbest3} may be solved using technics from stochastic control theory. Such a general resolution is again beyond the scope of this paper, but we will illustrate our method by solving completely the simplest possible case (for which the proof is already far from being trivial).
 
\subsection{Application to the non-learning model}
In this section, we illustrate the previous explanations within the "non--learning" model introduced previously in Example \ref{exemple}. We will see then that we can actually simplify even more the set $\widetilde \Qc$ above and introduce the set
$$\mathcal Q:= \left\{\xi \in \mathcal C,\; \xi=zB_T+\frac{\gamma} 2 \langle B\rangle_T+\delta , \; (z,\gamma,\delta)\in  \mathbb R^3 \right\}.$$
Notice that the contracts are assumed to be in $\Cc$ and not in $\widetilde \Cc$. It will actually be one of our results that $\Qc\subset\widetilde \Cc$.

\vspace{0.5em}
From now on, noticing that any contract $\xi$ in $\mathcal Q$ is uniquely defined by the corresponding triplet of processes $(z,\gamma,\delta)$. For any triplet $(z,\gamma,\delta)$, we set $\xi^{z,\gamma,\delta}:=zB_T+\frac{\gamma} 2 \langle B\rangle_T+\delta$. We thus aim at solving the suboptimal problem
\begin{equation}
\label{firstbestsub}\underline{U}^{P,{\rm FB}}_0:=\underset{(z,\gamma,\delta)\in \mathcal P(\mathbb F)^3}{\sup}\ \underset{a\in\mathcal A}{\sup}\left\{\underset{\P\in\mathcal P^a_P}{\inf}\mathbb E^\P\left[\mathcal U_P\left(B_T-\xi^{z,\gamma,\delta}\right)\right]+\rho\underset{\P\in\mathcal P^a_A}{\inf}\mathbb E^\P\left[\mathcal U_A\left(\xi^{z,\gamma,\delta}-\int_0^Tk(a_s)ds\right)\right]\right\}.
\end{equation}

\noindent We insist on the fact that such a situation is different from the original Holmstr\"om--Milgrom \cite{holmstrom1987aggregation} problem, where the first--best contract was linear in $B_T$, and is thus much closer to its recent generalisation in \cite{cvitanic2014moral} where the Agent is allowed to control the volatility of the output, where optimal contracts are shown to be linear in $B_T$ and its quadratic variation $\langle B\rangle _T$. Nonetheless, in the setting of \cite{cvitanic2014moral}, moral hazard arises from the multi--dimensional nature of the output process, while it comes from the worst--case attitude of both the Principal and the Agent in our framework.

\subsubsection{Degeneracy for disjoints $\Pc_P$ and $\Pc_A$}
Our first result shows that if the sets of ambiguity of the Principal and the Agent are completely disjoint, then there are sequences of contracts in $\Qc$ such that the Principal can attain the universal upper bound $0$ of his utility, while ensuring that the Agent still receives his reservation utility $R_0$.

\begin{Theorem}
\label{thm:boundaryFB}
$(i)$ Assume that $\overline\alpha^P<\underline\alpha^A$. Then, considering the sequence of contracts $(\xi^n)_{n\in \mathbb N^\star }$ and the recommended effort $a^\textrm{max}$, with
$$ \xi^n:= \frac12 n \langle B\rangle_T-\frac T2n\underline\alpha^A+\delta^\star , \ \delta^\star :=Tk(a^{\textrm{max}})-\frac{\log(-R)}{R_A},$$
we have $\lim\limits_{n\to +\infty}u_0^{P,{\rm FB}}(\xi^n,a^{\rm{max}})=0$ and $u_0^{A}(\xi^n,a^{\max})=R$, for any $n\geq 1$.

\vspace{0.5em}
$(ii)$ Assume that $\underline\alpha^P>\overline\alpha^A$. 
Then, considering the sequence of contracts $(\xi^n)_{n\in \mathbb N^\star }$ and the recommended effort $a^\textrm{max}$, with
$$ \xi^n:= -\frac12 n \langle B\rangle_T+\frac T2n\overline\alpha^A+\delta^\star ,\ \delta^\star :=Tk(a^{\textrm{max}})-\frac{\log(-R)}{R_A}$$
we have $\lim\limits_{n\to +\infty}u_0^{P,{\rm FB}}(\xi^n,a^{\rm{max}})=0$ and $u_0^{A}(\xi^n,a^{\max})=R$, for any $n\geq 1$.
\end{Theorem}

Before proving this result, let us comment on it. We will see during the proof that when the sets of uncertainty for the Principal and the Agent are completely disjoint, the Principal can use the quadratic variation component in the contract in order to make appear in the exponential a term which he can make arbitrarily large, but which is not seen at all by the Agent in his utility, as it is constructed so that it disappears under the worst--case probability measure the Agent. This is therefore the combination of this difference between the worst--case measures of the Principal and the Agent, as well as the fact that their uncertainty sets are disjoints which make the problem degenerate. This is, from a mathematical point of view, quite a surprising result, which is however not so surprising from the economics point of view. Indeed, first of all in this case the Agent and the Principal do not somehow live in the same world, since they have totally different beliefs. Moreover, none of them learns from the observation of the realised volatility and updates his beliefs, which makes this specific case rather crude. However, we still believe that it can be of interest as a toy model, as long as one if careful with the conclusions that are derived from it. We will prove later that this phenomenon also always happens in the second--best case.

\begin{proof} 
\textbf{$(i)$ First case: $\mathbf{\underline\alpha^A>\overline\alpha^P}$.} We aim at showing that the sequence of contracts $(\xi^n)$ is a maximising sequence of contracts, allowing the Principal to reach utility $0$, when recommending in addition the level of effort $a^{\rm{max}}$. We have
\begin{align*}
u_0^P(\xi^n,a^{\rm{max}})&=\underset{\P\in\mathcal P^{a^{\max}}_P}{\inf}\mathbb E^\P\left[\mathcal U_P\left(B_T-\xi^n\right)\right]\\
&=\underset{\P\in\mathcal P^{a^{\max}}_P}{\inf}\mathbb E^\P\left[-e^{-R_P\left(B_T-\frac12 n \int^T_0 \alpha_s^\mathbb P  ds+\frac T2 n \underline\alpha^A-\delta^\star \right)}\right]\\
&= e^{-R_P (\frac T2 n \underline\alpha^A-\delta^\star )} \underset{\P\in\mathcal P^{a^{\max}}_P}{\inf}\mathbb E^\P\left[-e^{-R_P\left( \int_0^T (\alpha_s^\mathbb P)^{1/2} dW_s^{a^{\rm{max}}}+Ta^{\rm{max}}-\frac12 n\int_0^T \alpha_s ds\right)}\right]\\
&=-e^{-R_P\left( Ta^{\rm{max}} +\frac T2 n\underline\alpha^A-\delta^\star \right)}\underset{\P\in\mathcal P^{a^{\max}}_P}{\sup}\mathbb E^\P\left[\mathcal{E}\left(-R_P \int_0^T (\alpha_s^\mathbb P)^{1/2} dW_s^{a^{\rm{max}}}\right)\right.\\
&\left.\hspace{15.4em}\times\exp\left(\frac12 R_P^2\int_0^T \alpha_s^\mathbb P ds+\frac{R_P}{2}n\int_0^T \alpha_s^\mathbb P  ds\right)\right]\\
&= -\exp\left(-R_P \left( a^{\rm{max}}T-\delta^\star +\frac T2 n(\underline\alpha^A-\overline\alpha^P)-\frac12 R_PT\overline{\alpha}^P\right) \right),
\end{align*}
where we have used the fact that for any $\P\in\mathcal P^{a^{\max}}_P$, we have
$$\exp\left(\frac12 R_P^2\int_0^T \alpha_s^\mathbb P ds+\frac{R_P}{2}n\int_0^T \alpha_s^\mathbb P  ds\right)\leq \exp\left(\frac T2 R_P^2\overline \alpha^P+\frac{R_P}{2}nT \overline\alpha^P\right),\ \P-a.s.,$$
and that the stochastic exponential appearing above is clearly a $\P-$martingale for any $\P\in\mathcal P^{a^{\max}}_P$, so that the value of the supremum is clear and attained for the measure $\P^{\overline \alpha^P}_{a^{\max}}$.

\vspace{0.5em}
Hence, we obtain 
$u_0^P(\xi^n,a^{\rm{max}})\longrightarrow 0$ when $n\to +\infty$. Since $U_0^{P,{\rm FB}}\leq 0$, we deduce that the sequence $(\xi^n)$ approaches the best utility for the Principal when $n$ goes to $+\infty$. It remains to prove that for any $n\in \mathbb N^\star $, $\xi^n$ is admissible, \textit{i.e.}, $\xi^n$ satisfies
$$ \underset{\P\in\mathcal P^{a^{\max}}_A}{\inf}\mathbb E^\P\left[\mathcal U_A\left(\xi^n-K_T^{a^{\rm{max}})}\right)\right]\geq R, \; n\in \mathbb N^\star .$$
Indeed,
\begin{align*}
&\underset{\P\in\mathcal P^{a^{\max}}_A}{\inf}\mathbb E^\P\left[\mathcal U_A\left(\xi^n-Tk(a^{\rm{max}})\right)\right]\\
&=\underset{\P\in\mathcal P^{a^{\max}}_A}{\inf}\mathbb E^\P\left[- \exp\left(-R_A\left(\xi^n-Tk(a^{\rm{max}} \right) \right)\right]\\
&=\underset{\P\in\mathcal P^{a^{\max}}_A}{\inf}\mathbb E^\P\left[- \exp\left(-R_A\left(\frac12n \int_0^T \alpha_s^\mathbb P  ds-\frac T2n \underline \alpha^A+\delta^\star  -Tk(a^{\rm{max}}) \right) \right)\right]\\
&=-\exp\left(-R_A\left(\delta^\star -Tk(a^{\rm{max}})-\frac T2 n \underline\alpha^A \right) \right)\underset{\P\in\mathcal P^{a^{\max}}_A}{\sup}\mathbb E^\P\left[ \exp\left(-\frac {R_An}{2}\int_0^T \alpha_s^\mathbb P  ds\right)\right]\\
&=-e^{-R_A(\delta^\star -Tk(a^{\rm{max}}))}=R.
\end{align*}

\paragraph{$(ii)$ Second case: $\mathbf{\overline\alpha^A<\underline\alpha^P}$.} The proof is similar so that we omit it.

\end{proof}

\subsubsection{Optimal contracts with intersecting uncertainty sets}

We now study the non--degenerating case by applying Method \ref{methodFB}. We introduce the following maps, defined for any $(a,z,\gamma,\delta,\alpha_P,\alpha_A)\in \Ac\times\mathbb R^3\times [\underline\alpha^P,\overline\alpha^P]\times  [\underline\alpha^A,\overline\alpha^A],$ will play an important role in what follows
\begin{equation}\label{defF}
F(a,z,\gamma,\delta, \alpha_P, \alpha_A) :=\Gamma_P(a,z,\gamma,\delta,\alpha_P)+\rho \Gamma_A(a,z,\gamma,\delta,\alpha_A),
\end{equation}
where
\begin{align*}
\Gamma_P(a,z,\gamma,\delta,\alpha_P)&:=-\exp\left(R_P\left( \delta-(1-z)\int_0^T a_s ds +\left( \frac {R_P (1-z)^2}{2}+\frac{\gamma}{2}\right)\alpha_P T\right)\right),\\
\Gamma_A(a,z,\gamma,\delta,\alpha_A)&:=-\exp\left(R_A\left( \int_0^T k(a_s)ds-z\int_0^T a_s ds-\delta+\left( \frac{R_A z^2}{2}-\frac{\gamma}{2}\right)\alpha_A T\right)\right).
\end{align*}

 Let us define the subset $\Ac_{\rm det}\subset\Ac$ of actions which are deterministic. We define for any $(a,z,\gamma,\alpha_P,\alpha_A)\in \Ac\times\mathbb R^2\times [\underline\alpha^P,\overline\alpha^P]\times  [\underline\alpha^A,\overline\alpha^A],$ 
\begin{align*}
\nonumber G(a,z,\gamma,\alpha_P,\alpha_A):=& -\rho^{\frac{R_P}{R_A+R_P}}\frac{R_A+R_P}{R_P}\left(\frac{R_A}{R_P}\right)^{-\frac{R_A}{R_A+R_P}}e^{\frac{R_AR_P}{R_A+R_P}\int_0^T (k(a_s)-a_s) ds}\\
\nonumber &\times e^{\frac{R_AR_P}{R_A+R_P}\left(\frac{\gamma}{2} T(\alpha_P-\alpha_A)+\frac T2\left(\alpha_P R_P(1-z)^2+\alpha_A R_A z^2\right)\right)}.
\end{align*}
When $\alpha_P=\alpha_A$, by noticing that $G(a,z,\gamma,\alpha_P,\alpha_A)$ does not depend on $\gamma$ we will simply write, without any ambiguity, $G(a,z,\alpha_P):= G(a,z,\gamma,\alpha_P,\alpha_P).$ To alleviate later computations, we partition the set $\Qc$ into
\begin{align*}
 \mathcal Q^{\underline\gamma}&:=\left\{\xi\equiv(z,\gamma,\delta)\in \mathcal Q, \; \gamma<-R_P(1-z)^2 \right\},\\
  \mathcal Q^{|\gamma|}&:= \left\{\xi\equiv(z,\gamma,\delta)\in \mathcal Q, \; -R_P(1-z)^2< \gamma< R_A z^2 \right\},\\
 \mathcal Q^{d}&:= \left\{\xi\equiv(z,\gamma,\delta)\in \mathcal Q, \; -R_P(1-z)^2= \gamma \right\},\\
  \mathcal Q^{u}&:= \left\{\xi\equiv(z,\gamma,\delta)\in \mathcal Q, \; \gamma=R_A z^2 \right\},\\
   \mathcal Q^{\overline\gamma}&:= \left\{\xi\equiv(z,\gamma,\delta)\in \mathcal Q, \; \gamma>R_A z^2 \right\},
   \end{align*}
   and define for any $(a,\xi)\in\Ac\times\Cc$
   $$\widetilde u_0^{P,{\rm FB}}(a,\xi):=\underset{\P\in\mathcal P^a_P}{\inf}\mathbb E^\P\left[\mathcal U_P\left(B_T-\xi\right)\right]+\rho\underset{\P\in\mathcal P^a_A}{\inf}\mathbb E^\P\left[\mathcal U_A\left(\xi-\int_0^Tk(a_s)ds\right)\right].$$
The following lemma computes the Principal's utility $\widetilde u_0^{P,{\rm FB}}(a,\xi)$ for a recommended level of effort in $\Ac_{\rm det}$ and any contract $\xi\in\Qc$. Its proof is relegated to the Appendix.
\begin{Lemma} \label{lemmaF} 
We have $\mathcal Q\subset \widetilde{\mathcal C}$. Besides, fix some $a\in\mathcal A_{\text{det}}$ and some $\xi\in\mathcal Q$, with $\xi\equiv (z,\gamma,\delta)$. 

\vspace{0.5em}
$(i)$ If $\xi \in \mathcal Q^{\underline\gamma}$, then 
$\widetilde u_0^{P,{\rm FB}}(a,\xi)=F(a,z,\gamma,\delta, \underline\alpha^P,\overline\alpha^A).$

\vspace{0.5em}
\noindent $(ii)$ {\rm a)} If $\xi\in \mathcal Q^d$, then for any $\mathbb P\in \mathcal P_P^a$, $\mathbb E^\P\left[\mathcal U_P\left(B_T-\xi\right)\right]=\underset{\P\in\mathcal P^a_P}{\inf}\mathbb E^\P\left[\mathcal U_P\left(B_T-\xi\right)\right],$
and in particular
\begin{align*}
\widetilde u_0^{P,{\rm FB}}(a,\xi)=F(a,z,\gamma,\delta, \alpha_P,\overline\alpha^A), \ \text{for any $\alpha_P \in [\underline \alpha^P, \overline\alpha^P]$.}
\end{align*}

\hspace{1.3em} {\rm b)} If $\xi\in \mathcal Q^{|\gamma|}$, then $
\widetilde u_0^{P,{\rm FB}}(a,\xi)=F(a,z,\gamma,\delta, \overline\alpha^P,\overline\alpha^A).$

\vspace{0.5em}
\hspace{1.3em} {\rm c)}  If $\xi\in \mathcal Q^{u}$, then for any $\mathbb P\in \mathcal P_A^a$, 
$$\mathbb E^\P\left[\mathcal U_A\left(\xi-\int_0^Tk(a_s)ds\right)\right]=\underset{\P\in\mathcal P^a_A}{\inf}\mathbb E^\P\left[\mathcal U_A\left(\xi-\int_0^Tk(a_s)ds\right)\right],$$ 
and in particular $
\widetilde u_0^{P,FB}(a,\xi)=F(a,z,\gamma,\delta, \overline\alpha^P,\alpha_A),\ \text{for any $\alpha_A \in [\underline \alpha^A, \overline\alpha^A]$}.$

\vspace{0.5em}
$(iii)$ If $\xi\in \mathcal Q^{\overline\gamma}$, then $\widetilde u_0^{P,FB}(a,\xi)=F(a,z,\gamma,\delta, \overline\alpha^P,\underline\alpha^A).$
\end{Lemma}
The next lemma computes the supremum of $F$ with respect to $\delta$. Its proof is also relegated to the Appendix.
\begin{Lemma}\label{lemmadelta*}
For any $(a,z,\gamma,\alpha_P,\alpha_A)\in \Ac\times\mathbb R\times \mathbb R\times [\underline\alpha^P,\overline\alpha^P]\times[\underline\alpha^A,\overline\alpha^A]$ we have
\begin{align*}\underset{\delta\in\R}{\sup} \ F(a,z,\gamma,\delta,\alpha_P,\alpha_A)&= F(a,z,\gamma,\delta^\star (z,\gamma,\alpha_P,\alpha_A),\alpha_P,\alpha_A)= G(a,z,\gamma,\alpha_P,\alpha_A),
\end{align*}
where 
\begin{align*}\nonumber \delta^\star (z,\gamma,\alpha_P,\alpha_A)&:= \frac{1}{R_A+R_P}\left[ \log\left(\frac{\rho R_A}{R_P}\right) +\int_0^T\left( (R_P(1-z)-R_Az)a_s +R_Ak(a_s)\right) ds\right.\\
\label{delta*}&\hspace{2.5cm}\left. -\frac{R_P}{2}(R_P(1-z)^2+\gamma)\alpha_P T +\frac{R_A}{2}\left(R_A z^2-\gamma\right)\alpha_A T\right].
\end{align*}

\end{Lemma}
The next lemma gives the optimal contracts and efforts in $\Ac_{\rm det}$ and each subset of our partition of $\Qc$ for the Principal problem, for a fixed Lagrange multiplier $\rho$.
\begin{Lemma}\label{lemma:partitionQ} Let $a^\star $ be the minimiser of the strictly convex map $a\longmapsto k(a)-a$ and define $z^\star :=\frac{R_P}{R_A+R_P}$.

\vspace{0.5em}
$(i)$ Optimal contracts in $\mathcal Q^{\underline\gamma}$.

\vspace{0.5em}
\hspace{0.9em}$a)$  If $\underline\alpha^P<\overline\alpha^A$, then $\underset{a \in \mathcal A_{\text{det}}}{\sup}\ \underset{\xi \in \mathcal Q^{\underline\gamma}}{\sup}\ \widetilde u_0^{P,{\rm FB}}(a,\xi)=F(a^\star ,z^\star ,\gamma^\star ,\delta^\star ,\underline\alpha^P,\overline\alpha^A),$
where $\gamma^\star :=-R_P(1-z^\star )^2$ and
\begin{align*}
\delta^\star &:= \frac{1}{R_A+R_P}\left[ \log\left(\rho\frac{R_A}{R_P}\right) +R_AT k(a^\star ) ds +\frac{R_A^2R_PT}{2(R_A+R_P)}\overline\alpha^A\right].
\end{align*}

\hspace{0.9em}$b)$ If $\underline\alpha^P=\overline\alpha^A=:\tilde\alpha$, then $\underset{a \in \mathcal A_{\text{det}}}{\sup}\ \underset{\xi \in \mathcal Q^{\underline\gamma}}{\sup}\ \widetilde u_0^{P,{\rm FB}}(a,\xi)=F(a^\star ,z^\star ,\gamma^\star ,\delta^\star ,\tilde\alpha,\tilde\alpha),$
 for any $\gamma^\star <-R_P(1-z^\star )^2$ and
\begin{align*}
\delta^\star &:=  \frac{1}{R_A+R_P}\left[ \log\left(\rho\frac{R_A}{R_P}\right) +R_AT k(a^\star )\right] -\frac{\gamma^\star }{2}\tilde\alpha T.
\end{align*}
In these two cases $\underset{a \in \mathcal A_{\text{det}}}{\sup}\ \underset{\xi \in \mathcal Q^{\underline\gamma}}{\sup}\ \widetilde u_0^{P,{\rm FB}}(a,\xi)=G(a^\star ,z^\star ,\overline\alpha^A).$

\vspace{0.5em}
$(ii)$ Optimal contracts in $\mathcal Q^d$. 

\vspace{0.5em}
For $\alpha_P\in [\underline\alpha^P, \overline\alpha^P]$, $\underset{a \in \mathcal A_{\text{det}}}{\sup}\; \underset{\xi \in \mathcal Q^{d}}{\sup} \widetilde u_0^{P,{\rm FB}}(a,\xi)=F(a^\star ,z^\star ,\gamma^\star ,\delta^\star ,\alpha_P,\overline\alpha^A)=G(a^\star ,z^\star ,\overline\alpha^A),$
with $\gamma^\star :=-R_P(1-z^\star )^2$ and
\begin{align*}
\delta^\star &:= \frac{1}{R_A+R_P}\left[ \log\left(\rho\frac{R_A}{R_P}\right) +R_AT k(a^\star ) +\frac{R_A^2R_PT}{2(R_A+R_P)}\overline\alpha^A\right].
\end{align*}
$(iii)$ Optimal contracts in $\mathcal Q^{|\gamma|}$.

\vspace{0.5em}
\hspace{0.9em}$a)$ If $\overline\alpha^P<\overline\alpha^A$, then $\underset{a \in \mathcal A_{\text{det}}}{\sup}\; \underset{\xi \in \mathcal Q^{|\gamma|}}{\sup} \widetilde u_0^{P,{\rm FB}}(a,\xi)=F(a^\star ,z^\star ,\gamma^\star ,\delta^\star , \overline\alpha^P,\overline\alpha^A)=G(a^\star ,z^\star ,\overline\alpha^P),$ where $\gamma^\star :=R_A |z^\star |^2$ and
\begin{align*}\delta^\star &:=  \frac{1}{R_A+R_P}\left[ \log\left(\rho\frac{R_A}{R_P}\right) +R_AT k(a^\star )\right] -\frac{\gamma^\star }{2}\overline\alpha^P T.
\end{align*}
\hspace{0.9em}$b)$ If $\overline\alpha^P=\overline\alpha^A=:\overline\alpha$, then $\underset{a \in \mathcal A_{\text{det}}}{\sup}\; \underset{\xi \in \mathcal Q^{|\gamma|}}{\sup} \widetilde u_0^{P,{\rm FB}}(a,\xi)=F(a^\star ,z^\star ,\gamma^\star ,\delta^\star , \overline\alpha,\overline\alpha)=G(a^\star ,z^\star ,\overline\alpha),$ for any $\gamma^\star \in (-R_P(1-z^\star )^2,R_A |z^\star |^2)$ and
$$\delta^\star := \frac{1}{R_A+R_P}\left[ \log\left(\rho\frac{R_A}{R_P}\right) +R_AT k(a^\star )\right] -\frac{\gamma^\star }{2}\overline\alpha T.$$

\hspace{0.9em} $c)$ If $\overline\alpha^P>\overline\alpha^A$, then $\underset{a \in \mathcal A_{\text{det}}}{\sup}\ \underset{\xi \in \mathcal Q^{|\gamma|}}{\sup}\ \widetilde u_0^{P,{\rm FB}}(a,\xi)=F(a^\star ,z^\star ,\gamma^\star ,\delta^\star , \overline\alpha^P,\overline\alpha^A)=G(a^\star ,z^\star ,\overline\alpha^A),$
where
$\gamma^\star :=-R_P (1-z^\star )^2$ and
\begin{align*}\delta^\star &:=  \frac{1}{R_A+R_P}\left[ \log\left(\rho\frac{R_A}{R_P}\right) +R_AT k(a^\star ) +\frac{R_A^2R_PT}{2(R_A+R_P)}\overline\alpha^A\right].
\end{align*}

$(iv)$ Optimal contracts in $\mathcal Q^u$. 

\vspace{0.5em}
For any $\alpha_A\in [\underline\alpha^A, \overline\alpha^A]$, $\underset{a \in \mathcal A_{\text{det}}}{\sup}\; \underset{\xi \in \mathcal Q^{u}}{\sup} \widetilde u_0^{P,{\rm FB}}(a,\xi)=F(a^\star ,z^\star ,\gamma^\star ,\delta^\star ,\overline\alpha^P,\alpha_A)=G(a^\star ,z^\star ,\overline\alpha^P),$ with $\gamma^\star :=R_A|z^\star |^2$ and
\begin{align*}
\delta^\star &:= \frac{1}{R_A+R_P}\left[ \log\left(\rho\frac{R_A}{R_P}\right) +R_AT k(a^\star ) -\frac{R_AR_P^2T}{2(R_A+R_P)}\overline\alpha^P\right].
\end{align*}

$(v)$ Optimal contracts in $\mathcal Q^{\overline\gamma}$. 

\vspace{0.5em}
\hspace{0.9em} $a)$  If $\overline\alpha^P=\underline\alpha^A=:\check\alpha$, then $\underset{a \in \mathcal A_{\text{det}}}{\sup}\ \underset{\xi \in \mathcal Q^{\overline\gamma}}{\sup}\ \widetilde u_0^{P,{\rm FB}}(a,\xi)=F(a^\star ,z^\star ,\gamma^\star ,\delta^\star , \check\alpha,\check\alpha),$
for any $\gamma^\star >R_A |z^\star |^2$ and
\begin{align*}
\delta^\star &:=  \frac{1}{R_A+R_P}\left[ \log\left(\rho\frac{R_A}{R_P}\right) +R_AT k(a^\star )\right] -\frac{\gamma^\star }{2}\check\alpha T.
\end{align*}

\hspace{0.9em} $b)$ If $\overline\alpha^P>\underline\alpha^A$, then $\underset{a \in \mathcal A_{\text{det}}}{\sup}\ \underset{\xi \in \mathcal Q^{\overline\gamma}}{\sup}\ \widetilde u_0^{P,{\rm FB}}(a,\xi)=F(a^\star ,z^\star ,\gamma^\star ,\delta^\star , \overline\alpha^P,\underline\alpha^A),$
with $\gamma^\star :=R_A |z^\star |^2$ and
\begin{align*}
\delta^\star &:= \frac{1}{R_A+R_P}\left[ \log\left(\rho\frac{R_A}{R_P}\right) +R_ATk(a^\star )  -\frac{R_AR_P^2T}{2(R_A+R_P)}\overline\alpha^P\right].
\end{align*}

In these two cases $\underset{a \in \mathcal A_{\text{det}}}{\sup}\ \underset{\xi \in \mathcal Q^{\underline\gamma}}{\sup}\ \widetilde u_0^{P,{\rm FB}}(a,\xi)=G(a^\star ,z^\star ,\overline\alpha^P).$
\end{Lemma}

We now conclude the first step of Method \ref{methodFB} proposed in Section \ref{section:methodeoptimalityFB} by proving that we can indeed restrict the study to contracts in $\mathcal Q$, and thus solve the problem \eqref{firstbest2}.
 The following lemma studies when \eqref{propertyopt} holds for contracts having the form of the optimal contracts in $\Qc$. Its proof is postponed to the Appendix.
\begin{Lemma}\label{lemma:dxi}
Fix some $a\in\mathcal A$ and let $\xi_a:=z^\star B_T+\frac{\gamma^\star }{2}\langle B\rangle_T +\delta^\star (a)$ where $\gamma^\star \in \mathbb R$, and
$$z^\star :=\frac{R_P}{R_A+R_P}, \; \delta^\star (a):=\frac{1}{R_A+R_P}\left(\log\left(\rho\frac{R_A}{R_P}\right)+R_A \int_0^Tk(a_s)ds \right)+\lambda,\; \lambda\in \mathbb R. $$
Then, 
\begin{itemize}
\item if $R_A=R_P$, Property \eqref{propertyopt} is satisfied for $\xi_a$ if $\alpha_P=\alpha_A$. 
\item if $R_A\neq R_P$, Property \eqref{propertyopt} is satisfied for $\xi_a$ if $\alpha_P=\alpha_A=:\alpha$ and the following condition holds
\begin{equation}\label{propopti} \frac{\gamma^\star }{2}\alpha T+\lambda= 0. \end{equation} 
\end{itemize}
\end{Lemma}
We can now give our main result stating that the optimal contract in the first best problem belongs to $\Qc$. The proof is mainly based on applying the second step of Method \ref{methodFB} and is postponed to the Appendix.
\begin{Theorem}\label{thm:FB:Q} We have

\vspace{0.5em}
$(i)$ Assume that $\underline\alpha^A\= \overline\alpha^P$. Then, the set
\begin{align*}
\nonumber\overline{\mathcal Q^{\overline\gamma}}&:=\Big\{ \xi^\star \equiv(z^\star ,\gamma^\star ,\delta^\star )\in \mathcal Q,\;  z^\star =\frac{R_P}{R_A+R_P},\; \gamma^\star \geq R_A|z^\star |^2, \\
  &\hspace{2.5em} \delta^\star = Tk(a^\star )-\frac{R_P}{R_A+R_P}Ta^\star +\frac{\overline\alpha^PT}{2}\left(\frac{R_AR_P^2}{(R_A+R_P)^2}-\gamma^\star \right)-\frac{1}{R_A}\log(-R)\Big\},
  \end{align*} is the subset of optimal contracts in $\mathcal Q$ for the first best problem \eqref{firstbest}
with the optimal recommended level effort 
$$a^\star := \text{argmax}\left(k(a)-a\right).$$

$(ii)$ Assume that $\underline\alpha^A< \overline\alpha^P<\overline\alpha^A$. Then, an optimal contract is given by $$\xi^\star :=z^\star B_T+\frac{\gamma^\star }{2}\langle B\rangle_T +\delta^\star ,$$
where $\gamma^\star =R_A(z^\star )^2$, and $z^\star :=\frac{R_P}{R_A+R_P}, \; \delta^\star :=Tk(a^\star )-\frac{R_P}{R_A+R_P}Ta^\star -\frac{1}{R_A}\log(-R).$

\vspace{0.5em}
$(iii)$ Assume that $\overline\alpha^A\= \overline\alpha^P$. Then, the set
\begin{align*}
\nonumber\overline{\mathcal Q^{|\gamma|}}&:=\Big\{\xi^\star \equiv(z^\star ,\gamma^\star ,\delta^\star )\in \mathcal Q,\; z^\star =\frac{R_P}{R_A+R_P},\; \gamma^\star \in[-R_P(1-z^\star )^2,R_A|z^\star |^2], \\
  &\hspace{2.5em} \delta^\star = Tk(a^\star )-\frac{R_P}{R_A+R_P}Ta^\star +\frac{\overline\alpha^PT}{2}\left(\frac{R_AR_P^2}{(R_A+R_P)^2}-\gamma^\star \right)-\frac{1}{R_A}\log(-R)\Big\},
\end{align*} is the subset of optimal contracts in $\mathcal Q$ for the first best problem \eqref{firstbest}
with the optimal recommended level effort 
$$a^\star := \text{argmax}\left(k(a)-a\right).$$

$(iv)$ Assume that $\underline\alpha^P=\overline\alpha^A$. 
Then, the set
\begin{align*}
\nonumber\overline{\mathcal Q^{\underline\gamma}}&:=\Big\{\xi^\star \equiv(z^\star ,\gamma^\star ,\delta^\star )\in \mathcal Q,\; z^\star =\frac{R_P}{R_A+R_P},\; \gamma^\star \leq -R_P(1-z^\star )^2, \\
  &\hspace{2.5em} \delta^\star = Tk(a^\star )-\frac{R_P}{R_A+R_P}Ta^\star +\frac{\underline\alpha^PT}{2}\left(\frac{R_AR_P^2}{(R_A+R_P)^2}-\gamma^\star \right)-\frac{1}{R_A}\log(-R)\Big\},
\end{align*} is the subset of optimal contracts in $\mathcal Q$ for the first best problem \eqref{firstbest}
with the optimal recommended level effort 
$$a^\star := \text{argmax}\left(k(a)-a\right).$$

$(v)$ Assume that $\underline\alpha^P<\overline\alpha^A<\overline\alpha^P$. Then, an optimal contract is given by $$\xi^\star :=z^\star B_T+\frac{\gamma^\star }{2}\langle B\rangle_T +\delta^\star ,$$
where $\gamma^\star =-R_P|1-z^\star |^2$, and
$$z^\star :=\frac{R_P}{R_A+R_P}, \; \delta^\star :=Tk(a^\star )-\frac{R_P}{R_A+R_P}Ta^\star +\frac{\overline\alpha^AT}{2}\frac{R_AR_P}{R_A+R_P}-\frac{1}{R_A}\log(-R).$$

\end{Theorem}

\subsubsection{Comments and comparison with the case without ambiguity}
Using Theorem \ref{thm:FB:Q}, we recover the classical result that when $\underline\alpha^P=\overline\alpha^P=\underline\alpha_A=\overline\alpha^A=:\alpha$ (that is to say when there is no ambiguity), the optimal first--best contract is given by
\begin{equation}\label{eq:firstw}
z^\star B_T+\frac{R_AR_P^2\alpha}{2(R_A+R_P)^2}T+Tk(a^\star )-\frac{R_P}{R_A+R_P}Ta^\star -\frac{1}{R_A}\log(-R),
\end{equation}
which provides the Principal with utility
\begin{equation}\label{eq:utprinc}
-(-R_0)^{-\frac{R_P}{R_A}}\exp\left(R_PT\left(k(a^*)-a^*+\frac{\alpha}{2}\frac{R_AR_P}{R_A+R_P}\right)\right).
\end{equation}

Therefore, as mentioned above, the first main difference with the ambiguity case is that in our framework, one has in general to rely on path--dependent contracts using the quadratic variation of the output. There is nonetheless an exception. Indeed, in the case where $\overline\alpha^A=\overline\alpha^P$, the choice $\gamma^\star=0$ is allowed, so that there is a linear optimal contract in this case (which coincides with \eqref{eq:firstw} above), and in this case only. Furthermore, in the three cases $\overline\alpha^A=\overline\alpha^P$, $\overline\alpha^A=\underline\alpha^P$, $\overline\alpha^P=\underline\alpha^A$, we have identified uncountably many optimal contracts in the class $\Qc$. This is really different from the case without ambiguity, where the optimal contract is essentially unique.

\vspace{0.5em}
Finally, let us compare the utility of the Principal can get out of the problem (since the Agent always receives his reservation utility, there is nothing to compare for him). Again by Theorem \ref{thm:FB:Q}, whenever we have $\underline\alpha^A\leq \overline\alpha^P\leq \overline\alpha^A$, the Principal receives
$$-(-R)^{-\frac{R_P}{R_A}}\exp\left(R_PT\left(k(a^*)-a^*+\frac{\overline\alpha^P}{2}\frac{R_AR_P}{R_A+R_P}\right)\right),$$
which is always less than \eqref{eq:utprinc}, for any $\alpha\in[\underline\alpha^P,\overline\alpha^P]$, which means that, as intuition would dictate, the Principal is worse off compared to the case where he would not have any aversion to ambiguity.

\vspace{0.5em}
Then, when we have $\underline\alpha^P\leq\overline\alpha^A\leq\overline\alpha^P$, the Principal gets
$$-(-R)^{-\frac{R_P}{R_A}}\exp\left(R_PT\left(k(a^*)-a^*+\frac{\overline\alpha^A}{2}\frac{R_AR_P}{R_A+R_P}\right)\right),$$
which is actually larger than \eqref{eq:utprinc} if $\alpha\geq\overline\alpha^A$. In other words, compared to a situation where the Principal would have no ambiguity, but were more pessimistic than the Agent and believed in a level of volatility higher than $\overline\alpha^A$, the ambiguity averse Principal actually obtains a larger utility. 

\vspace{0.5em}
The situation is the same, though even more extreme, when $\overline\alpha^P<\underline\alpha^A$ or $\overline\alpha^A<\underline\alpha^P$, since the Principal can reach utility $0$ and is therefore always better off compared to the case without ambiguity. We nonetheless insist once more on the fact that these results are obviously in part due to the "non--learning" assumption we made on both the Principal and the Agent, and a deeper understanding of the problem would obviously be achieved from studying more realistic situations. We emphasise that in the second--best problem treated below, we will give general results allowing to solve completely the problem for general ambiguity sets, with solutions which are amenable to numerical computations, thus opening the door to such an exploration.

\section{Moral hazard and the second--best problem}\label{sec:2}

We now study the so--called second best problem, corresponding to a Stackelberg--like equilibrium between the Principal and the Agent. Now, the Principal has no control (or cannot observe) the effort level chosen by the Agent. Hence, his strategy is to first compute the best--reaction function of the Agent to a given contract, and to determine his corresponding optimal effort (if it exists) and then use this in his own utility function to maximise over all the contracts. Obviously, the above approach can only work if the Principal can actually find the optimal effort of the Agent. Therefore, the set of admissible contracts in the second best setting must at least be reduced to the contracts $\xi$ such that there exists (possibly several) $a^\star \in\Ac$ with 
$$R\leq U_0^A(\xi)= \underset{\P\in\mathcal P_A^{a^\star }}{\inf}\E^\P\left[\mathcal U_A\left(\xi-\int_0^Tk(a^\star _s)ds\right)\right].$$

As we will see below, this set of contracts is actually equal to $\Cc$, so that the above restriction is without  loss of generality.\vspace{0.5em}

\begin{Remark}
\label{section:generalRK}
Before turning to the solution to the moral hazard problem, notice that solving the problem of the Agent only involves looking at the contract $\xi$ on the support of his beliefs set $\mathcal P_A$. Therefore, the only information about $\xi$ that we can obtain from solving the Agent's problem will be in a $\Pc_A-$quasi sure sense. Therefore, the Principal will always have a degree of freedom when choosing the contracts on the support of $\Pc_P\backslash\Pc_A$. This will be important later on.
\end{Remark}


\subsection{The Agent's problem}
The aim of this section is to prove that despite the generality of our setting, the utility of the Agent, as well as his optimal effort, can always be characterised completely for any admissible contract $\xi\in\mathcal C$. Our result relies essentially on the recent theory of second--order BSDEs, introduced by Soner, Touzi and Zhang \cite{soner2012wellposedness}, and revisited in a framework suitable for our purpose by Possama\"i, Tan and Zhou \cite{possamai2015stochastic}. 

\vspace{0.5em}
Before starting, we will need to introduce the following spaces

\vspace{0.5em}
$\bullet$ $\mathbb D^{\rm exp}_{\Pc_A}$ is the set of processes $Y$, $\mathbb G^{\Pc_A,+}-$progressively measurable, $\Pc_A-q.s.$ c\`adl\`ag, and such that
$$\underset{\P\in\Pc_A}{\sup}\E^\P\left[\exp\left(p\underset{0\leq t\leq T}{\sup}|Y_t|\right)\right]<+\infty,\ \forall p\geq 0.$$ 
$\bullet$ $\mathbb I_{\Pc_A}$ is the set of processes $K$, $\mathbb G^{\Pc_A}-$predictable, $\Pc_A-q.s.$ c\`adl\`ag and non--decreasing, null at $0$ with
$$\underset{\P\in\Pc_A}{\sup}\E^\P[K_T^p]<+\infty,\ \forall p\geq 0.$$ 
Let us next introduce the following 2BSDE
\begin{equation}
\label{eq:2bsde}
Y_t=\xi-\int_t^T\left(\frac{R_A}{2}\abs{Z_s}^2\widehat\alpha_s+\underset{a\in[0,a_{\max}]}{\inf}\left\{k(a)-aZ_s\right\}\right)ds-\int_t^TZ_s\widehat\alpha_s^{1/2}dW_s-\int_t^TdK_s,\ t\in[0,T],\ \Pc_A-q.s.
\end{equation}  
We say that the triplet $(Y,Z,K)\in\mathbb D^{\rm exp}_{\Pc_A}\times\cup_{p\geq 0}\mathbb H^p_{\Pc_A}(\R,\mathbb G^{\Pc_A})\times\mathbb I_{\Pc_A}$ is the maximal solution to \eqref{eq:2bsde} if it indeed satisfies \eqref{eq:2bsde} $\Pc_A-q.s.$, if $K$ satisfies the following minimality condition for any $\P\in\Pc_A$
\begin{equation}\label{eq:minmini}
K_t=\underset{\P^\prime\in\mathcal P_A(\P,t^+)}{{\rm essinf}^\P}\E^{\P^\prime}\left[\left.K_T\right|\mathcal F_t\right],\ \P-a.s.,
\end{equation}
and if for any other solution $(Y^\prime,Z^\prime,K^\prime)\in\mathbb D^{\rm exp}_{\Pc_A}\times\cup_{p\geq 0}\mathbb H^p_{\Pc_A}(\R,\mathbb G^{\Pc_A})\times\mathbb I_{\Pc_A},$ we have $Y_t\geq Y_t^\prime$, $t\in[0,T]$, $\Pc_A-q.s.$

\vspace{0.5em}
Our main result for this section is then the following representation, whose proof is deferred to the Appendix.
\begin{Proposition}\label{prop:optimaleffort}
For any $\xi\in\Cc$, the value function of the Agent verifies
$$U_0^A(\xi)=-\exp(-R_AY_0),$$
and the optimal effort of the Agent is given by the unique $(a^\star (Z_s))_{s\in[0,T]}$ which satisfies
$$\underset{a\in[0,a_{\max}]}{\inf}\left\{k(a)-aZ_s\right\}= k(a^\star (Z_s))-a^\star (Z_s)Z_s, \; s\in [0,T],$$
where $(Y,Z)$ is the maximal solution to \eqref{eq:2bsde}. Furthermore, $\xi\in\Cc$ if and only if \begin{equation}\label{eq:defR0}
Y_0\geq -\frac{\log(-R)}{R_A}=:R_0.
\end{equation}
\end{Proposition} 

\subsection{Admissible contracts}
We have thus solved the problem of the Agent for any $\xi\in\Cc$, in the sense that we have systematically found his optimal action for a given $\xi\in\Cc$. Along the way, we proved that any $\xi\in\Cc$ had the following decomposition
\begin{equation}\label{admiss:contract}
\xi=Y_0+\int_0^Tg(Z_s, \widehat\alpha_s)ds+\int_0^TZ_s\widehat\alpha_s^{1/2}dW_s+\int_0^TdK_s,\ \Pc_A-q.s.,
\end{equation}
for some $Z\in\cup_{p\geq 0}\mathbb H_{\Pc_A}(\R,\G^{\Pc_A})$, $Y_0\geq R_0$ and some $K\in\mathbb I_{\Pc_A}$ satisfying the minimality condition \reff{eq:minmini}, where we have defined for simplicity for any $(z,a)\in\R^2$
$$g(z,a):=\frac{R_A}{2}|z|^2a+\underset{a^\prime\in[0,a_{\max}]}{\inf}\left\{k(a^\prime)-a^\prime z\right\}.$$

Notice that, as already mentioned in Remark \ref{section:generalRK}, we only retrieved information about $\xi$ in a $\Pc_A-q.s.$ sense. 

\vspace{0.5em}
Conversely, let $\widehat \Cc$ be the set of random variables $\xi$ such that there exist processes $(Z,K)\in\cap_{p\geq 0}\mathbb H^p_{\Pc_A}(\R,\G^{\Pc_A})\times\mathbb I_{\Pc_A}$ and some $\widehat\xi\in \Cc$ such that
\begin{equation}\label{decompXi}\xi=\begin{cases}
\displaystyle \xi^{Y_0,Z,K}:=Y_0+\int_0^Tg(Z_s,\widehat\alpha_s)ds+\int_0^TZ_s\widehat\alpha_s^{1/2}dW_s+K_T,\ \Pc_A-q.s.,\\[0.8em]
\displaystyle \widehat\xi,\ \Pc_P\backslash\Pc_A-q.s.
\end{cases}
\end{equation}
The above reasoning proves that $\widehat \Cc\cap\Cc=\Cc$, and that an arbitrary element of $\widehat \Cc$ will belong to $\Cc$ if $(Z,K)$ satisfy appropriate integrability conditions ensuring that $\xi$ satisfies \eqref{eq:moexp}. We let $\mathcal K$ be the corresponding set of processes $Z$ and $K$, which, for technical reasons, verify in addition that 
$$ \mathcal E\left(-R_P\int_0^\cdot\widehat\alpha_s^{\frac12}(1-Z_s)dW_s^{a^\star (Z_\cdot)}\right)\text{ is a $\P-$martingale, $\forall\P\in\Pc_P^{a^\star (Z_\cdot)}$}.$$
Finally, we will take as admissible contracts the following class
$$\mathcal C^{\rm {\rm SB}}:=\left\{\xi\in \Cc,\; \exists (Y_0,Z,K)\in[R_0,+\infty)\times\mathcal K,\ \xi=Y_0+\int_0^Tg(Z_s,\widehat\alpha_s)ds+\int_0^TZ_s\widehat\alpha_s^{1/2}dW_s+K_T,\; \Pc_A-q.s.\right\}.$$
\subsection{The Principal's problem}\label{sec:princ}
\subsubsection{General formulation of the problem and degeneracies}
From Remark \ref{section:generalRK} and the previous section, the Principal's problem can always be written as
\begin{equation}
\label{pbPrincipal}U_0^P:= \underset{(Y_0,Z,K, \xi)\in [R_0,\infty)\times\Kc\times\Cc}{\sup}\min \Bigg\{ U_0^P(\xi^{Y_0,Z,K}),   \underset{\P\in\Pc_P^{a^\star (Z_\cdot)}\setminus \Pc_A^{a^\star (Z_\cdot)}}{\inf}\E^\P\big[\mathcal U_P\big(B_T- \xi\big)\big] \Bigg\},
\end{equation}
with $$U_0^P(\xi^{Y_0,Z,K}):=\underset{\P\in\Pc_P^{a^\star (Z_\cdot)}\cap \Pc_A^{a^\star (Z_\cdot)}}{\inf}\E^\P\big[\mathcal U_P\big(B_T-\xi^{Y_0,Z,K}\big)\big]$$
To simplify our study, we set the following assumption, which merely imposes a boundedness property on the volatilities in which the Principal believes.
\begin{Assumption}\label{assumption:bornePrincipal}
There exists some positive constant $M$ such that for any $(t,\omega,\P)\in[0,T]\times\Omega\times\Pc_P(t,\omega)$, 
$$0\leq \widehat{\alpha}_s\leq M,\; \P-a.s.$$
\end{Assumption}
We then have the following result, which states that as soon as the beliefs of the Principal and the Agent are disjoint, the problem degenerates, and the Principal can obtain his maximal possible utility, that is to say $0$. In other words, we are in a situation that is reminiscent of an arbitrage opportunity.
\begin{Proposition}\label{prop:degenerate2nd} Let Assumption \ref{assumption:bornePrincipal} hold. If $\Pc_P\cap \Pc_A=\emptyset$,  then $U_0^P=0$.
\end{Proposition}
\begin{proof} First notice that from \eqref{equiv:intersection}, we automatically have
$$U_0^P= \underset{\xi\in \Cc}{\sup}\  \underset{\P\in\Pc_P^{a^\star (Z_\cdot)}}{\inf}\E^\P\left[\mathcal U_P\big(B_T-\xi\big)\right].$$ 
Let $c\in\R$ be such that $-e^{-R_A(c-k(0)T)}=R$. We define the following contract in $\mathcal C$.
 \begin{equation}\label{decompxiasympt}\xi^n=\begin{cases}
\displaystyle c,\ \Pc_A-q.s.,\\[0.8em]
\displaystyle -n,\ \Pc_P\backslash\Pc_A-q.s.
\end{cases}
\end{equation}
The utility that the Agent receives is then
\begin{align*}
U_0^A(c)&= \sup_{a\in \mathcal A}\mathbb E^{\mathbb P^a}\left[-e^{-R_A\left( c-\int_0^Tk(a_s) ds\right)}\right]= -e^{-R_A\left( c-k(0)T\right)}= R,
\end{align*}
using the definition of $c$. We now turn to the Principal's utility. We directly have for any $a\in \mathcal A$
\begin{align*}
U_0^P&\geq   \underset{\P\in\Pc_P^{a}\setminus \Pc_A^{a}}{\inf}\E^\P\left[\mathcal U_P\left(B_T+n\right)\right]\\
&= e^{-R_P n}\underset{\P\in\Pc_P^{a}\setminus \Pc_A^{a}}{\inf}\E^\P\left[-e^{-R_P B_T}\right]\\
&=e^{-R_P n}\underset{\P\in\Pc_P^{a}\setminus \Pc_A^{a}}{\inf}\E^\P\left[-\mathcal E\left(- R_P\int_0^T\widehat{\alpha}_s^{\frac12} dW_s^a \right)e^{-R_P \int_0^T a_s ds} e^{\frac{R_P^2}2\int_0^T \widehat{\alpha}_s ds}\right]\\
&\geq -e^{-R_P n+R_P^2T/2M}.
\end{align*}
Thus, $U_0^P \geq \lim\limits_{n\to +\infty} \underset{\P\in\Pc_P^{a}\setminus \Pc_A^{a}}{\inf}\E^\P\left[\mathcal U_P\left(B_T+n\right)\right]=0$. Since the reverse inequality is trivial, this ends the proof.
\end{proof}
The case $\Pc_P\cap \Pc_A=\emptyset$ can be seen as a non realistic approach to the problem, since the Agent and the Principal have completely disjoint estimates on the volatility. We now turn to a more realistic case, where we assume that $\Pc_P\cap \Pc_A\neq \emptyset$. We then have the following proposition which simplifies greatly the problem of the Principal \eqref{pbPrincipal}

\begin{Proposition}\label{prop:super} Let Assumption \ref{assumption:bornePrincipal} hold. If $\Pc_P\cap \Pc_A\neq\emptyset$, 
\begin{equation}
\label{pbPrincipalNonvide}U_0^P=\underset{(Y_0,Z,K)\in [R_0,\infty)\times\Kc}{\sup}\  \underset{\P\in\Pc_P^{a^\star (Z_\cdot)}\cap \Pc_A^{a^\star (Z_\cdot)}}{\inf}\E^\P\left[\mathcal U_P\left(B_T-\xi^{Y_0,Z,K}\right)\right].
\end{equation}
\end{Proposition}
\begin{proof}
We have $U_0^P= \underset{\xi\in \Cc^{\rm {\rm SB}}}{\sup}\ \widetilde U_0^P(\xi),$
where for any $\xi\in \Cc^{\rm {\rm SB}}$
$$ \widetilde U_0^P(\xi):=   \underset{\widetilde\xi\in \widetilde\Cc}{\sup}\ \min \left\{U_0^P(\xi), \,  \underset{\P\in\Pc_P^{a^\star (Z^\xi_\cdot)}\setminus \Pc_A^{a^\star (Z^\xi_\cdot)}}{\inf}\E^\P\left[\mathcal U_P\left(B_T-\widetilde\xi\right)\right] \right\}. $$
Notice that the first term in the minimum does not depend on $\widetilde \xi$. Furthermore, we know by Proposition \ref{prop:degenerate2nd} that for any $\xi\in\Cc^{\rm {\rm SB}}$
$$0=\underset{\widetilde\xi\in \widetilde\Cc}{\sup}\underset{\P\in\Pc_P^{a^\star (Z^\xi_\cdot)}\setminus \Pc_A^{a^\star (Z^\xi_\cdot)}}{\inf}\E^\P\left[\mathcal U_P\left(B_T-\widetilde\xi\right)\right] \geq U_0^P(\xi).$$
Therefore, $\widetilde U_0^P(\xi)=U_0^P(\xi)$. 
\end{proof}

From now on, we will always assume that $\Pc_P\cap \Pc_A\neq\emptyset$ and that Assumption \ref{assumption:bornePrincipal} holds, since we have already solved the other case in Proposition \ref{prop:degenerate2nd}.

\subsubsection{A sub--optimal version of \eqref{pbPrincipalNonvide}}
From \eqref{decompXi} and Proposition \ref{prop:super}, we know that we can restrict our attention to contracts of the form $\xi^{Y_0,Z,K}.$
However, in order to solve this problem, we actually need to have more information on the non--decreasing process $K$.  Using similar intuitions as the ones given in \cite{cvitanic2014moral,cvitanic2015dynamic}, we expect that when the contract $\xi$ is sufficiently "smooth", we can find a $\G^{\Pc^A}-$predictable process $\Gamma$ such that for every $\P\in\mathcal P^{a^\star (Z^\xi_\cdot)}_A$
\begin{equation}
K_t= \int_0^t \left(\frac12 \widehat \alpha_s \Gamma_s -\underset{a \in  \mathbb R}{\inf}\left\{\frac12 a\Gamma_s +\mathbf{1}_{a\in \mathbf{D}(s,B_\cdot)}  \right\}\right)ds,
\label{eq:kk}
\end{equation}
where the indicator function ${\bf 1}_{a\in\mathbf{D}(t,\omega)}$ is the one from convex analysis, and is equal to $0$ when $a$ indeed belongs to $\mathbf{D}(t,\omega)$, and $+\infty$ otherwise. However, in general, such a decomposition for $K$ is not true for every $\xi\in\Cc^{\rm {\rm SB}}$. We will therefore start by solving the Principal problem for a particular sub--class of contracts in $\Cc^{\rm {\rm SB}}$ such that the process $\Gamma$ exists, and then show, under appropriate assumptions, that the Principal's value function is not actually affected by this restriction. For simplicity, we denote by $\mathfrak K$ the set of processes $Z$ and $\Gamma$ such that $\Gamma$ is $\G^{\Pc^A}-$predictable and $(Z,K^\Gamma)\in\Kc$, where
$$K_t^\Gamma:= \int_0^t \left(\frac12 \widehat \alpha_s \Gamma_s -\underset{a \in  \mathbb R}{\inf}\left\{\frac12 a\Gamma_s +\mathbf{1}_{a\in \mathbf{D}(s,B_\cdot)}  \right\}\right)ds.$$

\vspace{0.5em}
Building upon \eqref{admiss:contract} and \eqref{eq:kk}, we consider the class $\mathfrak C^{\rm {\rm SB}}\subset\Cc^{\rm {\rm SB}}$ of contracts $\xi$ admitting the decomposition $\xi=Y^{Y_0,Z,\Gamma}_T$ for some $Y_0\geq R_0$, and some $(Z,\Gamma)\in\mathfrak K$, where for any $t\in[0,T]$
\begin{align}\label{admiss:contract2}
 Y^{Y_0,Z,\Gamma}_t:=&\ Y_0+\int_0^t\left(\frac12 \widehat \alpha_s \Gamma_s  -\underset{\alpha \in \R}{\inf}\left\{\frac12 \alpha\Gamma_s +\mathbf{1}_{\alpha\in \mathbf{D}(s,B_\cdot)}   \right\}+g(Z_s,\widehat\alpha_s)\right)ds+\int_0^tZ_s\widehat\alpha_s^{1/2}dW_s,\; \mathcal P_A-q.s.
\end{align}
\begin{Remark}
Let us say a word about implementability of the contracts in the class $\mathfrak C^{\rm {\rm SB}}$. They can actually be rewritten as follows
$$Y^{Y_0,Z,\Gamma}_T= Y_0+\frac12\int_0^T\Gamma_s d\langle B\rangle_s -\int_0^T\left(\underset{\alpha \in \R}{\inf}\left\{\frac12 \alpha\Gamma_s +\mathbf{1}_{\alpha\in \mathbf{D}(s,B_\cdot)}   \right\}+g(Z_s,\widehat\alpha_s)\right)ds+\int_0^TZ_sdB_s,\; \mathcal P_A-q.s.$$
Therefore, the Principal should be able to reward the Agent using the path of $B$, which could easily be done in practice using stocks on the value $B$ of the firm, or forward contracts for instance $($one would then of course have to approximate the stochastic integral by a finite Riemann sum$)$. The term involving the quadratic variation is more complex however. Since it is deeply linked to the volatility of the output, one could try to replicate it using variance swaps contracts or log contracts, which are known to be linked to this quadratic variation. This is in line with the classical managerial compensation using stock options for instance. Another important point to realise is that, in general, the Principal can substitute in the contract the quadratic variation by $d(B_t^2)$ instead, without losing too much utility. This point has been raised recently by A\"id, Possama\"i and Touzi {\rm\cite{aidpt}} in a Principal--Agent problem for electricity pricing where the Agent $($the client$)$ controls the variability of his consumption. Their numerical results show that the loss of utility for the Principal when doing this substitution is not that large in many situations. We believe that this result should still hold in our framework, which would provide a more practical way to implement the optimal contract.
\end{Remark}

\vspace{0.5em}
We define the following sub--optimal problem for the Principal
\begin{equation}\label{secondbest:subgeneral}
\underline{U}_0^P:=\underset{Y_0\geq R_0}{\sup}\;\underline U_0^P(Y_0),
\end{equation}
where for any $Y_0\geq R_0$
$$\underline U_0^P(Y_0):=\underset{(Z,\Gamma)\in\mathfrak K}{\sup}\; \underset{\P\in\Pc_P^{a^\star (Z_\cdot)}\cap \Pc_A^{a^\star (Z_\cdot)}}{\inf}\E^\P\left[\mathcal U_P\left(B_T-Y^{Y_0,Z,\Gamma}_T\right)\right].$$
Notice first that by linearity of $Y^{Y_0,Z,\Gamma}$ in $Y_0$ and the fact that $U_P$ is non--decreasing, we deduce immediately that 
$$\underline{U}_0^P=\underline U_0^P(R_0).$$
Now, all the interest of concentrating on $\underline U_0^P(R)$ is that it is simply the value function of zero--sum stochastic differential game, under weak formulation, with the two controlled state variables $B$ and $Y^{R_0,Z,\Gamma}$, with controls $(Z,\Gamma)\in\mathfrak K$ for the Principal and $\widehat\alpha$ for the "Nature", and dynamics
$$\begin{cases}
\displaystyle B_t=\int_0^ta^\star(Z_s)ds+\int_0^t\widehat\alpha_s^{1/2}dW_s^{a^\star(Z_\cdot)},\; t\in[0,T],\; \Pc_A\cap\Pc_P-q.s.,\\
\displaystyle  Y^{R_0,Z,\Gamma}_t= R_0+\int_0^t\left(\frac12 \widehat \alpha_s \Gamma_s  -\underset{a \in \R}{\inf}\left\{\frac12 a\Gamma_s +\mathbf{1}_{a\in \mathbf{D}(s,B_\cdot)}   \right\}+\frac{R_A}{2}\widehat\alpha_s|Z_s|^2+k(a^\star(Z_s))\right)ds\\
\hspace{4.6em}\displaystyle+\int_0^tZ_s\widehat\alpha_s^{1/2}dW_s^{a^\star(Z_\cdot)},\; t\in[0,T],\;\mathcal P_A\cap\Pc_P-q.s.
\end{cases}$$
Under this form, the sub--optimal problem of the Principal becomes amenable to the dynamic programming approach to differential games, and in particular, in a Markovian setting, $\underline U_0^P$ can be linked to the associated Hamilton--Jacobi--Bellman--Isaacs (HJBI for short) PDE. The aim of the next section is to take advantage of this representation of the value function to prove that the restriction to contracts in $\mathfrak C^{\rm {\rm SB}}$ is, under natural assumptions, without loss of generality.

\subsubsection{The Hamilton--Jacobi--Bellman--Isaacs approach}
As mentioned above, we now restrict our attention to the Markovian setting\footnote{\label{foot:ppde}Actually, our approach would also work in non--Markovian case, provided that one uses the recently developed theory of viscosity solutions for path--dependent PDEs, in a series of papers by Ekren, Keller, Ren, Touzi and Zhang \cite{ekren2014viscosity,ekren2016viscosity,ekren2012viscosity,ren2014comparison,Ren2014,ren2015comparison}. We preferred to present our arguments in the Markovian case to avoid additional technicalities. See however Section \ref{sec:comp} for a specific non--Markovian case.}, which requires the following assumption.
\begin{Assumptionn}[\textbf{M}] For $\Psi\in \{ A,P\}$ and any $(s,\omega)\in [0,T]\times \Omega$, we have, abusing notations slightly
$$\mathbf{D}_\Psi(s,\omega)=\mathbf{D}_\Psi(s,B_s(\omega)).$$
\end{Assumptionn}

We next define for any $0\leq t\leq s\leq T$, $(Z,K)\in \mathcal K$ and $y\in \R$
\begin{align*}
Y^{t,y,Z,K}_s:=y+\int_t^s\left(\frac{R_A}{2}\abs{Z_r}^2\widehat\alpha_r+\underset{a\in[0,a_{\max}]}{\inf}\left\{k(a)-aZ_r\right\}\right)dr+\int_t^sZ_r\widehat\alpha_r^{1/2}dW_r+\int_t^TdK_r,\; s\in[t,T].
\end{align*}

Since our setting is now Markovian, we can simplify the notations for the ambiguity sets of the Principal and the Agent to $\mathcal P_P(t,x)$ and $\mathcal P_A(t,x)$, for any $(t,x)\in[0,T]\times\R$. Consider now the dynamic version of the value function of the Principal

$$u(t,x,y):= \underset{(Z,K)\in \mathcal K}{\sup}\ \underset{\P\in\Pc_P^{a^\star (Z_\cdot)}(t,x)\cap \Pc_A^{a^\star (Z_\cdot)}(t,x)}{\inf}\E^\P\left[\mathcal U_P\left(B_T-Y_T^{t,y, Z,K}\right)\right].$$ 

Since it is clear from the definition of $Y^{t,y,Z,K}=y+Y^{t,0,Z,K}$, we deduce immediately that
$$u(t,x,y)=-e^{R_P y} v(t,x),$$
where
\begin{equation}\label{pb:fcontrol}v(t,x):= \underset{(Z,K)\in \mathcal K}{\inf}\ \underset{\P\in\Pc_P^{a^\star (Z_\cdot)}(t,x)\cap \Pc_A^{a^\star (Z_\cdot)}(t,x)}{\sup}\E^\P\left[e^{-R_P\left(B_T-Y_T^{t,0, Z,K}\right)}\right].\end{equation}

We are now going to make a series of assumptions concerning the function $v$ defined above. These assumptions will be related to standard properties of stochastic control or stochastic differential games, and although expected and standard, they may be quite hard to prove in an extremely general setting. We will comment on this further after the statement of the assumptions themselves. 
\begin{Assumptionn}[\textbf{PPD}]\label{PPD}
The map $v$ is continuously differentiable in time on $[0,T]$, and twice continuously differentiable with respect to $x$ on $\R$. Besides, for any family $\left\{ \theta^{Z,K,\mathbb P}, \; (Z,K) \in \mathcal K, \,\P\in\Pc_P^{a^\star (Z_\cdot)}\cap \Pc_A^{a^\star (Z_\cdot)} \right\}$ of stopping times independent of $\mathcal F_t$, we have
\begin{equation}\label{ppd}
v(t,x)= \underset{(Z,K)\in \mathcal K}{\inf}\ \underset{\P\in\Pc_P^{a^\star (Z_\cdot)}(t,x)\cap \Pc_A^{a^\star (Z_\cdot)}(t,x)}{\sup}\E^\P\left[v(\theta^{Z,K,\mathbb P}, B_{\theta^{Z,K,\P}})e^{R_P Y_{\theta^{Z,K,\P}}^{t,0, Z,K}}\right]. 
\end{equation}
Furthermore, there exists $(Z^\star,K^\star)\in \mathcal K$ such that 
\begin{equation}\label{optimum:ZK}
v(t,x)=\underset{\P\in\Pc_P^{a^\star (Z^\star_\cdot)}(t,x)\cap \Pc_A^{a^\star (Z^\star_\cdot)}(t,x)}{\sup}\E^\P\left[v(\theta^{Z^\star,K^\star,\mathbb P}, B_{\theta^{Z^\star,K^\star,\P}})e^{R_P Y_{\theta^{Z^\star,K^\star,\P}}^{t,0, Z^\star,K^\star}}\right]. 
\end{equation}
\end{Assumptionn}
First of all, the smoothness assumption on $v$ is for simplicity and because we do not want to add an extra layer of technicalities by having to rely on the notion of viscosity solutions. However, the same line of reasoning would still go through in that case. Next, roughly speaking, \eqref{ppd} means that we are assuming that $v$ satisfies the dynamic programming principle. For standard stochastic control problems, this is a result known to hold in extremely general settings (see for instance \cite{karoui2013capacities,karoui2013capacities2} for a recent account). However, it has been known that this is a much more complex problem for differential games, since the seminal paper on the subject by Fleming and Souganidis \cite{fleming1989existence}. However, such results have already been proved in the literature, for instance by Buckdahn and Li \cite{buckdahn2008stochastic}, or Bouchard, Moreau and Nutz \cite{bouchard2014stochastic}. The fact that we take it as an assumption here is once more mainly for simplicity, and because we want to give a general idea on the strategy of proof that we think should be used. The verification itself of whether the dynamic programming principle holds should be done on a case by case basis. We also want to insist on the fact that we only require this principle to hod in order to be able to relate the value function of the game to the associated HJBI PDE. Therefore, any other approach not requiring this principle and still allowing to prove such a relationship would work for us. This would be the case for the so--called stochastic viscosity solutions introduced by Bayraktar and S\^irbu for instance, see \cite{bayraktar2012stochastic,bayraktar2013stochastic,sirbu2014stochastic}. Finally, Relation \eqref{optimum:ZK} simply stipulates that there is an optimal control in the maximisation part of the problem of the Principal. This could be in principle relaxed to the existence of $\varepsilon-$optimal controls.

\vspace{0.5em}
Define next the map $G:[0,T]\times \mathbb R^{6}\times \R^+\longrightarrow \R$ for any $(t,x,v,p,q,z,\gamma,\alpha)\in [0,T]\times \mathbb R^{5}\times \R^+$ by
\begin{align*}
G(t,x,v,p,q,z,\gamma,\alpha)=&\ a^\star(z)p+\left(\frac{R_A}2 \alpha |z|^2+k(a^\star(z))+\frac12\alpha\gamma-\inf_{\tilde\alpha\in \R^+}\left\{ \frac12 \tilde\alpha \gamma+\mathbf{1}_{\tilde\alpha\in \mathbf{D}_A(t,x)}\right\}\right)R_P v \\
&+\frac12 \alpha q+\frac12\alpha |z|^2R_P^2 v+\alpha z R_P v+\bf 1_{\alpha\in{\bf D}_P(t,x)}.
\end{align*} 

We can then introduce the following HJBI equation, which should be related to our problem. 
\begin{align}\label{eq:hjb2}
\begin{cases}
\displaystyle -\partial_t\psi(t,x)-\underset{\alpha\in\R^+}{\sup} \underset{(z,\gamma)\in \R^2}{\inf}G(t,x,\psi,\partial_x \psi, \partial_{xx}\psi,z,\gamma,\alpha)=0,\ (t,x)\in[0,T)\times\R,\\
\psi(T,x)=e^{-R_Px},\ x\in\R.
\end{cases}
\end{align}
Our strategy now is to prove that both the value function of the original and sub--optimal problems of the Principal solve this PDE. Then, by a uniqueness argument (which will require a further assumption), we will be able to affirm that they are equal.

\vspace{0.5em} 
First, we introduce the notion of (classical) super--solution to PDE \eqref{eq:hjb2}.
\begin{Definition}\label{def:supersol} A map $v$ from $[0,T]\times\mathbb R$ into $\mathbb R$ is a smooth super--solution to PDE \eqref{eq:hjb2} if $v$ is once continuously differentiable in time and twice continuously differentiable with respect to $x$ and satisfies 
\begin{align}\label{eq:hjb2:sub}
\begin{cases}
\displaystyle -\partial_tv(t,x)-\underset{\alpha\in \mathbb R^+}{\sup} \underset{ (z,\gamma)\in\R^2}{\inf}G(t,x,v(t,x),\partial_x v(t,x), \partial_{xx}v(t,x),z,\gamma,\alpha)\geq 0,\ (t,x)\in[0,T)\times\R,\\
v(T,x)=e^{-R_P x},\ x\in\R.
\end{cases}
\end{align}

\end{Definition}

We now assume that a comparison theorem for the HJBI equation \eqref{eq:hjb2} holds, and that the latter admits a unique classical solution, which by a standard verification argument would then be equal to the dynamic version of the value function of the sub--optimal problem of the Principal. This is why \eqref{eq:psi} below holds in this case.
\begin{Assumptionn}[\textbf{C}]\label{thm:comparison}
There exists a unique smooth solution $\psi$ to PDE \eqref{eq:hjb2} such that $\psi$ is once continuously differentiable with respect to time and twice continuously differentiable with respect to $x$ such that
\begin{equation}\label{eq:psi}
 -e^{R_PR_0}\psi(0,0)=\underline U_0^P.
 \end{equation}
Moreover, assume that $v$ is a smooth super--solution to PDE \eqref{eq:hjb2} then for all $(t,x)\in [0,T]\times \R$ we have $\psi(t,x)\leq v(t,x)$.
\end{Assumptionn}

We finally consider an extra technical assumption, ensuring that the maximum in the Hamiltonian of the HJBI equation is attained and that the corresponding maximiser is sufficiently "smooth".
\begin{Assumptionn}[\textbf{A}]\label{assumption:tildealpha}
We assume that for any $(t,x,v,p,q)\in [0,T)\times \mathbb R^4$, there exists a maximiser $\widetilde\alpha(t,x,v,p,q)$ for the map
$$\alpha\longmapsto  \underset{ (z,\gamma)\in\R^2}{\inf} \, G(t,x,v,p,q,z,\gamma,\alpha),$$
such that the following SDE has a unique strong solution
$$ X_t= \int_0^t\widetilde\alpha(s,X_s, v(s,X_s),\partial_x v(s,X_s),\partial_{xx}v(s,X_s) ) dB_s,\ \P_0-a.s.$$
\end{Assumptionn}
The following lemma will be useful in the sequel.
\begin{Lemma}\label{lemma:approximation}
Under Assumptions $(\mathbf{M})$, $(\mathbf{PPD})$ and $(\mathbf{A})$, we can define for any $(t,x)\in[0,T]\times\R$ a probability measure $\widetilde{\mathbb P}\in \Pc_P^{a^\star (Z^\star_\cdot)}(t,x)\cap\Pc_A^{a^\star (Z^\star_\cdot)}(t,x)$ such that 
\begin{equation}\label{eq:alpha}
\widehat{\alpha}_t=\widetilde\alpha\big(t,B_t,v(t,B_t),\partial_xv(t,B_t),\partial_{xx}v(t,B_t)\big),\, dt\otimes\widetilde{\mathbb P}-a.e.
\end{equation}
Moreover, let $\varphi$ be a map from $[0,T]\times \R$ into $\mathbb R^+\setminus \{ 0\}$. Then, there is a sequence $(\Gamma^{\star,n})_{n\in \mathbb N}$ of $\G^{\Pc^A}-$predictable processes such that
\begin{equation}\label{convergence:kstar}\E^{\tilde\P}\left[  \int_{0}^{T} \varphi(r, B_r)dK^\star_r- \int_{0}^{T} \varphi(r, B_r) k^{\star,n}_r dr\right] \underset{n\to +\infty}{\longrightarrow} 0 ,\end{equation}

with $$k^{\star,n}_r:=\frac12 \widehat \alpha_r \Gamma^{\star,n}_r -\underset{a \in  \mathbb R}{\inf}\left\{\frac12 a\Gamma^{\star,n}_r +\mathbf{1}_{a\in \mathbf{D}_A(r,B_r)}  \right\},\;  r\in [0,T].$$
\end{Lemma}
\begin{proof}
Assumption $(\mathbf{A})$ provides directly the existence of the probability $\widetilde{\mathbb P}\in \Pc_P^{a^\star (Z^\star_\cdot)}(t,x)\cap\Pc_A^{a^\star (Z^\star_\cdot)}(t,x) $. Next, we set
$$H_T:=  \int_{0}^{T} \varphi(r, B_r)dK^\star_r \geq 0,\;  \widetilde{\mathbb P}-a.s. $$
Then, since $\varphi$ is a positive function, there exists by standard density arguments a sequence of non--negative predictable processes $(k^{\star,n})_{n\in \mathbb N}$ such that 
$$\E^{\widetilde\P}\left[ H_T- \int_{0}^{T} \varphi(r, B_r) k^{\star,n}_r dr\right] \underset{n\to +\infty}{\longrightarrow} 0.$$

Fix some $\omega$ in the support of the probability measure $\widetilde\P$. We will now prove that the map
$$s_t(\omega,\cdot):\gamma\longmapsto \frac{1}{2}\widetilde\alpha(t,B_t(\omega),  v(t,B_t(\omega)),\partial_xv(t,B_t(\omega)),\partial_{xx}v(t,B_t(\omega)) )\gamma-\underset{a \in  \mathbb R}{\inf}\left\{\frac12 a\gamma +\mathbf{1}_{a\in \mathbf{D}_A(t,B_t(\omega))}  \right\}, $$ is surjective from $\mathbb R$ into $\mathbb R^+$. 

\vspace{0.5em}
First notice that since $\widetilde\alpha(t,B_t(\omega), v(t,B_t(\omega)),\partial_x v(t,B_t(\omega)),\partial_{xx}v(t,B_t(\omega)) ) \in  \mathbf{D}_A(t,B_t(\omega))$, we have $s_t(\omega,\gamma)\geq 0$. Indeed, $s_t(\omega,0)=0$, and if $\gamma>0$, we have
$$ s_t(\omega,\gamma)= \frac{1}{2}\gamma \left(\tilde\alpha(t,B_t(\omega), v(t,B_t(\omega)),\partial_x v(t,B_t(\omega)),\partial_{xx}v(t,B_t(\omega)) )-a_-(t,B_t(\omega))\right)\geq 0,$$
and if $\gamma<0$, 
$$ s_t(\omega,\gamma)= \frac{1}{2}\gamma \left(\widetilde\alpha(t,B_t(\omega),v(t,B_t(\omega)), \partial_x v(t,B_t(\omega)),\partial_{xx}v(t,B_t(\omega)) )-a_+(t,B_t(\omega))\right)\geq 0,$$
where we defined
$$a_-(t,B_t(\omega)):=\underset{a \in  \mathbb R}{\inf} \left\{a+\mathbf{1}_{a\in \mathbf{D}(t,B_t(\omega))}\right\},\; a_+(t,B_t(\omega)):=\underset{a \in  \mathbb R}{\sup} \left\{a-\mathbf{1}_{a\in \mathbf{D}_A(t,B_t(\omega))}  \right\}.$$
Next, it is clear that if $a_{-}(t,B_t(\omega))<\widetilde\alpha(t, B_t(\omega),v(t,B_t(\omega)),\partial_x v(t,B_t(\omega)),\partial_{xx}v(t,B_t(\omega)) )<a_{+}(t,B_t(\omega)),$ then $s_t(\omega,\gamma)$ goes to $+\infty$ as $\gamma$ goes to $\pm\infty$. Besides, if $\widetilde\alpha(t, B_t(\omega),v(t,B_t(\omega)),\partial_x v(t,B_t(\omega)),\partial_{xx}v(t,B_t(\omega)) ) =a_-(t,B_t(\omega))$ (resp. $a_+(t,B_t(\omega))$), then by letting $\gamma\longrightarrow -\infty$ (resp.  $\gamma\longrightarrow +\infty$) we still have $s_t(\omega,\gamma)\to +\infty$. Hence, the surjectivity of $s_t(\omega,\cdot)$ from $\mathbb R$ into $\mathbb R^+$ is clear, since this map is also continuous. Thus, using a classical mesurable selection argument we deduce that for any $n\in \mathbb N$, there exists a $\mathbb G^{\mathcal P_A}-$predictable process $\Gamma^{n,\star}\in \mathbb I_{\Pc_A}$ such that
$$ k^{\star,n}_t(\cdot)= s_t(\cdot, \Gamma^{n,\star}_t(\cdot)).$$
Thus the approximation \eqref{convergence:kstar} holds.
\end{proof}
\begin{Proposition}\label{prop:supersol} Let Assumptions $(\mathbf{M})$, $(\mathbf{PPD})$, $(\mathbf{A})$ and $(\mathbf{C})$ hold. 
Then $v$ is a super--solution of HJB equation \eqref{eq:hjb2} in the sense of Definition \ref{def:supersol}.
\end{Proposition}

\begin{proof}
We proceed by contradiction. We assume that there exists $(t_0,x_0)\in [0,T)\times \R$ and $\delta>0$
$$ -\partial_tv(t_0,x_0)-\underset{\alpha\in \mathbb R^+}{\sup} \underset{ (z,\gamma)\in\R^2}{\inf}G(t_0,x_0,v(t_0,x_0),\partial_x v(t_0,x_0), \partial_{xx}v(t_0,x_0),z,\gamma,\alpha)\leq -3\delta<0.$$
Using the continuity of the function $v$ and Assumption $(\mathbf{A})$, we know that for any $(t,x)\in [0,T)\times \mathbb R$ there exists $\widetilde \alpha(t,x,v(t,x),\partial_x v(t,x),\partial_{xx} v(t,x))\in \mathbb R^+$ and there exists $\varepsilon>0$ such that for any $(t,x)$ in $\mathcal V(t_0,x_0):=[t_0,(t_0+\varepsilon)\wedge T)\times\mathcal B(x_0, \varepsilon)$, where $\mathcal B(x_0,\varepsilon)$ denotes the ball centered at $x_0$ with radius $\varepsilon$, we have for any $(z,\gamma)\in \R^2$
\begin{equation}\label{superosl:contrary} \partial_tv(t,x)+G(t,x,v(t,x),\partial_x v(t,x), \partial_{xx}v(t,x),z,\gamma,\widetilde\alpha(t,x,v(t,x),\partial_x v(t,x),\partial_{xx} v(t,x)))\geq2\delta.
\end{equation}
 For any $\nu:=(Z,\Gamma)\in \mathcal K$ and $\mathbb P\in \Pc_P^{a^\star (Z_\cdot)}(t_0,x_0)\cap\Pc_A^{a^\star (Z_\cdot)}(t,x)$, let $\theta^{\nu,\mathbb P}$ be the first exit time of $(t,B_t, Y_t^{t_0,0, \nu})$ of $\mathcal V(t_0,x_0)\times\mathcal B(y_0,\varepsilon)$. Using Assumption $(\mathbf{PPD})$ we have
\begin{align*}
0&= \underset{\nu\in \mathcal K}{\inf}\; \underset{\P\in\Pc_P^{a^\star (Z_\cdot)(t_0,x_0)}\cap \Pc_A^{a^\star (Z_\cdot)}(t_0,x_0)}{\sup}\E^\P\left[v(\theta^{\nu,\P}, B_{\theta^{\nu,\P}})e^{R_P Y_{\theta^{\nu,\P}}^{t_0,0, \nu}}-v(t_0,x_0) \right]\\
&=\underset{\P\in\Pc_P^{a^\star (Z_\cdot)}(t_0,x_0)\cap \Pc_A^{a^\star (Z_\cdot)}(t_0,x_0)}{\sup}\E^\P\left[v(\theta^{\nu^\star,\P}, B_{\theta^{\nu^\star,\P}})e^{R_P Y_{\theta^{\nu^\star,\P}}^{t_0,0, \nu^\star}}-v(t_0,x_0)\right],
\end{align*}
with $\nu^\star:=(Z^\star,K^\star)$. Since Assumption $(\mathbf{A})$ holds, we have a probability $\widetilde{\mathbb P}\in\Pc_P^{a^\star (Z_\cdot)}(t_0,x_0)\cap \Pc_A^{a^\star (Z_\cdot)}(t_0,x_0)$ such that \eqref{eq:alpha} holds. Denote for simplicity
$$\widetilde \alpha_r:=\widetilde\alpha(r,B_r,v(r,B_r),\partial_xv(r,B_r),\partial_{xx}v(r,B_r) ).$$
We thus obtain by applying It\^o's formula
\begin{align*}
0&\geq  \E^{\widetilde\P}\left[v(\theta^{\nu^\star,\tilde\P}, B_{\theta^{\nu^\star,\widetilde\P}})e^{R_P Y_{\theta^{\nu^\star,\widetilde\P}}^{t_0,0, \nu^\star}}-v(t_0,x_0) \right]\\
&=  \E^{\widetilde\P}\left[\int_{t_0}^{\theta^{\nu^\star,\widetilde\P}} e^{R_P Y_{r}^{t_0,0, \nu^\star}}\left( \partial_t v(r,  B_{r})+G(r,B_r,v(r,B_r),\partial_x v(r,B_r), \partial_{xx}v(r,B_r),Z^\star_r,\Gamma^{n,\star}_r,\widetilde\alpha_r\right)\right]\\
&+ \E^{\widetilde\P}\left[\int_{t_0}^{\theta^{\nu^\star,\widetilde\P}} e^{R_P Y_{r}^{t_0,0, \nu^\star}}  v(r,  B_{r}) \left( dK_r^\star- k^{\star,n}_r dr  \right) \right],
\end{align*} using the same notation that those in Lemma \ref{lemma:approximation} with the choice $\varphi=v$ (we recall that by definition $v>0$). From \eqref{convergence:kstar}, we deduce that there exists $n_0\in \mathbb N$ such that for any $n\geq n_0$, 
$$ \E^{\widetilde\P}\left[\int_{t_0}^{\theta^{\nu^\star,\widetilde\P}} e^{R_P Y_{r}^{t_0,0, \nu^\star}} v(r,  B_{r}) \left( dK_r^\star- k^{\star,n}_r dr  \right) \right]\geq -\delta. $$
Using \eqref{superosl:contrary}, we finally get for $n\geq n_0$
\begin{align*}
0&\geq  \E^{\widetilde\P}\left[v(\theta^{\nu^\star,\widetilde\P}, B_{\theta^{\nu^\star,\widetilde\P}})e^{R_P Y_{\theta^{\nu^\star,\widetilde\P}}^{t_0,0, \nu^\star}}-v(t_0,x_0) \right]\geq   \delta>0,
\end{align*}
which provides the desired contradiction. Thus $f$ is a super--solution of \eqref{eq:hjb2}.
\end{proof}

\begin{Corollary}
Let Assumptions \ref{assumption:bornePrincipal}, $(\mathbf{M})$, $(\mathbf{PPD})$, $(\mathbf{A})$, $(\mathbf{C})$ hold. Then, 
\begin{equation}\label{main}U_0^P=\underline U_0^P\end{equation}
\end{Corollary}

\begin{proof} From Proposition \ref{prop:supersol} together with Assumption $(\mathbf{C})$, we deduce that
$$ U_0^P=-e^{R_PR_0}v(0,0)\leq -e^{R_PR_0}\psi(0,0)=\underline U_0^P.$$
The other inequality being clear by definition, this concludes the proof.
 \end{proof}

\subsubsection{Application to the non-learning model}

In this section, we concentrate on the "non--learning" model. Instead of following the general approach outlined above (which works also in this simple setting), we follow an alternative route and compute directly and explicitly the value function of the Principal. By the above calculations, we immediately have that
$$
\underline{U}_0^P=\underset{ (Z,\Gamma)\in\mathfrak U}{\sup}\ \underset{\P\in\Pc_P^{a^\star (Z_\cdot)}}{\inf}\E^\P\left[-\mathcal E\left(-R_P\int_0^T\widehat\alpha_s^{1/2}(1-Z_s)dW_s^{a^\star (Z_\cdot)}\right)e^{R_P\left( R_0-\int_0^TH(\widehat\alpha_s,Z_s,\Gamma_s)ds\right)}\right],$$
where
\begin{equation}\label{def:F}H(\alpha,z,\gamma):=a^\star (z)-k(a^\star (z))-\frac{\alpha}{2}\left(R_Az^2+R_P(1-z)^2\right)-\frac12 \alpha \gamma +\underset{\alpha \in [\underline{\alpha}_A, \overline{\alpha}^A] }{\inf}\left\{\frac12 \alpha\gamma\right\}.
\end{equation}

In order to pursue the computations, we need to specify a form for the cost function $k$. Namely, we will assume in what follows that
\begin{Assumption}\label{assumption:k}
 The cost function of the Agent is quadratic, defined, for some $k>0$, by
 $$k(a):=k\frac{a^2}2, \; a\geq 0.$$
 \end{Assumption}
We deduce from Proposition \ref{prop:optimaleffort} that the Agent chooses the control $a^\star (z)= \frac{z}{k}$. Hence, Equality \eqref{def:F} can be rewritten 
\begin{align}\label{eq:F:k} H(\alpha,z,\gamma)&= \frac{z}{k}-\frac{z^2}{2k}-\frac{\alpha}{2}\left(R_Az^2+R_P(1-z)^2\right)-\frac12 \alpha \gamma +\underset{\alpha \in [\underline{\alpha}_A, \overline{\alpha}^A] }{\inf}\left\{\frac12 \alpha\gamma\right\}=:H^z(\alpha,z)+H^\gamma(\alpha, \gamma),
\end{align}
where
$$ H^z(\alpha,z):= \frac{z}{k}-\frac{z^2}{2k}-\frac{\alpha}{2}\left(R_Az^2+R_P(1-z)^2\right),\; H^\gamma(\alpha, \gamma):= -\frac12 \alpha \gamma +\underset{\alpha \in [\underline{\alpha}_A, \overline{\alpha}^A] }{\inf}\left\{\frac12 \alpha\gamma\right\}.$$
Notice that for any $\alpha\geq 0$
\begin{equation}\label{equalityF}
H(\overline\alpha^A,z,0)=H(\alpha,z,-R_Az^2-R_P(1-z)^2).
\end{equation}
The following lemma computes the maximum of the map $(z,\gamma)\longmapsto H(\alpha,z,\gamma)$, depending on the value of $\alpha\in\R_+$.
\begin{Lemma}\label{lemma:F} We distinguish three cases.

\vspace{0.5em}
$(i)$ If $\underline{\alpha}^A\leq\alpha\leq\overline{\alpha}^A $, then $(z,\gamma) \longmapsto H(\alpha,z,\gamma)$ admits a $($global$)$ maximum at 
\begin{equation}\label{def:z*}
z^\star (\alpha):=\frac{1+k\alpha R_P}{1+\alpha k(R_A+R_P)},\ \gamma^\star := 0.
\end{equation}

\vspace{0.5em}
$(ii)$ If $\alpha<\underline{\alpha}^A $, $\gamma \longmapsto H(\alpha,z,\gamma)$ is increasing and attains its maximum at $\gamma^\star = +\infty$, with $H(\alpha,z,\gamma^\star )=+\infty.$

\vspace{0.5em}
$(iii)$ If $\overline\alpha^A< \alpha$, $\gamma \longmapsto H(\alpha,z,\gamma)$ is decreasing and attains its maximum at $\gamma^\star = -\infty,$ with $H(\alpha,z,\gamma^\star )=+\infty.$
\end{Lemma}
\begin{proof}
We have
$$\frac{\partial H}{\partial z}(\alpha,z,\gamma) =\frac{\partial H^z}{\partial z}(\alpha,z) =\frac{1}{k}-\frac{z}{k}-\alpha(R_A z-R_P(1-z)),$$
so that
$$\frac{\partial H}{\partial z}(\alpha,z,\gamma) =0 \Longleftrightarrow z=z^\star (\alpha):=\frac{1+k\alpha R_P}{1+\alpha k(R_A+R_P)}.$$
Since $z\longmapsto H^z(\alpha, z)$ is concave for any $\alpha\geq 0$, we deduce that the maximum of $H^z$ is attained at $z^\star (\alpha)$.

\vspace{0.5em}
Furthermore, for any $\gamma\neq 0$
\begin{align*}
\frac{\partial H}{\partial \gamma}(\alpha,z,\gamma)&=\frac{\partial H^\gamma}{\partial \gamma}(\alpha,\gamma) = \frac12 (\underline{\alpha}^A-\alpha) \mathds{1}_{\gamma>0}+\frac12 (\overline{\alpha}^A-\alpha) \mathds{1}_{\gamma<0}.
\end{align*}
If $\underline{\alpha}^A\leq\alpha\leq\overline{\alpha}^A $, then $(z,\gamma) \longmapsto H(\alpha,z,\gamma)$ admits a global maximum at $(z^\star (\alpha),0)$ which proves $(i)$. Then, $(ii)$ and $(iii)$ are clear.\end{proof}

We can now state the main result of this section, which gives the optimal contracts for the second--best problem, when contracts are restricted to the class $\mathfrak C^{{\rm SB}}$.
\begin{Theorem}\label{thm:SB} Let Assumption \ref{assumption:k} hold. Define for any $\alpha\geq 0$
$$z^\star (\alpha):=\frac{1+k\alpha R_P}{1+\alpha k(R_A+R_P)}.$$
\vspace{0.5em}
$(i)$ If $\underline \alpha^A\leq\overline\alpha^P\leq\overline\alpha^A $, then an admissible optimal contract is given by $\xi^{R_0,z^\star (\overline\alpha^P), 0}.$ In this case,
$$ \underline U_0^P= -\exp\left(-R_P(TH(\overline\alpha^P, z^\star (\overline\alpha^P), 0)-R_0)\right).$$

$(ii)$ If $\underline \alpha^P\leq\overline\alpha^A\leq\overline\alpha^P $, then an admissible optimal contract is given by $\xi^{R_0,z^\star (\overline\alpha^A), \gamma^\star },$
where $ \gamma^\star := -R_A (z^\star (\overline\alpha^A))^2 -R_P(1-z^\star (\overline\alpha^A))^2.$
In this case,
$$ \underline U_0^P= -\exp\left(-R_P(TH(\overline\alpha^A, z^\star (\overline\alpha^A), \gamma^\star )-R_0)\right).$$

$(iii)$ Assume that $\overline\alpha^P<\underline\alpha^A$. Then $ \underline U_0^P=0.$

$(iv)$  Assume that $\overline\alpha^A<\underline\alpha^P$. Then $\underline U_0^P=0.$
\end{Theorem}


We now prove the following result.

\begin{Theorem}
Let Assumption \ref{assumption:k} hold. Then
$$U_0^P=\underline U_0^P.$$
\end{Theorem}

\begin{proof}
First of all, notice that when $\overline\alpha^P<\underline\alpha^A$ or $\overline\alpha^A<\underline\alpha^P$, we have $\overline U_0^P=0$, so that $U_0^P=0$ as well. Let us now assume that $\underline\alpha^A\leq \overline\alpha^P\leq \overline\alpha^A$. Since $K_T\geq 0$, we easily have
\begin{align*}
U_0^P&\leq \underset{\xi\in\Cc^{{\rm SB}}}{\sup}\ \E^{\P^{\overline\alpha^P}_{a^\star (Z^\xi_\cdot)}}\left[-\mathcal E\left(-R_P\int_0^T(\overline\alpha^P)^{\frac12}(1-Z_s^\xi)dW_s^{a^\star (Z^\xi_\cdot)}\right)e^{R_P\left(R_0-\int_0^Tf(Z^\xi_s,\overline\alpha^P)ds\right)}\right].
\end{align*}
Then, we easily have that the map $z\longmapsto f(z,\overline\alpha^P)$ attains its maximum at $z^\star (\overline\alpha^P)$, where it is actually equal to $H(\overline\alpha^P,z^\star (\overline\alpha^P),0)$. Thus
$$U_0^P\leq e^{R_P R_0} e^{-R_PTH(\overline\alpha^P, z^\star (\overline\alpha^P), 0)}=\underline U_0^P,$$
by Theorem \ref{thm:SB}(i).

\vspace{0.5em}
Assume now that $\underline\alpha^P\leq \overline\alpha^A\leq \overline\alpha^P$. Then, with the same arguments
\begin{align*}
U_0^P&\leq \underset{\xi\in\Cc^{{\rm SB}}}{\sup}\ \E^{\P^{\overline\alpha^P}_{a^\star (Z^\xi_\cdot)}}\left[-\mathcal E\left(-R_P\int_0^T(\overline\alpha^A)^{\frac12}(1-Z_s^\xi)dW_s^{a^\star (Z^\xi_\cdot)}\right)e^{R_P\left(R_0-\int_0^Tf(Z^\xi_s,\overline\alpha^A)ds\right)}\right]\\
&\leq  e^{R_P R_0} e^{-R_PTH(\overline\alpha^A, z^\star (\overline\alpha^A), 0)}=\underline U_0^P,
\end{align*}
by Theorem \ref{thm:SB}(ii).
 \end{proof}
 
 \subsection{Comments}
 The comparison with the case without ambiguity is actually very similar to the first best problem. First, notice that when $\overline\alpha^P\in[\underline\alpha^A,\overline\alpha^A]$, an optimal contract can be chosen to be linear in the terminal value of the output, and it is actually the exact same contract as the optimal one for a Principal who would only believe in a constant volatility process equal to $\overline\alpha^P$. Since the utility of the Principal is then a decreasing function of the volatility, this means that the Principal always gets less utility than in a context without ambiguity. 
 
\vspace{0.5em} 
However, as soon as $\overline\alpha^A\in[\underline\alpha^P,\overline\alpha^P]$, the second--best optimal contract makes use of the quadratic variation of the output and is therefore path--dependent. Besides, as in the first--best case, the Principal may get an higher utility level than in the case without ambiguity.

\vspace{0.5em}
Finally, in the degenerated cases $(iii)$ and $(iv)$ of Theorem \ref{thm:SB}, we have seen that the optimal effort for the Agent is equal to $0$ since $z^*=0$ and $a^*(z^*)=0$, on the contrary to the first--best problem where, in the same case, the optimal level of effort for the Agent, chosen by the Principal to obtained his best utility $0$, was $a_\text{max}$. Hence, to solve the second--best problem, the Agent does not provide any effort and attains his reservation utility. It can be explained by the fact that in the second--best problem, an optimal contract is a Stackelberg equilibrium, where the Principal has to anticipate the reaction of the Agent given an admissible contract, unlike the first--best problem for which the Principal chooses the level of effort for the Agent.
 \section{Possible extensions and comparison with the literature}\label{sec:comp}
 In this section, we examine several potential generalizations of the problem at hand, and we try to explain how to tackle it in each case.
 
 \subsection{More general dynamics}\label{sec:gen}
 The first possible extension would be to consider an output with more general dynamics. Typically, one could have a general non--Markovian model where
 $$B_t=\int_0^tb_s(B_\cdot,a_s^\P,\alpha_s^\P)ds+\int_0^t\sigma_s(B_\cdot,\alpha_s^\P)dW_s^{a^\P},\ \P-a.s.,$$
 that is to say that the impact of the effort choice of the Agent on the drift of the output is now non--linear, and the value of this drift may also depend on the past values of the output itself, which could model some synergy effects.
 
 \vspace{0.5em}
 Furthermore, the cost function $k$ could also take the form $k_s(B_\cdot,a_s^\P,\alpha_s^\P)$.
 
 \vspace{0.5em}
 In the first--best problem, if the map $b$ actually only depends on $a$, and not on $B$ and $\alpha$, and if $k$ does not depend on $B$, then it is not difficult to see that our approach will still work, albeit with more complicated computations. Notably, the optimal effort of the Agent will either be $a_{\rm max}$ or any (deterministic) minimizer of $a\longmapsto k_s(a)-b_s(a)$ (which exist since $k$ is super--linear). It is however not clear to us how to handle the general dynamics.
 
 \vspace{0.5em}
 In the second--best problem, the representation of the value function of the Agent in terms of 2BSDEs will always work, provided that one can indeed check that it is well--posed (which requires obviously some assumptions on $k$ and $b$). Then, in the Markovian case, our approach depicted above using the Hamilton--Jacobi--Bellman--Isaacs (HJBI) equation 
\begin{align*}
\begin{cases}
\displaystyle-\partial_tu(t,x,y)-\underset{\alpha\in\R_+}{\inf}\underset{(z,\gamma)\in\R^2}{\sup}\Bigg\{b_t(x,a_t^\star(x,z,\alpha),\alpha)\partial_x u(t,x,y)+\bigg(k_t(x,a^\star_t(x,z,\alpha),\alpha)+\frac12\sigma_t^2(x,\alpha)(\gamma+R_Az^2)\\
\displaystyle\hspace{11.8em}-\underset{\tilde\alpha \in \R_+ }{\inf}\bigg\{\frac12 \sigma_t^2(x,\tilde\alpha)\gamma  +{\bf 1}_{\sigma_t^2(x,\tilde\alpha)\in{\bf D}_A(t,x)}\bigg\}\bigg)\partial_yu(t,x,y)+\frac12\sigma_t^2(x,\alpha)\partial_{xx}u(t,x,y)\\
\displaystyle\hspace{11.8em}+\frac12 \sigma_t^2(x,\alpha)(z^2\partial_{yy}u+z\partial_{xy}u)(t,x,y)+{\bf 1}_{\sigma^2_t(x,\alpha)\in{\bf D}_P(t,x)}\Bigg\},\; (t,x,y)\in[0,T)\times\R^2,\\
v(T,x,y)=\mathcal U_P(x-y),\; (x,y)\in\R^2,
\end{cases}
\end{align*}
where $a_t^\star(x,z,\alpha)$ is the unique (for simplicity)\footnote{If the minimizer is not unique, then we assume as usual that the Principal has sufficient bargaining power to make the Agent choose the best minimizer for him. This means that one has also to take the supremum over all minimizers in the Hamiltonian above.} minimizer of the map $a\longmapsto k_t(x,a,\alpha)-b_t(x,a,\alpha)z$. 

\vspace{0.5em}
A particular non--Markovian case could prove very interesting in this framework. It would correspond to the case where the for $\Psi=A,P$
$${\bf D}_\Psi(t,\omega)={\bf D}_\Psi(t,B_t(\omega),\langle B\rangle_t(\omega)),\; (t,\omega)\in[0,T]\times\Omega,$$
which means that the Principal and the Agent update their beliefs according to both the current value of the output and of its quadratic variation.  In this case, we could simply consider $\langle B\rangle$ as an additional state variable, with dynamics
$$\langle B\rangle_t=\int_0^t\sigma^2_s(B_\cdot,\alpha_s^\P)ds,\ \P-a.s.,$$
and have then an HJBI PDE with 3 state variables, thus avoiding to have to rely on the theory of path--dependent PDEs (recall Footnote \ref{foot:ppde}).

\subsection{More general utility functions}
Another possible generalization would be to go beyond the case of exponential utility functions for the Principal and the Agent. As usual, if the utility of the Agent is separable (that is to say if the cost comes out of the utility), then the 2BSDE characterisation of his value function would still hold. The main problem would then be to solve the Principal problem. Once again, one could write down an HJBI equation similar to the one above and try to study it.

\subsection{Comparison with \cite{sung2015optimal}}
As mentioned in the Introduction, independently of our work, Sung \cite{sung2015optimal} has studied a similar model of Principal--Agent with ambiguity. For the sake of understanding the specificities of these two approaches, we will now list what we believe are the main differences.

\begin{itemize}
\item[$(i)$] The modelisation considered in \cite{sung2015optimal} is roughly the same as the one we described in Section \ref{sec:gen} above in terms of the dynamics of the output. As explained, our approach would work similarly in such a context.

\vspace{0.5em}
\item[$(ii)$] There is a first important difference in terms of the ambiguity sets. Indeed, as far as we understand it, \cite{sung2015optimal} considers that the Principal and the Agent share the exact same ambiguity set $D$, which is defined through a map $\pi$ satisfying the KKT conditions or the Slater constraint qualification conditions (see \cite[page 12]{sung2015optimal} for more details). In our work, we do not require any of these conditions, and our modelisation allows for general, and different ambiguity sets for the Principal and the Agent. The only assumptions we impose on these sets are the ones necessary so that the dynamic programming principle holds, which are necessary to ensure that the problem of the Agent is time--consistent, and allows as a consequence to characterise it through 2BSDEs. 

\vspace{0.5em}
\item[$(iii)$] Another important difference lies in the choice of admissible contracts. In \cite{sung2015optimal}, the author takes right from the start as class of admissible contracts the terminal values at time $T$ of some semi--martingales with a triplet of characteristics which is absolutely continuous with respect to the Lebesgue measure, of a form similar to our $\xi^{Y_0,Z,\Gamma}$, see Equation $(11)$ and Theorem $1$ in \cite{sung2015optimal}. As far as we can understand it, \cite{sung2015optimal} justifies this choice through informal arguments, with which we obviously agree, and claims (see \cite[Footnote 14]{sung2015optimal}) that the restriction is without loss of generality, but without any proof. We do not understand this point, since, as we proved it in Proposition \ref{prop:optimaleffort}, the value function of the Agent is always represented through a 2BSDE, in which there is the non--decreasing process $K$. However, it is known that such a process is not always absolutely continuous, for any choice of $\xi$. Indeed, Peng, Song and Zhang \cite{peng2014complete} have characterised completely the set of random variables for which this was the case in the context of our non--learning model, and proved that integrability was not a sufficient condition. Therefore, it is our understanding that it is a result by itself, and actually one of the important contributions of our work, to justify\footnote{One could argue that it suffices to use the same arguments as in \cite{cvitanic2015dynamic} to obtain this result, however their argument does not go through in this case, as we already explained earlier.} that the restriction is without loss of generality, which we explain how to do for general models, and prove under natural assumptions.

\vspace{0.5em}
\item[$(iv)$] The last difference lies in the methods used to solve the problem itself. First of all, the specification of the model is similar in both papers, which had to be expected, since this is the natural way to give the weak formulation for stochastic control problems or differential games. As for the problem of the Agent, since \cite{sung2015optimal} concentrates on a class of contracts similar to our $\mathfrak C^{\rm {\rm SB}}$, the problem becomes a simple verification result, and doesn't have to rely on the 2BSDE theory. We feel that the main difference resides in the approach to the Principal's problem. In \cite[Theorems 3 and 4]{sung2015optimal}, the author characterises the value function of the Principal through some predictable process $Z^P$, for which existence is obtained, but, as far as we understand it, no explicit construction is given. The only exemple where the author manages to compute this $Z^P$ roughly corresponds to our "non--learning" model, see \cite[Proposition 1]{sung2015optimal}. In this regard, our method based on HJBI PDEs (or path--dependent PDEs in the non--Markovian case, see Pham and Zhang \cite{pham2014two}) provides a clear way to compute, at least numerically\footnote{Numerical schemes in the Markovian case are by now extremely well--known, schemes for PPDEs have recently been considered by Zhang and Zhuo \cite{zhang2014monotone}, and Ren and Tan \cite{ren2015convergence}.}, both the value function of the Principal and the associated optimal contract, even in very general models. We believe that this is of the utmost importance for the practical application of this theory.

\vspace{0.5em}
\item[$(iv)$] Finally, we would like to point out that we believe that \cite{sung2015optimal} may be more finance oriented, and as a result provides arguments and explanations which may prove more accessible for a less technical audience. We therefore believe that the two papers complement each other very well,  and \cite{sung2015optimal} is an excellent companion to our work, especially in terms of studying the managerial implications of the results that we have both obtained, which are described at length in \cite{sung2015optimal}.
\end{itemize}
 \begin{appendix}
\section{Appendix}

\begin{proof}[Proof of Lemma \ref{lemmaF}] Let $a\in \mathcal A^{\text{det}}$ and $\xi\equiv (z,\gamma,\delta)\in \mathcal Q$.
We compute on the one hand
\begin{align*}
&\underset{\P\in\mathcal P^a_P}{\inf}\mathbb E^\P\left[\mathcal U_P\left(B_T-\xi\right)\right]=-e^{R_P (\delta-(1-z)\int_0^T a_s ds)}\underset{\P\in\mathcal P^a_P}{\sup} \mathbb E^\P\left[\mathcal E\left( R_P (z-1)\int_0^T (\alpha_s^\mathbb P)^\frac12 dW_s^a\right) e^{\frac{R_P}{2}\left(R_P(1-z)^2+\gamma\right)\int_0^T \alpha_s^\mathbb P ds}\right].
\end{align*}
Hence, using the fact that the stochastic exponential appearing above is a true martingale under any $\P\in\Pc^a_P$, we deduce easily that
\begin{equation}\label{caseP}
\underset{\P\in\mathcal P^a_P}{\inf}\mathbb E^\P\left[\mathcal U_P\left(B_T-\xi\right)\right]=
\begin{cases}
\displaystyle \Gamma_P(a,z,\gamma,\delta,\underline\alpha^P), \text{ if } \gamma<-R_P(1-z)^2\\
\displaystyle  \Gamma_P(a,z,\gamma,\delta,\overline\alpha^P), \text{ if } \gamma> -R_P(1-z)^2\\
\displaystyle \mathbb E^\P\left[\mathcal U_P\left(B_T-\xi\right)\right], \; \forall \mathbb P\in \mathcal P_P^a \text{ if }  \gamma= -R_P(1-z)^2.
\end{cases}
\end{equation} 

We compute on the other hand
\begin{align*}
&\underset{\P\in\mathcal P^a_A}{\inf}\mathbb E^\P\left[\mathcal U_A\left(\xi-\int_0^T k(a_s) ds\right)\right]=\underset{\P\in\mathcal P^a_A}{\inf}\mathbb E^\P\left[-e^{-R_A\left(zB_T +\frac\gamma 2 \int_0^T \alpha_s^\mathbb P ds +\delta-\int_0^T k(a_s) ds\right)}\right]\\
&=-e^{R_A (\int_0^T k(a_s) ds-\delta-z\int_0^T a_s ds)} \underset{\P\in\mathcal P^a_P}{\sup} \mathbb E^\P\left[\mathcal E\left( -R_A z\int_0^T (\alpha_s^\mathbb P)^\frac12 dW_s^a\right) e^{R_A\left(\frac{R_Az^2}{2}-\frac{\gamma}{2}\right)\int_0^T \alpha_s^\mathbb P ds}\right].
\end{align*}
Hence,
\begin{equation}\label{caseA}
\underset{\P\in\mathcal P^a_A}{\inf}\mathbb E^\P\left[\mathcal U_A\left(\xi-\int_0^Tk(a_s)ds\right)\right]=
\begin{cases}
\displaystyle \Gamma_A(a,z,\gamma,\delta,\overline\alpha^A), \text{ if } \gamma<R_A z^2,\\[0.5em]
\displaystyle  \Gamma_A(a,z,\gamma,\delta,\underline\alpha^A), \text{ if } \gamma>R_A z^2,\\[0.5em]
\displaystyle \mathbb E^\P\left[\mathcal U_A\left(\xi-\int_0^T k(a_s) ds\right)\right], \; \forall \mathbb P\in \mathcal P_A^a \text{ if }  \gamma= R_A z^2.
\end{cases}
\end{equation} 
By combining \eqref{caseP} and \eqref{caseA} and using the definition \eqref{defF}, we conclude the proof of the lemma.
\end{proof}

\begin{proof}[Proof of Lemma \ref{lemmadelta*}] Let $(a,z,\gamma,\alpha_P,\alpha_A)\in \mathcal A\times\mathbb R\times \mathbb R\times [\underline\alpha^P,\overline\alpha^P]\times[\underline\alpha^A,\overline\alpha^A]$. First notice that the map $\delta\longmapsto F(a,z,\gamma,\delta,\alpha_P,\alpha_A)$ is clearly concave (we remind the reader that $\rho>0$). Using the first order condition for $\delta$, we obtain after some calculations
\begin{align*}
 \delta=&\ \frac{1}{R_A+R_P}\left[ \log\left(\rho\frac{R_A}{R_P}\right) +(R_P(1-z)-R_Az)\int_0^T a_s ds  +R_A\int_0^T k(a_s) ds\right.\\
&\left. -\frac{R_P}{2}(R_P(1-z)^2+\gamma)\alpha_P T+\frac{R_A}{2}\left(R_A z^2-\gamma\right)\alpha_A T\right],
\end{align*}
which ends the proof.
\end{proof}


\begin{proof}[Proof of Lemma \ref{lemma:partitionQ}]
$(i)$ From Lemma \ref{lemmaF}(i) together with Lemma \ref{lemmadelta*}, we have 
\begin{align*}
&\underset{a \in \mathcal A_{\text{det}}}{\sup}\ \underset{\xi \in \mathcal Q^{\underline\gamma}}{\sup}\ \underset{\P\in\mathcal P^a_P}{\inf}\mathbb E^\P\left[\mathcal U_P\left(B_T-\xi\right)\right]+\rho\underset{\P\in\mathcal P^a_A}{\inf}\mathbb E^\P\left[\mathcal U_A\left(\xi-\int_0^Tk(a_s)ds\right)\right]\\
&=\underset{a \in \mathcal A_{\text{det}}}{\sup}\ \underset{z \in \mathbb R}{\sup}\ \underset{\gamma<-R_P(1-z)^2 }{\sup}\  F(a,z,\gamma,\delta^\star (z,\gamma,\underline\alpha^P,\overline\alpha^A), \underline\alpha^P,\overline\alpha^A),
\end{align*}
where 
\begin{align*}
\nonumber \delta^\star (z,\gamma,\underline\alpha^P,\overline\alpha^A)&:= \frac{1}{R_A+R_P}\left[ \log\left(\rho\frac{R_A}{R_P}\right) +\int_0^T\left((R_P(1-z)-R_Az) a_s +R_A k(a_s)\right) ds\right.\\
&\hspace{2.5cm}\left. -\frac{R_P}{2}(R_P(1-z)^2+\gamma)\underline\alpha^P T +\frac{R_A}{2}\left(R_A z^2-\gamma\right)\overline\alpha^A T\right],
\end{align*}and where we recall that then
\begin{align*}
F(a,z,\gamma,\delta^\star (z,\gamma,\underline\alpha^P,\overline\alpha^A), \underline\alpha^P,\overline\alpha^A)=& -\rho^{\frac{R_P}{R_A+R_P}}\frac{R_A+R_P}{R_P}\left(\frac{R_A}{R_P}\right)^{-\frac{R_A}{R_A+R_P}}e^{\frac{R_AR_P}{R_A+R_P}\left( \int_0^T (k(a_s)-a_s) ds+\frac{\gamma}{2} T(\underline\alpha^P-\overline\alpha^A)\right)}\\
&\hspace{0.9em}\times e^{\frac{R_AR_P}{R_A+R_P}\frac T2\left(\underline\alpha^P R_P(1-z)^2+\overline\alpha^A R_A z^2\right)}.
 \end{align*}
\hspace{1em}$a)$ Assume that $\underline\alpha^P<\overline\alpha^A$. Then $\gamma\longmapsto F(a,z,\gamma,\delta^\star (z,\gamma,\underline\alpha^P,\overline\alpha^A), \underline\alpha^P,\overline\alpha^A)$ is increasing for $\gamma<-R_P(1-z)^2$ and is thus maximal at $\gamma^\star (z):=-R_P(1-z)^2$. Hence, by setting
\begin{align*}\delta^\star (z)&:=  \frac{1}{R_A+R_P}\left[ \log\left(\rho\frac{R_A}{R_P}\right) +(R_P(1-z)-R_Az)\int_0^T a_s ds +R_A\int_0^T k(a_s) ds\right.\\
&\hspace{2.5cm}\left. +\frac{R_AT}{2}\left(R_Az^2+R_P(1-z)^2\right)\overline\alpha^A \right],
\end{align*}
we have
\begin{align*}
&\underset{a \in \mathcal A_{\text{det}}}{\sup}\ \underset{z \in \mathbb R}{\sup}\ \underset{\gamma<-R_P(1-z)^2 }{\sup}\  F(a,z,\gamma,\delta^\star (z,\gamma,\underline\alpha^P,\overline\alpha^A), \underline\alpha^P,\overline\alpha^A)=\underset{a \in \mathcal A_{\text{det}}}{\sup}\ \underset{z \in \mathbb R}{\sup}\ F(a,z,\gamma^\star (z),\delta^\star (z), \underline\alpha^P,\overline\alpha^A),
\end{align*}
with 
\begin{align*}
&F(a,z,\gamma^\star (z),\delta^\star (z), \underline\alpha^P,\overline\alpha^A)= -\rho^{\frac{R_P}{R_A+R_P}}\frac{R_A+R_P}{R_P}\left(\frac{R_A}{R_P}\right)^{-\frac{R_A}{R_A+R_P}} e^{\frac{R_AR_P}{R_A+R_P}\left( \int_0^T (k(a_s)-a_s) ds+\frac T2\overline\alpha^A\left(R_P(1-z)^2+ R_A z^2\right)\right)}.
\end{align*}
Hence, by choosing $z^\star :=\frac{R_P}{R_A+R_P}$, $a^\star $ the constant minimiser of $k(a)-a$, $\gamma^\star =-R_P(1-z^\star )^2$ and 
\begin{align*}
\delta^\star &:=\delta^\star (z^\star )=  \frac{1}{R_A+R_P}\left[ \log\left(\rho\frac{R_A}{R_P}\right) +R_AT k(a^\star ) +\frac{R_A^2R_PT}{2(R_A+R_P)}\overline\alpha^A\right],
\end{align*}
we have
\begin{align*}
&\underset{a \in \mathcal A_{\text{det}}}{\sup}\ \underset{\xi \in \mathcal Q^{\underline\gamma}}{\sup}\ \widetilde u_0^{P,FB}(a,\xi)=F(a^\star ,z^\star ,\gamma^\star ,\delta^\star , \underline\alpha^P,\overline\alpha^A)\\
&=-\rho^{\frac{R_P}{R_A+R_P}}\frac{R_A+R_P}{R_P}\left(\frac{R_A}{R_P}\right)^{-\frac{R_A}{R_A+R_P}}\exp\left(\frac{R_AR_P}{R_A+R_P}T(k(a^\star )-a^\star )+\frac T2\frac{R_A^2R_P^2}{(R_A+R_P)^2}\overline\alpha^A\right).
\end{align*}

\vspace{0.5em}
\hspace{1em}$b)$ Assume that $\underline\alpha^P=\overline\alpha^A=:\tilde\alpha$. Then $\gamma\longmapsto F(a,z,\gamma,\delta^\star (z,\gamma,\tilde\alpha,\tilde\alpha), \tilde\alpha,\tilde\alpha)$ is constant for $\gamma<-R_P(1-z)^2$. Hence for any $\gamma<-R_P(1-z)^2$
\begin{align*}
&\underset{a \in \mathcal A_{\text{det}}}{\sup}\ \underset{z \in \mathbb R}{\sup}\ \underset{\gamma<-R_P(1-z)^2 }{\sup}\  F(a,z,\gamma,\delta^\star (z,\gamma,\underline\alpha^P,\overline\alpha^A), \underline\alpha^P,\overline\alpha^A)=\underset{a \in \mathcal A_{\text{det}}}{\sup}\ \underset{z \in \mathbb R}{\sup}\ F(a,z,\gamma,\delta^\star (z,\gamma,\tilde\alpha,\tilde\alpha), \tilde\alpha),
\end{align*}
where 
\begin{align*}
\nonumber \delta^\star (z,\gamma,\tilde \alpha)&:= \frac{1}{R_A+R_P}\left[ \log\left(\rho\frac{R_A}{R_P}\right) +(R_P(1-z)-R_Az)\int_0^T a_s ds +R_A\int_0^T k(a_s) ds\right.\\
&\hspace{2.5cm}\left. (R_A^2z^2-R_P^2(1-z)^2)\tilde\alpha \frac{T}{2}\right]-\frac{\gamma}{2}\tilde\alpha T,
\end{align*}
 and
 \begin{align*} &F(a,z,\gamma,\delta^\star (z,\gamma), \tilde\alpha,\tilde\alpha)= -\rho^{\frac{R_P}{R_A+R_P}}\frac{R_A+R_P}{R_P}\left(\frac{R_A}{R_P}\right)^{-\frac{R_A}{R_A+R_P}} e^{\frac{R_AR_P}{R_A+R_P}\left( \int_0^T (k(a_s)-a_s) ds+\frac T2\tilde\alpha\left( R_P(1-z)^2+R_A z^2\right)\right)}.
\end{align*}
Hence, by choosing $z^\star :=\frac{R_P}{R_A+R_P}$, $a^\star $ the constant minimiser of $k(a)-a$, for $\gamma^\star $ any value in $(-\infty,-R_P(1-z^\star )^2)$ and 
\begin{align*}
\delta^\star &:=\delta^\star (z^\star ,\gamma^\star )=  \frac{1}{R_A+R_P}\left[ \log\left(\rho\frac{R_A}{R_P}\right) +R_AT k(a^\star )\right] -\frac{\gamma^\star }{2}\tilde\alpha T,
\end{align*}
we have
\begin{align*}
&\underset{a \in \mathcal A_{\text{det}}}{\sup}\ \underset{\xi \in \mathcal Q^{\underline\gamma}}{\sup}\ \widetilde u_0^{P,FB}(a,\xi)=F(a^\star ,z^\star ,\gamma^\star ,\delta^\star , \tilde\alpha,\tilde\alpha)\\
&=-\rho^{\frac{R_P}{R_A+R_P}}\frac{R_A+R_P}{R_P}\left(\frac{R_A}{R_P}\right)^{-\frac{R_A}{R_A+R_P}}\exp\left(\frac{R_AR_PT}{R_A+R_P}\left(k(a^\star )-a^\star +\frac{R_AR_P}{2(R_A+R_P)}\tilde\alpha\right)\right).
\end{align*}

$(ii)$ From Lemma \ref{lemmaF}$(ii)$$a)$, together with Lemma \ref{lemmadelta*}, we have for any $\alpha_P\in [\underline\alpha^P, \overline\alpha^P]$
\begin{align*}
&\underset{a \in \mathcal A_{\text{det}}}{\sup}\ \underset{\xi \in \mathcal Q^{d}}{\sup}\ \widetilde u_0^{P,FB}(a,\xi)=\underset{a \in \mathcal A_{\text{det}}}{\sup}\ \underset{z \in \mathbb R}{\sup}\   F(a,z,\gamma^\star ,\delta^\star (z,\gamma^\star ,\overline\alpha^A), \alpha_P,\overline\alpha^A),
\end{align*}
where $\gamma^\star =-R_P(1-z)^2$, and
\begin{align*}
\nonumber \delta^\star (z,\gamma^\star ,\overline\alpha^A)&:= \frac{1}{R_A+R_P}\left[ \log\left(\rho\frac{R_A}{R_P}\right) +(R_P(1-z)-R_Az)\int_0^T a_s ds +R_A\int_0^T k(a_s) ds\right.\\
&\hspace{2.5cm}\left. +\frac{R_A}{2}\left(R_A z^2+R_P(1-z)^2\right)\overline\alpha^A T\right],
\end{align*}
with also
\begin{align*}
&F(a,z,\gamma^\star ,\delta^\star (z,\gamma,\overline\alpha^A),\alpha_P,\overline\alpha^A)=-\frac{R_A+R_P}{R_P}\left(\frac{\rho^{\frac{R_P}{R_A}}R_P}{R_A}\right)^{\frac{R_A}{R_A+R_P}}e^{\frac{R_AR_P}{R_A+R_P}\left( \int_0^T (k(a_s)-a_s) ds+\frac T2\overline\alpha^A\left( R_P(1-z)^2+ R_A z^2\right)\right)},
 \end{align*} which does not depend on $\alpha_P$. Hence, by choosing $z^\star :=\frac{R_P}{R_A+R_P}$, $a^\star $ the constant minimiser of $k(a)-a$, $\gamma^\star =-R_P(1-z^\star )^2$ and 
\begin{align*}
\delta^\star &:=  \frac{1}{R_A+R_P}\left[ \log\left(\rho\frac{R_A}{R_P}\right) +R_AT k(a^\star )+\frac{R_A^2R_PT}{2(R_A+R_P)}\overline\alpha^A\right],
\end{align*}
we have
\begin{align*}
&\underset{a \in \mathcal A_{\text{det}}}{\sup}\ \underset{\xi \in \mathcal Q^{d}}{\sup}\ \widetilde u_0^{P,FB}(a,\xi)=F(a^\star ,z^\star ,\gamma^\star ,\delta^\star , \alpha_P,\overline\alpha^A)\\
&=-\rho^{\frac{R_P}{R_A+R_P}}\frac{R_A+R_P}{R_P}\left(\frac{R_A}{R_P}\right)^{-\frac{R_A}{R_A+R_P}}\exp\left(\frac{R_AR_PT}{R_A+R_P}\left(k(a^\star )-a^\star +\frac{R_AR_P}{2(R_A+R_P)}\overline\alpha^A\right)\right).
\end{align*}
$(iii)$ From Lemma \ref{lemmaF}$(ii)b)$ together with Lemma \ref{lemmadelta*}, we have 
\begin{align*}
&\underset{a \in \mathcal A_{\text{det}}}{\sup}\ \underset{\xi \in \mathcal Q^{|\gamma|}}{\sup}\ \widetilde u_0^{P,FB}(a,\xi)=\underset{a \in \mathcal A_{\text{det}}}{\sup}\ \underset{z \in \mathbb R}{\sup}\ \underset{-R_P(1-z)^2<\gamma<R_A z^2 }{\sup}\  F(a,z,\gamma,\delta^\star (z,\gamma,\overline\alpha^P,\overline\alpha^A), \overline\alpha^P,\overline\alpha^A),
\end{align*}
where 
\begin{align*}
\nonumber \delta^\star (z,\gamma,\overline\alpha^P,\overline\alpha^A)&:= \frac{1}{R_A+R_P}\left[ \log\left(\rho\frac{R_A}{R_P}\right) +\int_0^T\left((R_P(1-z)-R_Az)a_s+R_A k(a_s)\right) ds\right.\\
&\hspace{2.5cm}\left. -\frac{R_P}{2}(R_P(1-z)^2+\gamma)\overline\alpha^P T +\frac{R_A}{2}\left(R_A z^2-\gamma\right)\overline\alpha^A T\right],
\end{align*}and with
\begin{align*}
F(a,z,\gamma,\delta^\star (z,\gamma,\overline\alpha^P,\overline\alpha^A), \overline\alpha^P,\overline\alpha^A)=&-\rho^{\frac{R_P}{R_A+R_P}}\frac{R_A+R_P}{R_P}\left(\frac{R_A}{R_P}\right)^{-\frac{R_A}{R_A+R_P}} e^{\frac{R_AR_P}{R_A+R_P}\left( \int_0^T (k(a_s)-a_s) ds+\frac{\gamma}{2} T(\overline\alpha^P-\overline\alpha^A)\right)}\\
&\hspace{0.9em}\times e^{\frac{R_AR_P}{R_A+R_P}\frac T2\left(\overline\alpha^P R_P(1-z)^2+\overline\alpha^A R_A z^2\right)}.
 \end{align*}
\hspace{1em} $a)$ Assume that $\overline\alpha^P<\overline\alpha^A$. Then $\gamma\longmapsto F(a,z,\gamma,\delta^\star (z,\gamma,\overline\alpha^P,\overline\alpha^A), \overline\alpha^P,\overline\alpha^A)$ is increasing for $-R_P(1-z)^2<\gamma<R_Az^2$ and is maximal at $\gamma^\star (z):=R_A z^2$. Hence, by setting
\begin{align*}\delta^\star (z)&:=  \frac{1}{R_A+R_P}\left[ \log\left(\rho\frac{R_A}{R_P}\right) +(R_P(1-z)-R_Az)\int_0^T a_s ds +R_A\int_0^T k(a_s) ds\right.\\
&\hspace{2.5cm}\left. -\frac{R_PT}{2}\left(R_Az^2+R_P(1-z)^2\right)\overline\alpha^P \right],
\end{align*}
we have
\begin{align*}
&\underset{a \in \mathcal A_{\text{det}}}{\sup}\ \underset{z \in \mathbb R}{\sup}\ \underset{-R_P(1-z)^2<\gamma<R_A z^2 }{\sup}\  F(a,z,\gamma,\delta^\star (z,\gamma,\overline\alpha^P,\overline\alpha^A), \overline\alpha^P,\overline\alpha^A)=\underset{a \in \mathcal A_{\text{det}}}{\sup}\ \underset{z \in \mathbb R}{\sup}\ F(a,z,\gamma^\star (z),\delta^\star (z), \overline\alpha^P,\overline\alpha^A),
\end{align*}
with 
\begin{align*}
&F(a,z,\gamma^\star (z),\delta^\star (z), \overline\alpha^P,\overline\alpha^A)=-\rho^{\frac{R_P}{R_A+R_P}}\frac{R_A+R_P}{R_P}\left(\frac{R_A}{R_P}\right)^{-\frac{R_A}{R_A+R_P}}e^{\frac{R_AR_P}{R_A+R_P}\left( \int_0^T (k(a_s)-a_s) ds+\frac T2\overline\alpha^P\left(R_P(1-z)^2+ R_A z^2\right)\right)}.
\end{align*}
Hence, by choosing $z^\star :=\frac{R_P}{R_A+R_P}$, $a^\star $ the constant minimiser of $k(a)-a$, $\gamma^\star =R_A|z^\star |^2$ and 
\begin{align*}
\delta^\star &:=\delta^\star (z^\star )=  \frac{1}{R_A+R_P}\left[ \log\left(\rho\frac{R_A}{R_P}\right) +R_AT k(a^\star ) -\frac{R_AR_P^2T}{2(R_A+R_P)}\overline\alpha^P\right],
\end{align*}
we have
\begin{align*}
&\underset{a \in \mathcal A_{\text{det}}}{\sup}\ \underset{\xi \in \mathcal Q^{|\gamma|}}{\sup}\ \widetilde u_0^{P,FB}(a,\xi)=F(a^\star ,z^\star ,\gamma^\star ,\delta^\star , \overline\alpha^P,\overline\alpha^A)\\
&=-\rho^{\frac{R_P}{R_A+R_P}}\frac{R_A+R_P}{R_P}\left(\frac{R_A}{R_P}\right)^{-\frac{R_A}{R_A+R_P}}\exp\left(\frac{R_AR_P}{R_A+R_P}T(k(a^\star )-a^\star )+\frac T2\frac{R_A^2R_P^2}{(R_A+R_P)^2}\overline\alpha^P\right).
\end{align*}

\vspace{0.5em}
\hspace{1em} $b)$ Assume that $\overline\alpha^P=\overline\alpha^A=:\overline\alpha$. Then $\gamma\longmapsto F(a,z,\gamma,\delta^\star (z,\gamma,\overline\alpha,\overline\alpha), \overline\alpha,\overline\alpha)$ is constant for $-R_P(1-z)^\star <\gamma<R_Az^2$. Hence for any $\gamma\in(-R_P(1-z)^2, R_A z^2)$
\begin{align*}
&\underset{a \in \mathcal A_{\text{det}}}{\sup}\ \underset{z \in \mathbb R}{\sup}\ \underset{\gamma\in(-R_P(1-z)^2, R_A z^2) }{\sup}\  F(a,z,\gamma,\delta^\star (z,\gamma,\overline\alpha^P,\overline\alpha^A), \overline\alpha^P,\overline\alpha^A)=\underset{a \in \mathcal A_{\text{det}}}{\sup}\ \underset{z \in \mathbb R}{\sup}\ F(a,z,\gamma,\delta^\star (z,\gamma), \overline\alpha,\overline\alpha),
\end{align*}
where 
\begin{align*}
\nonumber \delta^\star (z,\gamma)&:= \frac{1}{R_A+R_P}\left[ \log\left(\rho\frac{R_A}{R_P}\right) +(R_P(1-z)-R_Az)\int_0^T a_s ds +R_A\int_0^T k(a_s) ds\right.\\
&\hspace{2.5cm}\left. +(R_Az^2-R_P(1-z)^2)\overline\alpha \frac{T}{2}\right]-\frac{\gamma}{2}\overline\alpha T,
\end{align*}
 and with
\begin{align*} 
&F(a,z,\gamma,\delta^\star (z,\gamma), \overline\alpha,\overline\alpha)=-\rho^{\frac{R_P}{R_A+R_P}}\frac{R_A+R_P}{R_P}\left(\frac{R_A}{R_P}\right)^{-\frac{R_A}{R_A+R_P}}e^{\frac{R_AR_P}{R_A+R_P}\left( \int_0^T (k(a_s)-a_s) ds+\frac T2\overline\alpha\left( R_P(1-z)^2+R_A z^2\right)\right)}.
\end{align*}
Hence, by choosing $z^\star :=\frac{R_P}{R_A+R_P}$, $a^\star $ the constant minimiser of $k(a)-a$, any $\gamma^\star \in(-R_P(1-z)^2, R_A z^2)$ and 
\begin{align*}
\delta^\star &:=\delta^\star (z^\star ,\gamma^\star )=  \frac{1}{R_A+R_P}\left[ \log\left(\rho\frac{R_A}{R_P}\right) +R_AT k(a^\star )\right] -\frac{\gamma^\star }{2}\overline\alpha T,
\end{align*}
we have
\begin{align*}
&\underset{a \in \mathcal A_{\text{det}}}{\sup}\ \underset{\xi \in \mathcal Q^{\underline\gamma}}{\sup}\ \widetilde u_0^{P,FB}(a,\xi)=F(a^\star ,z^\star ,\gamma^\star ,\delta^\star , \overline\alpha,\overline\alpha)\\
&=-\rho^{\frac{R_P}{R_A+R_P}}\frac{R_A+R_P}{R_P}\left(\frac{R_A}{R_P}\right)^{-\frac{R_A}{R_A+R_P}}\exp\left(\frac{R_AR_PT}{R_A+R_P}\left((k(a^\star )-a^\star )+\frac{R_AR_P}{2(R_A+R_P)}\overline\alpha\right)\right).
\end{align*}

\vspace{0.5em}
\hspace{1em} $c)$ Assume that $\overline\alpha^P>\overline\alpha^A$. The proof is exactly the same as in the case $\overline\alpha^P<\overline\alpha^A$, and we obtain, with $z^\star =\frac{R_P}{R_A+R_P} $, $\gamma^\star =-R_P(1-z^\star )^2$, and
\begin{align*}\delta^\star &:= \frac{1}{R_A+R_P}\left[ \log\left(\rho\frac{R_A}{R_P}\right) +R_AT k(a^\star ) +\frac{R_A^2R_PT}{2(R_A+R_P)}\overline\alpha^A\right],
\end{align*}
that
\begin{align*}
&\underset{a \in \mathcal A_{\text{det}}}{\sup}\ \underset{\xi \in \mathcal Q^{u}}{\sup}\ \widetilde u_0^{P,FB}(a,\xi)=F(a^\star ,z^\star ,\gamma^\star ,\delta^\star , \overline\alpha^P,\overline\alpha^A)\\
&=-\rho^{\frac{R_P}{R_A+R_P}}\frac{R_A+R_P}{R_P}\left(\frac{R_A}{R_P}\right)^{-\frac{R_A}{R_A+R_P}}\exp\left(\frac{R_AR_P}{R_A+R_P}T(k(a^\star )-a^\star )+\frac T2\frac{R_A^2R_P^2}{(R_A+R_P)^2}\overline\alpha^A\right).
\end{align*}
$(iv)$ The proof is similar to the case $(ii)$. It suffices to change $\overline\alpha^A$ into $\overline\alpha^P$ and choose $\gamma^\star =R_A|z^\star |^2$.

\vspace{0.5em}
$(v)$ From Lemma \ref{lemmaF}$(iii)$ together with Lemma \ref{lemmadelta*}, we have 
\begin{align*}
&\underset{a \in \mathcal A_{\text{det}}}{\sup}\ \underset{\xi \in \mathcal Q^{\overline\gamma}}{\sup}\ \widetilde u_0^{P,FB}(a,\xi)=\underset{a \in \mathcal A_{\text{det}}}{\sup}\ \underset{z \in \mathbb R}{\sup}\ \underset{\gamma>R_Az^2 }{\sup}\  F(a,z,\gamma,\delta^\star (z,\gamma,\overline\alpha^P,\underline\alpha^A), \overline\alpha^P,\underline\alpha^A),
\end{align*}
where 
\begin{align*}
\nonumber \delta^\star (z,\gamma,\overline\alpha^P,\underline\alpha^A)&:= \frac{1}{R_A+R_P}\left[ \log\left(\rho\frac{R_A}{R_P}\right) +\int_0^T\left((R_P(1-z)-R_Az) a_s  +R_A k(a_s)\right) ds\right.\\
&\hspace{2.5cm}\left. -\frac{R_P}{2}(R_P(1-z)^2+\gamma)\overline\alpha^P T +\frac{R_A}{2}\left(R_A z^2-\gamma\right)\underline\alpha^A T\right],
\end{align*}and with
\begin{align*}
F(a,z,\gamma,\delta^\star (z,\gamma,\overline\alpha^P,\underline\alpha^A), \overline\alpha^P,\underline\alpha^A)=&-\rho^{\frac{R_P}{R_A+R_P}}\frac{R_A+R_P}{R_P}\left(\frac{R_A}{R_P}\right)^{-\frac{R_A}{R_A+R_P}}e^{\frac{R_AR_P}{R_A+R_P}\left( \int_0^T (k(a_s)-a_s) ds+\frac{\gamma}{2} T(\overline\alpha^P-\underline\alpha^A)\right)}\\
&\hspace{0.9em}\times e^{\frac{R_AR_P}{R_A+R_P}\frac T2\left(\overline\alpha^P R_P(1-z)^2+\underline\alpha^A R_A z^2\right)}.
 \end{align*}

\vspace{0.5em}
\hspace{1em} $a)$ Assume that $\overline\alpha^P=\underline\alpha^A=:\check\alpha$. Then $\gamma\longmapsto F(a,z,\gamma,\delta^\star (z,\gamma,\check\alpha,\check\alpha),\check\alpha,\check\alpha)$ is constant for $\gamma>R_A z^2$. Hence for any $\gamma>R_A z^2$
\begin{align*}
&\underset{a \in \mathcal A_{\text{det}}}{\sup}\ \underset{z \in \mathbb R}{\sup}\ \underset{\gamma>R_A z^2 }{\sup}\  F(a,z,\gamma,\delta^\star (z,\gamma,\overline\alpha^P,\underline\alpha^A), \overline\alpha^P,\underline\alpha^A)=\underset{a \in \mathcal A_{\text{det}}}{\sup}\ \underset{z \in \mathbb R}{\sup}\ F(a,z,\gamma,\delta^\star (z,\gamma), \check\alpha,\check\alpha),
\end{align*}
where 
\begin{align*}
\nonumber \delta^\star (z,\gamma)&:= \frac{1}{R_A+R_P}\left[ \log\left(\rho\frac{R_A}{R_P}\right) +(R_P(1-z)-R_Az)\int_0^T a_s ds +R_A\int_0^T k(a_s) ds\right.\\
&\hspace{2.5cm}\left. +(R_Az^2-R_P(1-z)^2)\check\alpha \frac{T}{2}\right]-\frac{\gamma}{2}\check\alpha T,
\end{align*}
 and
\begin{align*} 
&F(a,z,\gamma,\delta^\star (z,\gamma), \check\alpha,\check\alpha)=-\rho^{\frac{R_P}{R_A+R_P}}\frac{R_A+R_P}{R_P}\left(\frac{R_A}{R_P}\right)^{-\frac{R_A}{R_A+R_P}}e^{\frac{R_AR_P}{R_A+R_P}\left( \int_0^T (k(a_s)-a_s) ds+\frac T2\check\alpha\left( R_P(1-z)^2+R_A z^2\right)\right)}.
\end{align*}
Hence, by choosing $z^\star :=\frac{R_P}{R_A+R_P}$, $a^\star $ the constant minimiser of $k(a)-a$, any $\gamma^\star >R_A |z^\star |^2$ and 
\begin{align*}
\delta^\star :=\delta^\star (z^\star ,\gamma^\star )
=  \frac{1}{R_A+R_P}\left[ \log\left(\rho\frac{R_A}{R_P}\right) +R_AT k(a^\star ) \right] -\frac{\gamma^\star }{2}\check\alpha T,
\end{align*}
we have
\begin{align*}
&\underset{a \in \mathcal A_{\text{det}}}{\sup}\ \underset{\xi \in \mathcal Q^{\underline\gamma}}{\sup}\ \widetilde u_0^{P,FB}(a,\xi)=F(a^\star ,z^\star ,\gamma^\star ,\delta^\star , \check\alpha,\check\alpha)\\
&=-\rho^{\frac{R_P}{R_A+R_P}}\frac{R_A+R_P}{R_P}\left(\frac{R_A}{R_P}\right)^{-\frac{R_A}{R_A+R_P}}\exp\left(\frac{R_AR_PT}{R_A+R_P}\left((k(a^\star )-a^\star )+\frac{R_AR_P}{2(R_A+R_P)}\check\alpha\right)\right).
\end{align*}

\vspace{0.5em}
\hspace{1em} $b)$ Assume that $\overline\alpha^P>\underline\alpha^A$. Then $\gamma\longmapsto F(a,z,\gamma,\delta^\star (z,\gamma,\underline\alpha^P,\overline\alpha^A), \underline\alpha^P,\overline\alpha^A)$ is decreasing for $\gamma>R_A z^2$ and is maximal at $\gamma^\star (z):=R_A z^2$. Hence, by setting
\begin{align*}\delta^\star (z)&:=  \frac{1}{R_A+R_P}\left[ \log\left(\rho\frac{R_A}{R_P}\right) +(R_P(1-z)-R_Az)\int_0^T a_s ds +R_A\int_0^T k(a_s) ds\right.\\
&\hspace{2.5cm}\left. -\frac{R_PT}{2}\left(R_Az^2+R_P(1-z)^2\right)\overline\alpha^P \right],
\end{align*}
we have
\begin{align*}
&\underset{a \in \mathcal A_{\text{det}}}{\sup}\ \underset{z \in \mathbb R}{\sup}\ \underset{\gamma>R_A z^2 }{\sup}\  F(a,z,\gamma,\delta^\star (z,\gamma,\overline\alpha^P,\underline\alpha^A), \overline\alpha^P,\underline\alpha^A)=\underset{a \in \mathcal A_{\text{det}}}{\sup}\ \underset{z \in \mathbb R}{\sup}\ F(a,z,\gamma^\star (z),\delta^\star (z), \overline\alpha^P,\underline\alpha^A),
\end{align*}
with 
\begin{align*}
&F(a,z,\gamma^\star (z),\delta^\star (z), \overline\alpha^P,\underline\alpha^A)=-\rho^{\frac{R_P}{R_A+R_P}}\frac{R_A+R_P}{R_P}\left(\frac{R_A}{R_P}\right)^{-\frac{R_A}{R_A+R_P}}e^{\frac{R_AR_P}{R_A+R_P}\left( \int_0^T (k(a_s)-a_s) ds+\frac T2\overline\alpha^P\left(R_P(1-z)^2+ R_A z^2\right)\right)}.
\end{align*}
Hence, by choosing $z^\star :=\frac{R_P}{R_A+R_P}$, $a^\star $ the constant minimiser of $k(a)-a$, $\gamma^\star =R_A|z^\star |^2$ and 
\begin{align*}
\delta^\star &:=\delta^\star (z^\star )=  \frac{1}{R_A+R_P}\left[ \log\left(\rho\frac{R_A}{R_P}\right) +R_AT k(a^\star ) -\frac{R_AR_P^2T}{2(R_A+R_P)}\overline\alpha^P\right],
\end{align*}
we have
\begin{align*}
&\underset{a \in \mathcal A_{\text{det}}}{\sup}\ \underset{\xi \in \mathcal Q^{\underline\gamma}}{\sup}\ \widetilde u_0^{P,FB}(a,\xi)=F(a^\star ,z^\star ,\gamma^\star ,\delta^\star , \overline\alpha^P,\underline\alpha^A)\\
&=-\rho^{\frac{R_P}{R_A+R_P}}\frac{R_A+R_P}{R_P}\left(\frac{R_A}{R_P}\right)^{-\frac{R_A}{R_A+R_P}}\exp\left(\frac{R_AR_P}{R_A+R_P}T(k(a^\star )-a^\star )+\frac T2\frac{R_A^2R_P^2}{(R_A+R_P)^2}\overline\alpha^P\right).
\end{align*}

\end{proof}

\begin{proof}[Proof of Lemma \ref{lemma:dxi}]
For any $a\in\Ac$, let us define $\xi_a:=z^\star B_T+\frac{\gamma^\star }{2}\langle B\rangle_T +\delta^\star (a)$ where $\gamma^\star \in \mathbb R$,
$$z^\star =\frac{R_P}{R_A+R_P}, \; \delta^\star (a):=\frac{1}{R_A+R_P}\left(\log\left(\rho\frac{R_A}{R_P}\right)+R_A \int_0^Tk(a_s)ds \right)+\lambda,\; \lambda\in \mathbb R. $$
Then, for any $h\in M^{\phi}$
\begin{align}
\nonumber\widetilde{D}\Xi_{a}^{\alpha_P,\alpha_A}(\xi_a)[h]&= \mathbb E^{\P_0}\Big[R_P h(X^{a,\alpha_P}_\cdot) e^{-R_P\left(\int_0^T a_s(X^{a,\alpha_P}_\cdot) ds +{\alpha_P}^\frac12 B_T-\xi_a(X^{a,\alpha_P}_\cdot)\right)}\\
\nonumber &\hspace{0.9em}-R_A h(X^{a,\alpha_A}_\cdot)\rho e^{-R_A\left(\xi_a(X_\cdot^{a,A})-\int_0^Tk(a_s(X_\cdot^{a,\alpha_A}))ds\right)}\Big]\\
\nonumber &\hspace{0.9em}-R_A h(X^{a,\alpha_A}_\cdot)\rho e^{-R_A\big(\frac{R_P}{R_A+R_P}\alpha_A^{\frac12} B_T+\frac{R_P}{R_A+R_P}\int_0^Ta_s(X^{a,\alpha_A}_\cdot)ds+\frac{\gamma^\star }{2}\alpha_AT+\delta^\star (a)-\int_0^Tk(a_s(X_\cdot^{a,\alpha_A}))ds\big)}\Big]\\
\nonumber &=R_P\left(\rho\frac{R_A}{R_P}\right)^\frac{R_P}{R_A+R_P}\\
\nonumber&\hspace{0.9em} \left(\mathbb E^{\P^{\alpha_P}_0}\Big[ h(X^{a,\alpha_P}_\cdot) e^{-\frac{R_AR_P}{R_A+R_P}\int_0^T(a_s(X_\cdot^{a,\alpha_P})-k(a_s(X_\cdot^{a,\alpha_P})))ds+R_P\left(\frac{\gamma^\star }{2}\alpha_PT+\lambda \right)+\frac{R_P^2R_A^2}{2(R_A+R_P)^2}\alpha_PT}\right.\\
 &\hspace{0.9em}-\left. \mathbb E^{\P^{\alpha_A}_0}\Big[ h(X^{a,\alpha_A}_\cdot) e^{-\frac{R_AR_P}{R_A+R_P}\int_0^T(a_s(X_\cdot^{a,\alpha_A})-k(a_s(X_\cdot^{a,\alpha_A})))ds+R_A\left(\frac{\gamma^\star }{2}\alpha_AT+\lambda \right)+\frac{R_P^2R_A^2}{2(R_A+R_P)^2}\alpha_AT}\right),\label{dxih}
\end{align}
where
 $$\frac{d\P_0^{\alpha_P}}{d\P_0}:=\mathcal{E}\left( -\frac{R_P R_A  (\alpha_P)^\frac12}{R_A+R_P} B_T\right),\ \frac{d\P_0^{\alpha_A}}{d\P_0}:=\mathcal{E}\left( -\frac{R_P R_A (\alpha_A)^\frac12}{R_A+R_P} B_T\right).$$ 
 
 Assume that $\alpha_P=\alpha_A=:\alpha$. Then, if $R_A=R_P$ or if $R_A\neq R_P$ and Property \eqref{propopti} holds then we automatically have
 \begin{align*}
 \widetilde{D}\Xi_{a}^{\alpha_P,\alpha_A}(\xi)[h-\xi_a]= 0,\end{align*}
 which proves the first result.

\end{proof}

\begin{proof}[Proof of Theorem \ref{thm:FB:Q} ]
\textbf{We begin by proving $(i)$.} Assume that $\underline\alpha^A= \overline\alpha^P$.  
First notice that
\begin{align*}
U_0^{P,{\rm FB}}&\leq \underset{\xi\in\mathcal C}{\sup}\ \underset{a\in\mathcal A}{\sup}\left\{\mathbb E^{\P^{\overline\alpha^P}_a}\left[\mathcal U_P\left(B_T-\xi\right)\right]+\rho\mathbb E^{\P^{\underline\alpha^A}_a}\left[\mathcal U_A\left(\xi-\int_0^Tk(a_s)ds\right)\right]\right\}=-\underset{\xi\in\mathcal C}{\inf}\ \underset{a\in\mathcal A}{\inf}\Xi_a^{\overline\alpha^P,\underline\alpha^A}(\xi),
\end{align*}
where we have used the fact that by definition, the law of $B$ under $\P^\alpha_a$ is equal to the law of $X^{a,\alpha}$ under $\P_0$. Let us then define for any $a\in\Ac$
$$\xi_a:=z^\star B_T+\frac{\gamma^\star }{2}\langle B\rangle_T +\delta^\star (a),$$
where $\gamma^\star \in [R_A(z^\star )^2,+\infty)$, and
$$z^\star :=\frac{R_P}{R_A+R_P}, \; \delta^\star (a):=\frac{1}{R_A+R_P}\left(\log\left(\rho\frac{R_A}{R_P}\right)+R_A \int_0^Tk(a_s)ds \right)-\frac{\gamma^\star }{2}\overline\alpha^PT.$$
Then by Lemma \ref{lemma:dxi}, we know that
\begin{align*}
\underset{\xi\in\mathcal C}{\inf}\ \underset{a\in\mathcal A}{\inf}\Xi_a^{\overline\alpha^P,\underline\alpha^A}(\xi)&=\underset{a\in\mathcal A}{\inf}\Xi_a^{\overline\alpha^P,\underline\alpha^A}(\xi_a).
\end{align*}
We then have
\begin{align*}
\Xi_a^{\overline\alpha^P,\underline\alpha^A}(\xi_a)=& \rho^{\frac{R_P}{R_A+R_P}} \left(\frac{R_A}{R_P}\right)^{-\frac{R_A}{R_A+R_P}}\frac{R_P+R_A}{R_P}e^{\overline\alpha^PT\frac{R_A^2R_P^2}{2(R_A+R_P)^2}}\\
&\times\E^{\P_0}\left[\Ec\left(-\frac{R_AR_P}{R_A+R_P}(\overline\alpha^P)^{1/2}B_T\right)e^{\frac{R_AR_P}{R_A+R_P}\int_0^T(k(a_s(X_\cdot^{a,\overline\alpha^P}))-a_s(X_\cdot^{a,\overline\alpha^P}))ds}\right],
\end{align*}
so that we clearly have
$$\underset{a\in\mathcal A}{\inf}\Xi_a^{\overline\alpha^P,\underline\alpha^A}(\xi_a)=\Xi_{a^\star }^{\overline\alpha^P,\underline\alpha^A}(\xi_{a^\star }).$$
Thus we have obtained 
$$U_0^{P,{\rm FB}}\leq  -\Xi_{a^\star }^{\overline\alpha^P,\underline\alpha^A}(\xi_{a^\star }).$$
Conversely, we have from Lemma \ref{lemma:partitionQ} $(iv)$ and $(v)$ $a.$,
$$U_0^{P,{\rm FB}}\geq \underset{\xi\in\overline{\mathcal Q^{\overline\gamma}}}{\sup}\ \underset{a\in\mathcal A_{det}}{\sup} \widetilde u_0^{P,FB}(a,\xi)= G(a^\star ,z^\star ,\overline\alpha^P)=-\Xi_{a^\star }^{\overline\alpha^P,\underline\alpha^A}(\xi_{a^\star }).$$
Therefore
$$U_0^{P,{\rm FB}}=-\Xi_{a^\star }^{\overline\alpha^P,\underline\alpha^A}(\xi_{a^\star }).$$
Finally, it remains to choose $\rho$ so as to satisfy the participation constraint of the Agent. Some calculations show that it suffices to take $\rho$ such that
$$\frac{1}{R_A+R_P}\log\left(\rho \frac{R_A}{R_P}\right)=-\frac{1}{R_A}\log(-R_0)+\frac{R_P}{R_A+R_P}T \left[k(a^\star )-a^\star +\overline\alpha^PT\frac{R_AR_P}{2(R_A+R_P)} \right]. $$
Thus,
$$\delta(a^\star )=Tk(a^\star )-\frac{R_P}{R_A+R_P}Ta^\star +\frac{\overline\alpha^PT}{2}\left(\frac{R_AR_P^2}{(R_A+R_P)^2}-\gamma^\star \right)-\frac{1}{R_A}\log(-R).$$
 \vspace{0.5em}
 
\noindent \textbf{We now turn to $(ii)$.} Since we have $\underline\alpha^A< \overline\alpha^P<\overline\alpha^A$, we deduce that
\begin{align*}
U_0^{P,{\rm FB}}&\leq \underset{\xi\in\mathcal C}{\sup}\ \underset{a\in\mathcal A}{\sup}\left\{\mathbb E^{\P^{\overline\alpha^P}_a}\left[\mathcal U_P\left(B_T-\xi\right)\right]+\rho\mathbb E^{\P^{\overline\alpha^P}_a}\left[\mathcal U_A\left(\xi-\int_0^Tk(a_s)ds\right)\right]\right\}=-\underset{\xi\in\mathcal C}{\inf}\ \underset{a\in\mathcal A}{\inf}\Xi_a^{\overline\alpha^P,\overline\alpha^P}(\xi).
\end{align*}
Let us then define for any $a\in\Ac$
$$\xi_a:=z^\star B_T+\frac{\gamma^\star }{2}\langle B\rangle_T +\delta^\star (a),$$
where $\gamma^\star =R_A(z^\star )^2$, and
$$z^\star :=\frac{R_P}{R_A+R_P}, \; \delta^\star (a):=\frac{1}{R_A+R_P}\left(\log\left(\rho\frac{R_A}{R_P}\right)+R_A \int_0^Tk(a_s)ds \right)-\frac{\gamma^\star }{2}\overline\alpha^PT.$$
Then by Lemma \ref{lemma:dxi}, we know that
\begin{align*}
\underset{\xi\in\mathcal C}{\inf}\ \underset{a\in\mathcal A}{\inf}\Xi_a^{\overline\alpha^P,\overline\alpha^P}(\xi)&=\underset{a\in\mathcal A}{\inf}\Xi_a^{\overline\alpha^P,\overline\alpha^P}(\xi_a).
\end{align*}
We then have
\begin{align*}
\Xi_a^{\overline\alpha^P,\overline\alpha^P}(\xi_a)=& \left(\rho\frac{R_A}{R_P}\right)^{\frac{R_P}{R_A+R_P}}\left(1+\frac{R_P}{R_A}\right)e^{\overline\alpha^PT\frac{R_A^2R_P^2}{2(R_A+R_P)^2}}\\
&\times\E^{\P_0}\left[\Ec\left(-\frac{R_AR_P}{R_A+R_P}(\overline\alpha^P)^{1/2}B_T\right)e^{\frac{R_AR_P}{R_A+R_P}\int_0^T(k(a_s(X_\cdot^{a,\overline\alpha^P}))-a_s(X_\cdot^{a,\overline\alpha^P}))ds}\right],
\end{align*}
so that we clearly have
$$\underset{a\in\mathcal A}{\inf}\Xi_a^{\overline\alpha^P,\overline\alpha^P}(\xi_a)=\Xi_{a^\star }^{\overline\alpha^P,\overline\alpha^P}(\xi_{a^\star }).$$
Thus we have obtained 
$$U_0^{P,{\rm FB}}\leq  -\Xi_{a^\star }^{\overline\alpha^P,\overline\alpha^P}(\xi_{a^\star }).$$
Conversely, using Lemma \ref{lemma:partitionQ}$(iv)$ we have
$$U_0^{P,{\rm FB}}\geq \underset{a \in \mathcal A_{\text{det}}}{\sup}\ \underset{\xi \in \mathcal Q^{u}}{\sup}\ \widetilde u_0^{P,{\rm FB}}(a,\xi)=G(a^\star ,z^\star ,\overline\alpha^P)=-\Xi_{a^\star }^{\overline\alpha^P,\overline\alpha^P}(\xi_{a^\star }).$$
Therefore
$$U_0^{P,{\rm FB}}=-\Xi_{a^\star }^{\overline\alpha^P,\overline\alpha^P}(\xi_{a^\star }).$$
Finally, it remains to choose $\rho$ so as to satisfy the participation constraint of the Agent. Some calculations show that it suffices to take $\rho$ such that
$$\frac{1}{R_A+R_P}\log\left(\rho \frac{R_A}{R_P}\right)=-\frac{1}{R_A}\log(-R)+\frac{R_P}{R_A+R_P}T \left[k(a^\star )-a^\star +\overline\alpha^PT\frac{R_AR_P}{2(R_A+R_P)} \right]. $$
Thus,
\begin{align*}
\delta(a^\star )&=Tk(a^\star )-\frac{R_P}{R_A+R_P}Ta^\star +\frac{\overline\alpha^PT}{2}\left(\frac{R_AR_P^2}{(R_A+R_P)^2}-\gamma^\star \right)-\frac{1}{R_A}\log(-R)\\
&= Tk(a^\star )-\frac{R_P}{R_A+R_P}Ta^\star -\frac{1}{R_A}\log(-R).
\end{align*}

\textbf{We now turn to $(v)$}. Since we have $\underline\alpha^P< \overline\alpha^A<\overline\alpha^P$, we deduce that
\begin{align*}
U_0^{P,{\rm FB}}&\leq \underset{\xi\in\mathcal C}{\sup}\ \underset{a\in\mathcal A}{\sup}\left\{\mathbb E^{\P^{\overline\alpha^A}_a}\left[\mathcal U_P\left(B_T-\xi\right)\right]+\rho\mathbb E^{\P^{\overline\alpha^A}_a}\left[\mathcal U_A\left(\xi-\int_0^Tk(a_s)ds\right)\right]\right\}=-\underset{\xi\in\mathcal C}{\inf}\ \underset{a\in\mathcal A}{\inf}\Xi_a^{\overline\alpha^A,\overline\alpha^A}(\xi).
\end{align*}
Let us then define for any $a\in\Ac$
$$\xi_a:=z^\star B_T+\frac{\gamma^\star }{2}\langle B\rangle_T +\delta^\star (a),$$
where $\gamma^\star =-R_P|1-z^\star |^2$, and
$$z^\star :=\frac{R_P}{R_A+R_P}, \; \delta^\star (a):=\frac{1}{R_A+R_P}\left(\log\left(\rho\frac{R_A}{R_P}\right)+R_A \int_0^Tk(a_s)ds \right)-\frac{\gamma^\star }{2}\overline\alpha^AT.$$
Then by Lemma \ref{lemma:dxi}, we know that
\begin{align*}
\underset{\xi\in\mathcal C}{\inf}\ \underset{a\in\mathcal A}{\inf}\Xi_a^{\overline\alpha^A,\overline\alpha^A}(\xi)&=\underset{a\in\mathcal A}{\inf}\Xi_a^{\overline\alpha^A,\overline\alpha^A}(\xi_a).
\end{align*}
We then have
\begin{align*}
\Xi_a^{\overline\alpha^A,\overline\alpha^A}(\xi_a)=& \left(\rho\frac{R_A}{R_P}\right)^{\frac{R_P}{R_A+R_P}}\left(1+\frac{R_P}{R_A}\right)e^{\overline\alpha^AT\frac{R_A^2R_P^2}{2(R_A+R_P)^2}}\\
&\times\E^{\P_0}\left[\Ec\left(-\frac{R_AR_P}{R_A+R_P}(\overline\alpha^A)^{1/2}B_T\right)e^{\frac{R_AR_P}{R_A+R_P}\int_0^T(k(a_s(X_\cdot^{a,\overline\alpha^A}))-a_s(X_\cdot^{a,\overline\alpha^A}))ds}\right],
\end{align*}
so that we clearly have
$$\underset{a\in\mathcal A}{\inf}\Xi_a^{\overline\alpha^A,\overline\alpha^A}(\xi_a)=\Xi_{a^\star }^{\overline\alpha^A,\overline\alpha^A}(\xi_{a^\star }).$$
Thus we have obtained 
$$U_0^{P,{\rm FB}}\leq  -\Xi_{a^\star }^{\overline\alpha^A,\overline\alpha^A}(\xi_{a^\star }).$$
Conversely, using Lemma \ref{lemma:partitionQ} $(ii)$ we have
$$U_0^{P,FB}\geq \underset{a \in \mathcal A_{\text{det}}}{\sup}\ \underset{\xi \in \mathcal Q^{d}}{\sup}\ \widetilde u_0^{P,{\rm FB}}(a,\xi)=G(a^\star ,z^\star ,\overline\alpha^A)=-\Xi_{a^\star }^{\overline\alpha^A,\overline\alpha^A}(\xi_{a^\star }).$$
Therefore
$$U_0^{P,{\rm FB}}=-\Xi_{a^\star }^{\overline\alpha^A,\overline\alpha^A}(\xi_{a^\star }).$$
Finally, it remains to choose $\rho$ so as to satisfy the participation constraint of the Agent. Some calculations show that it suffices to take $\rho$ such that
$$\frac{1}{R_A+R_P}\log\left(\rho \frac{R_A}{R_P}\right)=-\frac{1}{R_A}\log(-R)+\frac{R_P}{R_A+R_P}T \left[k(a^\star )-a^\star +\overline\alpha^AT\frac{R_AR_P}{2(R_A+R_P)} \right]. $$
Thus,
\begin{align*}
\delta(a^\star )&=Tk(a^\star )-\frac{R_P}{R_A+R_P}Ta^\star +\frac{\overline\alpha^AT}{2}\left(\frac{R_AR_P^2}{(R_A+R_P)^2}-\gamma^\star \right)-\frac{1}{R_A}\log(-R)\\
&= Tk(a^\star )-\frac{R_P}{R_A+R_P}Ta^\star +\frac{\overline\alpha^AT}{2} \frac{R_AR_P}{R_A+R_P}-\frac{1}{R_A}\log(-R).
\end{align*}

\vspace{0.8em}

\textbf{We now prove $(iii)$.} Assume that $\overline\alpha^A= \overline\alpha^P$, and notice that 
\begin{align*}
U_0^{P,{\rm FB}}&\leq \underset{\xi\in\mathcal C}{\sup}\ \underset{a\in\mathcal A}{\sup}\left\{\mathbb E^{\P^{\overline\alpha^P}_a}\left[\mathcal U_P\left(B_T-\xi\right)\right]+\rho\mathbb E^{\P^{\overline\alpha^A}_a}\left[\mathcal U_A\left(\xi-\int_0^Tk(a_s)ds\right)\right]\right\}=-\underset{\xi\in\mathcal C}{\inf}\ \underset{a\in\mathcal A}{\inf}\Xi_a^{\overline\alpha^P,\overline\alpha^A}(\xi),
\end{align*}
where we have used the fact that by definition, the law of $B$ under $\P^\alpha_a$ is equal to the law of $X^{a,\alpha}$ under $\P_0$. Let us then define for any $a\in\Ac$
$$\xi_a:=z^\star B_T+\frac{\gamma^\star }{2}\langle B\rangle_T +\delta^\star (a),$$
where $\gamma^\star \in [-R_P|1-z^\star |^2,R_A|z^\star |^2]$, and
$$z^\star :=\frac{R_P}{R_A+R_P}, \; \delta^\star (a):=\frac{1}{R_A+R_P}\left(\log\left(\rho\frac{R_A}{R_P}\right)+R_A \int_0^Tk(a_s)ds \right)-\frac{\gamma^\star }{2}\overline\alpha^PT.$$
Then by Lemma \ref{lemma:dxi}, we know that
\begin{align*}
\underset{\xi\in\mathcal C}{\inf}\ \underset{a\in\mathcal A}{\inf}\Xi_a^{\overline\alpha^P,\overline\alpha^A}(\xi)&=\underset{a\in\mathcal A}{\inf}\Xi_a^{\overline\alpha^P,\overline\alpha^A}(\xi_a).
\end{align*}
We then have
\begin{align*}
\Xi_a^{\overline\alpha^P,\overline\alpha^A}(\xi_a)=& \left(\rho\frac{R_A}{R_P}\right)^{\frac{R_P}{R_A+R_P}}\left(1+\frac{R_P}{R_A}\right)e^{\overline\alpha^PT\frac{R_A^2R_P^2}{2(R_A+R_P)^2}}\\
&\times\E^{\P_0}\left[\Ec\left(-\frac{R_AR_P}{R_A+R_P}(\overline\alpha^P)^{1/2}B_T\right)e^{\frac{R_AR_P}{R_A+R_P}\int_0^T(k(a_s(X_\cdot^{a,\overline\alpha^P}))-a_s(X_\cdot^{a,\overline\alpha^P}))ds}\right],
\end{align*}
so that we clearly have
$$\underset{a\in\mathcal A}{\inf}\Xi_a^{\overline\alpha^P,\overline\alpha^A}(\xi_a)=\Xi_{a^\star }^{\overline\alpha^P,\overline\alpha^A}(\xi_{a^\star }).$$
Thus we have obtained 
$$U_0^{P,{\rm FB}}\leq  -\Xi_{a^\star }^{\overline\alpha^P,\overline\alpha^A}(\xi_{a^\star }).$$
Conversely, we have from Lemma \ref{lemma:partitionQ} $(ii)$, $(iii)$$b)$, $(iv)$,
$$U_0^{P,{\rm FB}}\geq \underset{\xi\in\overline{\mathcal Q^{|\gamma|}}}{\sup}\ \underset{a\in\mathcal A_{det}}{\sup} \widetilde u_0^{P,FB}(a,\xi)= G(a^\star ,z^\star ,\overline\alpha^P)=-\Xi_{a^\star }^{\overline\alpha^P,\overline\alpha^A}(\xi_{a^\star }).$$
Therefore
$$U_0^{P,{\rm FB}}=-\Xi_{a^\star }^{\overline\alpha^P,\overline\alpha^A}(\xi_{a^\star }).$$
Finally, it remains to choose $\rho$ so as to satisfy the participation constraint of the Agent. Some calculations show that it suffices to take $\rho$ such that
$$\frac{1}{R_A+R_P}\log\left(\rho \frac{R_A}{R_P}\right)=-\frac{1}{R_A}\log(-R_0)+\frac{R_P}{R_A+R_P}T \left[k(a^\star )-a^\star +\overline\alpha^PT\frac{R_AR_P}{2(R_A+R_P)} \right]. $$
Thus,
$$\delta(a^\star )=Tk(a^\star )-\frac{R_P}{R_A+R_P}Ta^\star +\frac{\overline\alpha^PT}{2}\left(\frac{R_AR_P^2}{(R_A+R_P)^2}-\gamma^\star \right)-\frac{1}{R_A}\log(-R).$$

\textbf{We finally prove $(iv)$.} Assume that $\underline\alpha^P= \overline\alpha^A$, and notice that 
\begin{align*}
U_0^{P,{\rm FB}}&\leq \underset{\xi\in\mathcal C}{\sup}\ \underset{a\in\mathcal A}{\sup}\left\{\mathbb E^{\P^{\underline\alpha^P}_a}\left[\mathcal U_P\left(B_T-\xi\right)\right]+\rho\mathbb E^{\P^{\overline\alpha^A}_a}\left[\mathcal U_A\left(\xi-\int_0^Tk(a_s)ds\right)\right]\right\}=-\underset{\xi\in\mathcal C}{\inf}\ \underset{a\in\mathcal A}{\inf}\Xi_a^{\underline\alpha^P,\overline\alpha^A}(\xi),
\end{align*}
where we have used the fact that by definition, the law of $B$ under $\P^\alpha_a$ is equal to the law of $X^{a,\alpha}$ under $\P_0$. Let us then define for any $a\in\Ac$
$$\xi_a:=z^\star B_T+\frac{\gamma^\star }{2}\langle B\rangle_T +\delta^\star (a),$$
where $\gamma^\star \in (-\infty,-R_P|1-z^\star |^2]$, and
$$z^\star :=\frac{R_P}{R_A+R_P}, \; \delta^\star (a):=\frac{1}{R_A+R_P}\left(\log\left(\rho\frac{R_A}{R_P}\right)+R_A \int_0^Tk(a_s)ds \right)-\frac{\gamma^\star }{2}\underline\alpha^PT.$$
Then by Lemma \ref{lemma:dxi}, we know that
\begin{align*}
\underset{\xi\in\mathcal C}{\inf}\ \underset{a\in\mathcal A}{\inf}\Xi_a^{\underline\alpha^P,\overline\alpha^A}(\xi)&=\underset{a\in\mathcal A}{\inf}\Xi_a^{\underline\alpha^P,\overline\alpha^A}(\xi_a).
\end{align*}
We then have
\begin{align*}
\Xi_a^{\underline\alpha^P,\overline\alpha^A}(\xi_a)=& \left(\rho\frac{R_A}{R_P}\right)^{\frac{R_P}{R_A+R_P}}\left(1+\frac{R_P}{R_A}\right)e^{\underline\alpha^PT\frac{R_A^2R_P^2}{2(R_A+R_P)^2}}\\
&\times\E^{\P_0}\left[\Ec\left(-\frac{R_AR_P}{R_A+R_P}(\underline\alpha^P)^{1/2}B_T\right)e^{\frac{R_AR_P}{R_A+R_P}\int_0^T(k(a_s(X_\cdot^{a,\underline\alpha^P}))-a_s(X_\cdot^{a,\underline\alpha^P}))ds}\right],
\end{align*}
so that we clearly have
$$\underset{a\in\mathcal A}{\inf}\Xi_a^{\underline\alpha^P,\overline\alpha^A}(\xi_a)=\Xi_{a^\star }^{\underline\alpha^P,\overline\alpha^A}(\xi_{a^\star }).$$
Thus we have obtained 
$$U_0^{P,{\rm FB}}\leq  -\Xi_{a^\star }^{\underline\alpha^P,\overline\alpha^A}(\xi_{a^\star }).$$
Conversely, we have from Lemma \ref{lemma:partitionQ} $(i)$ $b.$, $(ii)$ and $(iii)$ $c.$,
$$U_0^{P,{\rm FB}}\geq \underset{\xi\in\overline{\mathcal Q^{\underline{\gamma}}}}{\sup}\ \underset{a\in\mathcal A_{det}}{\sup} \widetilde u_0^{P,FB}(a,\xi)= G(a^\star ,z^\star ,\underline\alpha^P)=-\Xi_{a^\star }^{\underline\alpha^P,\overline\alpha^A}(\xi_{a^\star }).$$
Therefore
$$U_0^{P,{\rm FB}}=-\Xi_{a^\star }^{\underline\alpha^P,\overline\alpha^A}(\xi_{a^\star }).$$
Finally, it remains to choose $\rho$ so as to satisfy the participation constraint of the Agent. Some calculations show that it suffices to take $\rho$ such that
$$\frac{1}{R_A+R_P}\log\left(\rho \frac{R_A}{R_P}\right)=-\frac{1}{R_A}\log(-R)+\frac{R_P}{R_A+R_P}T \left[k(a^\star )-a^\star +\underline\alpha^PT\frac{R_AR_P}{2(R_A+R_P)} \right]. $$
Thus,
$$\delta(a^\star )=Tk(a^\star )-\frac{R_P}{R_A+R_P}Ta^\star +\frac{\underline\alpha^PT}{2}\left(\frac{R_AR_P^2}{(R_A+R_P)^2}-\gamma^\star \right)-\frac{1}{R_A}\log(-R).$$
\end{proof}

\vspace{0.5em}
\begin{proof}[Proof of Proposition \ref{prop:optimaleffort}]
Our fist step is to look at the dynamic version of the value function of the Agent. Fix some $a\in\mathcal A$. We refer to the papers \cite{nutz2013constructing,neufeld2013superreplication} for the proofs that, for any $\Fc_T-$measurable contract $\xi\in\mathcal C$, one can define a process, which we denote by $u_t^A(\xi,a)$ (denoted by $Y_t$ in \cite{neufeld2013superreplication}), which is c\`adl\`ag, $\G^{\Pc_A,+}-$adapted (recall that for any $a\in\mathcal A$, $\G^{\Pc^a_A}=\G^{\Pc_A}$, since the polar sets of $\Pc^a_A$ are the same as the polar sets of $\Pc_A$) and such that
\begin{equation}\label{eq:Agent}
u_t^A(\xi,a)=\underset{\P^{'}\in\mathcal P^a_A(\P,t^+)}{{\rm essinf}^\P}\ \E^{\P^{'}}\left[\left.\mathcal U_A\left(\xi-\int_t^Tk(a_s)ds\right)\right|\mathcal F_t\right],\ \mathbb P-a.s., \text{ for all }\mathbb P\in\mathcal P^a_A.
\end{equation}
Notice that since $\xi\in\mathcal C$, it has exponential moments of any order, so that since in addition the effort process $a$ is bounded, we have that $u^A(\xi,a)$ has moments of any order, in the sense that
\begin{equation}\label{eq:momexp}
\underset{a\in\Ac}{\sup}\ \underset{\P\in\Pc^a_\Ac}{\sup}\E^\P\left[\underset{0\leq t\leq T}{\sup}\abs{u_t^A(\xi,a)}^p\right]<+\infty,\ \text{for all $p\geq 0$,}
\end{equation}
where we have used the generalized Doob inequality for sub--linear expectations given in Proposition A.1 in \cite{possamai2013second1}.

\vspace{0.5em}
Moreover, by \cite[step 2 in the proof of Theorem 2.3]{neufeld2013superreplication}, $e^{R_A\int_0^tk(a_s)ds}u_t^A(\xi,a)$ is a $(\mathbb P, \G^{\Pc_A,+})-$sub--martingale for every $\mathbb P\in\mathcal P^a_A$, and by \cite[step 3 in the proof of Theorem 2.3]{neufeld2013superreplication}, there is a $\G^{\Pc_A}-$predictable process $\widetilde Z$, and a family of non--decreasing and $\F^\P-$predictable processes $(\widetilde K^\P)_{\P\in\mathcal P^a_A}$, such that, $\text{for all }\mathbb P\in\mathcal P^a_A$
$$e^{R_A\int_0^tk(a_s)ds}u_t^A(\xi,a)=e^{R_A\int_0^Tk(a_s)ds}\mathcal U_A(\xi)-\int_t^T\widetilde Z_s\widehat\alpha_s^{\frac12}dW^{a}_s-\widetilde K_T^\P+\widetilde K_t^\P,\ \P-a.s.$$
Notice also that since every probability measure in $\mathcal P_A$ is equivalent, by definition, to a probability measure in $\Pc^a_A$ (and conversely), the above also holds $\P-a.s.$, for any $\P\in\mathcal P_A$, with the convention that we will still denote by $\widetilde K^\P$ the non-decreasing process associated to $\P\in\Pc^a_A$ or $\Pc_A$. Moreover, using the aggregation result of \cite{nutz2012pathwise}, we can actually aggregate the family $\widetilde K^\P$ into a universal process, which is $\G^{\Pc_A}-$predictable, and which we denote by $\widetilde K$.

\vspace{0.5em}
Define
$$Y_t^a:=-\frac{\ln\left(-u_t^A(\xi,a)\right)}{R_A},\; Z^a_t:=-\frac{e^{-R_A\int_0^tk(a_s)ds}}{R_Au_t(\xi,a)}\widetilde Z_t,\; K_t^{a}:=-\int_0^t\frac{e^{-R_A\int_0^sk(a_r)dr}}{R_Au^A_t(\xi,a)}d\widetilde K_r.$$
We have, after some computations, for all $\mathbb P\in\mathcal P_A$
$$Y_t^a=\xi-\int_t^T\left(\frac{R_A}{2}\abs{Z_s^a}^2\widehat\alpha_s+k(a_s)-a_sZ^a_s\right)ds-\int_t^TZ_s^a\widehat\alpha_s^{1/2}dW_s-\int_t^TdK_s^{a},\ \mathbb P-a.s.$$
Now notice that by \eqref{eq:momexp}, we immediately have
$$\underset{a\in\Ac}{\sup}\ \underset{\P\in\Pc_A}{\sup}\E^\P\left[\exp\left(p\underset{0\leq t\leq T}{\sup}\abs{Y_t^a}\right)\right]<+\infty,\ \text{for every $p\geq 0$.}$$
Moreover, remember that by \reff{eq:Agent}, we also have for every $\P\in\mathcal P^a_A$, by the exact same arguments as above applied under any fixed measure $\P\in\Pc_A$, that 
\begin{equation}\label{eq:minimality}
Y_t^{a}=\underset{\P^{'}\in\mathcal P^a_A(\P,t^+)}{{\rm essinf}^\P}\ \Yc^{\P^{'},a}_t,\ \P-a.s.,
\end{equation}
where for any $\P\in\Pc_A^a$, $(\Yc^{\P,a}, \Zc^{\P,a})$ is the unique\footnote{Wellposedness is clear here, since we have easily that $\Yc^{\P,a}_t=-\frac1{R_A}\log\left(-\E^{\P}\left[\left.\mathcal U_A\left(\xi-\int_t^Tk(a_s)ds\right)\right|\mathcal F_t\right]\right),\ \mathbb P-a.s.$} solution to the following BSDE defined under $\P$
$$\Yc^{\P,a}_t=\xi-\int_t^T\left(\frac{R_A}{2}\abs{\Zc_s^{\P,a}}^2\widehat\alpha_s+k(a_s)-a_s\Zc^{\P,a}_s\right)ds-\int_t^T\Zc_s^{\P,a}\widehat\alpha_s^{1/2}dW_s,\ \P-a.s.$$
Then, using \eqref{eq:momexp}, we can follow the proof of Lemma 3.1 in \cite{possamai2013second}\footnote{In this result, $\xi$ and $Y^a$ are assumed to be bounded, but the proof generalizes easily to our setting where $Y^a$ satisfies \eqref{eq:momexp}.} to obtain that $Z^a$ actually belongs to the BMO space defined in \cite{possamai2013second} (see Section 2.3.2). Then, we can follow exactly the proof of Theorem 6.1 in \cite{possamai2013second} to obtain with \eqref{eq:minimality}, that for any $\P\in\Pc_A^a$
$$K_t^a=\underset{\P^{'}\in\mathcal P^a_A(\P,t^+)}{{\rm essinf}^\P}\E^{\P^{'}}\left[\left.K_T^a\right|\mathcal F_t\right],\ \P-a.s.$$
Therefore, $(Y_t^a,Z^a_t)$ is the unique solution to the (quadratic-linear) 2BSDE with terminal condition $\xi$ and generator $R_A/2z^2\widehat \alpha_s+k(a_s)-a_sz$ (see for instance Definition 2.3 of \cite{possamai2013second}).

\vspace{0.5em}
The final step of the proof is now to relate the family $(Y^a)_{a\in\Ac}$ with the solution of the 2BSDE \eqref{eq:2bsde}. Before proceeding, let us explain why the 2BSDE \eqref{eq:2bsde} does indeed admit a maximal solution. First of all, the corresponding quadratic BSDEs admit a maximal solution, because, since the infimum in the generator is over a compact set, the generator of the BSDE is bounded from above by a function with linear growth in $z$. The existence of a maximal solution is then direct from Proposition 4 of \cite{briand2008quadratic}. Furthermore, since this maximal solution is obtained as a monotone approximation of Lipschitz BSDEs, it satisfies a comparison theorem. Hence, we can apply first Proposition $2.1$ of \cite{possamai2015stochastic} to obtain the existence of a maximal solution of the 2BSDE, in the sense of Definition $4.1$ of \cite{possamai2015stochastic}, and then use Remark $4.1$ of \cite{possamai2015stochastic} to aggregate the family of non-decreasing processes into $K$ (we remind the reader that all the measures in $\Pc_A$ satisfy the predictable martingale representation property).

\vspace{0.5em}
In particular, we have the following representation for any $\P\in\Pc_A$, 
\begin{equation}\label{eq:minimality2}
Y_t=\underset{\P^{'}\in\mathcal P^a_A(\P,t^+)}{{\rm essinf}^\P}\ \Yc^{\P^{'}}_t,\ \P-a.s.,
\end{equation}
where for any $\P\in\Pc_A^a$, $(\Yc^{\P}, \Zc^{\P})$ is the maximal solution of the quadratic BSDE
$$\Yc_t^\P=\xi-\int_t^T\left(\frac{R_A}{2}\abs{\Zc^\P_s}^2\widehat\alpha_s+\underset{a\in[0,a_{\max}]}{\inf}\left\{k(a)-a\Zc^\P_s\right\}\right)ds-\int_t^T\Zc^\P_s\widehat\alpha_s^{1/2}dW_s,\ \P-a.s.$$
Now it is a classical result dating back to \cite{el1991programmation,el1995dynamic,hamadene1995backward} (see also \cite{matoussi2015robust} for a similar result using 2BSDEs) that, using the comparison theorem satisfied by the maximal solution of the 2BSDEs (which is automatically inherited from the one satisfied by the BSDEs), that
$$Y_0=\underset{a\in\Ac}{\sup}\ Y_0^{a}=\underset{a\in\Ac}{\sup}\ \underset{\P\in\mathcal P_A^a}{\inf}\ \Yc^{\P,a}_0=\underset{a\in\Ac}{\sup}\ \underset{\P\in\mathcal P_A^a}{\inf}\left\{ -\frac{1}{R_A}\log\left(-\E^{\P}\left[\mathcal U_A\left(\xi-\int_0^Tk(a_s)ds\right)\right]\right)\right\},$$
so that
$$U_0^A(\xi)=-\exp(-R_AY_0).$$
Furthermore, it is then clear, since the function $k$ is strictly convex that there is some $a^\star (Z_\cdot)\in\Ac$ such that
$$\underset{a\in[0,a_{\max}]}{\inf}\left\{k(a)-aZ_s\right\}= k(a^\star (Z_s))-a^\star (Z_s)Z_s, \; s\in [0,T].$$
This implies that $Y_0=\underset{\P\in\mathcal P_A^{a^\star (Z_\cdot)}}{\inf}\ \Yc^{\P,a^\star (Z_\cdot)}_0.$
\end{proof}

\vspace{0.5em}
\begin{proof}[Proof of Theorem \ref{thm:SB}] We recall Definition \eqref{def:z*} for any $\alpha\geq 0$
$$z^\star (\alpha):=\frac{1+k\alpha R_P}{1+\alpha k(R_A+R_P)}.$$
\noindent \textbf{We begin with the proof of $(i)$.} Assume that $ \underline{\alpha}^A\leq\overline\alpha^P\leq\overline{\alpha}^A$, then
$$\underline U^P_0\geq   \underset{\P\in\Pc_P^{a^\star (z^\star (\widehat\alpha_\cdot))}}{\inf}\E^\P\left[-\mathcal E\left(-R_P\int_0^T\widehat\alpha_s^{\frac12}(1-z^\star (\widehat\alpha_s))dW_s^{a^\star (z^\star )}\right)e^{R_P\left(R_0-\int_0^TH(\widehat\alpha_s, z^\star (\widehat\alpha_s), 0)ds\right)}\right].$$
Then we have for any $\alpha\geq 0$
\begin{align*}
H(\alpha,z^\star (\alpha),0)
&=-\frac\alpha2 R_P+ \frac{(1+\alpha k R_P)^2}{2k(1+\alpha k(R_A+R_P))}.
\end{align*}
Hence,
\begin{align*}
\frac{\partial H}{\partial \alpha}(\alpha,z^\star (\alpha),0)=\frac{-R_A\left(1+2k\alpha R_P +k^2\alpha^2 R_P(R_A+R_P)\right)}{2(1+\alpha k(R_A+R_P))^2}\leq 0, \ \forall \alpha\in [\underline\alpha^P, \overline\alpha^P].
\end{align*}
Therefore, $\underline U^P_0\geq -e^{R_P R_0}e^{-R_P\int_0^TH(\overline\alpha^P, z^\star (\overline\alpha^P), 0)ds}.$ Indeed, $z^\star (\widehat\alpha_s)$ is bounded so that the stochastic exponential is trivially a true martingale.
We now turn to the converse inequality, we have
$$ 
\underline U_0^P\leq  \underset{(Z,\Gamma)\in \mathfrak K}{\sup}\ \E^{\P_{a^\star (Z_\cdot)}^{\overline\alpha^P}}\left[-\mathcal E\left(-R_P\int_0^T(\overline{\alpha}^P)^{1/2}(1-Z_s)dW_s^{a^\star (Z_\cdot)}\right)e^{R_P\left(R_0-\int_0^TH(\overline\alpha^P,Z_s,\Gamma_s)ds\right)}\right].
$$
According to Lemma \ref{lemma:F}$(i)$, we obtain 
\begin{align*}
\underline  U_0^P&\leq -e^{R_P R_0} e^{-R_PTH(\overline\alpha^P, z^\star (\overline\alpha^P), 0)}.
 \end{align*}

Hence, if $\underline \alpha^A\leq\overline\alpha^P\leq \overline\alpha^A $, then $\underline U_0^P=  -e^{R_PR_0}e^{-R_P\int_0^TH(\overline\alpha^P, z^\star (\overline\alpha^P), 0)ds}.$ We now prove that the contract $\xi^{R_0,z^\star ({\overline\alpha^P}), 0}\in\mathfrak C^{{\rm SB}}$ is indeed optimal. We have 
\begin{align*}
& \underset{\P\in\Pc_P^{a^\star (z^\star (\overline\alpha^P))}}{\inf}\E^\P\left[-\mathcal E\left(-R_P\int_0^T\widehat\alpha_s^{\frac12}(1-z^\star (\overline\alpha^P))dW_s^{a^\star (z^\star (\overline\alpha^P))}\right)e^{R_P\left(R_0-\int_0^TH(\widehat\alpha_s, z^\star (\overline\alpha^P), 0)ds\right)}\right]=\underline U_0^P,
\end{align*}
since by definition \eqref{eq:F:k} of $H$, $\alpha\longmapsto H(\alpha, z^\star (\overline\alpha^P),0)$ is decreasing, so that the above infimum is attained for the measure $\P^{\overline\alpha^P}_{a^\star (z^\star (\overline\alpha^P))}$.

\vspace{0.5em}

\noindent \textbf{We now turn to the proof of $(ii)$}. Assume that $\underline\alpha^P \leq \overline\alpha^A\leq \overline\alpha^P$. On the one hand
\begin{align*}
 &  \underset{\P\in\Pc_P^{a^\star (z^\star (\overline\alpha^A))}}{\inf}\E^\P\left[-\mathcal E\left(-R_P\int_0^T\widehat\alpha_s^{\frac12}(1-z^\star (\overline\alpha^A))dW_s^{a^\star (z^\star (\overline\alpha^A))}\right)e^{R_P\left(R_0-\int_0^TH(\widehat\alpha_s, z^\star (\overline\alpha^A), \gamma^\star )ds\right)}\right] \leq \underline U_0^P,
 \end{align*}
where $\gamma^\star := -R_A (z^\star (\overline\alpha^A))^2 -R_P(1-z^\star (\overline\alpha^A))^2.$
Thus, using Relation \eqref{equalityF}, we have
 $$ \underline U_0^P\geq  -e^{R_P R_0}e^{-R_PTH(\overline\alpha^A,z^\star (\overline\alpha^A),0)}.$$
 On the other hand, since $\overline\alpha^A\in [\underline\alpha^P, \overline\alpha^P]$
 $$\underline U_0^P\leq e^{R_P R_0}\underset{(Z,\Gamma)\in \mathfrak U}{\sup}\ \E^{\P^{\overline\alpha^A}_{a^\star (Z_\cdot)}}\left[-\mathcal E\left(-R_P\int_0^T{\overline{\alpha}^A}^{1/2}(1-Z_s)dW_s^{a^\star (Z_\cdot)}\right)e^{-R_P\int_0^TH(\overline\alpha^A,Z_s,\Gamma_s)ds}\right].
$$ By using Lemma \ref{lemma:F}$(i)$, we obtain
$$\underline U_0^P\leq -e^{R_P R_0}e^{-R_PTH(\overline\alpha^A,z^\star (\overline\alpha^A),0)}.$$

We consider now a contract $\xi^{R_0,z^\star (\overline\alpha^A), \gamma^\star }$ and we show that $\underline U_0^P(\xi^{R_0,z^\star (\overline\alpha^A), \gamma^\star } )=\underline U_0^P$. We have
\begin{align*}
&\underset{\P\in\Pc_P^{a^\star (z^\star (\overline\alpha^A)}}{\inf}\E^\P\left[-\mathcal E\left(-R_P\int_0^T\widehat\alpha_s^{1/2}(1-z^\star (\overline\alpha^A))dW_s^{a^\star (z^\star (\overline\alpha^A))}\right)e^{R_P\left(R_0-\int_0^TH(\widehat\alpha_s, z^\star (\overline\alpha^A), \gamma^\star )ds\right)}\right]\\
& =-e^{R_P R_0}e^{-R_PTH(\overline\alpha^A,z^\star (\overline\alpha^A), 0)} =\underline U_0^P,
\end{align*}
since $H(\alpha, z^\star (\overline\alpha^A),\gamma^\star )$ is actually independent of $\alpha$.

\vspace{0.5em}
The last two results are immediate consequences of Proposition \ref{prop:degenerate2nd}. \end{proof}
\end{appendix}

\small
 \bibliographystyle{plain}
\bibliography{bibliographyDylan}

\begin{thebibliography}{100}

\bibitem{adrian2009disagreement}
T.~Adrian and M.M. Westerfield.
\newblock Disagreement and learning in a dynamic contracting model.
\newblock {\em Review of Financial Studies}, 22(10):3873--3906, 2009.

\bibitem{aidpt}
R.~A{\"\i}d, D.~Possama{\"\i}, and N.~Touzi.
\newblock A principal--agent model for pricing electricity volatility demand.
\newblock {\em preprint}, 2016.

\bibitem{bayraktar2012stochastic}
E.~Bayraktar and M.~S{\^\i}rbu.
\newblock Stochastic {P}erron's method and verification without smoothness
  using viscosity comparison: the linear case.
\newblock {\em Proceedings of the American Mathematical Society},
  140(10):3645--3654, 2012.

\bibitem{bayraktar2013stochastic}
E.~Bayraktar and M.~S{\^\i}rbu.
\newblock Stochastic {P}erron's method for {H}amilton--{J}acobi--{B}ellman
  equations.
\newblock {\em SIAM Journal on Control and Optimization}, 51(6):4274--4294,
  2013.

\bibitem{biais2010large}
B.~Biais, T.~Mariotti, J.-C. Rochet, and S.~Villeneuve.
\newblock Large risks, limited liability, and dynamic moral hazard.
\newblock {\em Econometrica}, 78(1):73--118, 2010.

\bibitem{bolton2005contract}
P.~Bolton and M.~Dewatripont.
\newblock {\em Contract theory}.
\newblock MIT press, 2005.

\bibitem{borch1992equilibrium}
K.~Borch.
\newblock Equilibrium in a reinsurance market.
\newblock {\em Econometrica}, 30(3):424--444, 1962.

\bibitem{bouchard2014stochastic}
B.~Bouchard, L.~Moreau, and M.~Nutz.
\newblock Stochastic target games with controlled loss.
\newblock {\em The Annals of Applied Probability}, 24(3):899--934, 2014.

\bibitem{briand2008quadratic}
P.~Briand and Y.~Hu.
\newblock Quadratic {BSDE}s with convex generators and unbounded terminal
  conditions.
\newblock {\em Probability Theory and Related Fields}, 141(3-4):543--567, 2008.

\bibitem{buckdahn2008stochastic}
R.~Buckdahn and J.~Li.
\newblock Stochastic differential games and viscosity solutions of
  {H}amilton--{J}acobi--{B}ellman--{I}saacs equations.
\newblock {\em SIAM Journal on Control and Optimization}, 47(1):444--475, 2008.

\bibitem{cadenillas2007optimal}
A.~Cadenillas, J.~Cvitani{\'c}, and F.~Zapatero.
\newblock Optimal risk--sharing with effort and project choice.
\newblock {\em Journal of Economic Theory}, 133(1):403--440, 2007.

\bibitem{capponi2015dynamic}
A.~Capponi and C.~Frei.
\newblock Dynamic contracting: accidents lead to nonlinear contracts.
\newblock {\em SIAM Journal on Financial Mathematics}, 6(1):959--983, 2015.

\bibitem{carlier2007optimal}
G.~Carlier, I.~Ekeland, and N.~Touzi.
\newblock Optimal derivatives design for mean--variance agents under adverse
  selection.
\newblock {\em Mathematics and Financial Economics}, 1(1):57--80, 2007.

\bibitem{cvitanic2014moral}
J.~Cvitani{\'c}, D.~Possama{\"\i}, and N.~Touzi.
\newblock Moral hazard in dynamic risk management.
\newblock {\em Management Science}, to appear, 2014.

\bibitem{cvitanic2015dynamic}
J.~Cvitani{\'c}, D.~Possama{\"\i}, and N.~Touzi.
\newblock Dynamic programming approach to principal--agent problems.
\newblock {\em arXiv preprint arXiv:1510.07111}, 2015.

\bibitem{cvitanic2013dynamics}
J.~Cvitani{\'c}, X.~Wan, and H.~Yang.
\newblock Dynamics of contract design with screening.
\newblock {\em Management Science}, 59(5):1229--1244, 2013.

\bibitem{cvitanic2006optimal}
J.~Cvitani{\'c}, X.~Wan, and J.~Zhang.
\newblock Optimal contracts in continuous--time models.
\newblock {\em International Journal of Stochastic Analysis}, 2006(095203),
  2006.

\bibitem{cvitanic2009optimal}
J.~Cvitani{\'c}, X.~Wan, and J.~Zhang.
\newblock Optimal compensation with hidden action and lump--sum payment in a
  continuous--time model.
\newblock {\em Applied Mathematics and Optimization}, 59(1):99--146, 2009.

\bibitem{cvitanic2007optimal}
J.~Cvitani{\'c} and J.~Zhang.
\newblock Optimal compensation with adverse selection and dynamic actions.
\newblock {\em Mathematics and Financial Economics}, 1(1):21--55, 2007.

\bibitem{cvitanic2012contract}
J.~Cvitani{\'c} and J.~Zhang.
\newblock {\em Contract theory in continuous--time models}.
\newblock Springer, 2012.

\bibitem{demarzo2011learning}
P.~DeMarzo and Y.~Sannikov.
\newblock Learning, termination and payout policy in dynamic incentive
  contracts.
\newblock Stanford university and Princeton university, 2011.

\bibitem{denis2006theoretical}
L.~Denis and C.~Martini.
\newblock A theoretical framework for the pricing of contingent claims in the
  presence of model uncertainty.
\newblock {\em The Annals of Applied Probability}, 16(2):827--852, 2006.

\bibitem{djehiche2014principal}
B.~Djehiche and P.~Helgesson.
\newblock The principal--agent problem; a stochastic maximum principle
  approach.
\newblock {\em arXiv preprint arXiv:1410.6392}, 2014.

\bibitem{djehiche2015principal}
B.~Djehiche and P.~Helgesson.
\newblock The principal--agent problem with time inconsistent utility
  functions.
\newblock {\em arXiv preprint arXiv:1503.05416}, 2015.

\bibitem{duffie1994efficient}
D.~Duffie, P.-Y. Geoffard, and C.~Skiadas.
\newblock Efficient and equilibrium allocations with stochastic differential
  utility.
\newblock {\em Journal of Mathematical Economics}, 23(2):133--146, 1994.

\bibitem{dumas2000efficient}
B.~Dumas, R.~Uppal, and T.~Wang.
\newblock Efficient intertemporal allocations with recursive utility.
\newblock {\em Journal of Economic Theory}, 93(2):240--259, 2000.

\bibitem{ekren2014viscosity}
I.~Ekren, C.~Keller, N.~Touzi, and J.~Zhang.
\newblock On viscosity solutions of path dependent {PDE}s.
\newblock {\em The Annals of Probability}, 42(1):204--236, 2014.

\bibitem{ekren2016viscosity}
I.~Ekren, N.~Touzi, and J.~Zhang.
\newblock Viscosity solutions of fully nonlinear parabolic path dependent
  {PDE}s: part i.
\newblock {\em The Annals of Probability}, 44(2):1212--1253, 2016.

\bibitem{ekren2012viscosity}
I.~Ekren, N.~Touzi, and J.~Zhang.
\newblock Viscosity solutions of fully nonlinear parabolic path dependent
  {PDE}s: part ii.
\newblock {\em The Annals of Probability}, 44(4):2507--2553, 2016.

\bibitem{el1991programmation}
N.~El~Karoui and M.-C. Quenez.
\newblock Programmation dynamique et {\'e}valuation des actifs contingents en
  march{\'e} incomplet.
\newblock {\em Comptes rendus de l'Acad{\'e}mie des sciences. S{\'e}rie 1,
  Math{\'e}matique}, 313(12):851--854, 1991.

\bibitem{el1995dynamic}
N.~El~Karoui and M.-C. Quenez.
\newblock Dynamic programming and pricing of contingent claims in an incomplete
  market.
\newblock {\em SIAM journal on Control and Optimization}, 33(1):29--66, 1995.

\bibitem{karoui2013capacities}
N.~El~Karoui and X.~Tan.
\newblock Capacities, measurable selection and dynamic programming part i:
  abstract framework.
\newblock {\em arXiv preprint arXiv:1310.3363}, 2013.

\bibitem{karoui2013capacities2}
N.~El~Karoui and X.~Tan.
\newblock Capacities, measurable selection and dynamic programming part ii:
  application in stochastic control problems.
\newblock {\em arXiv preprint arXiv:1310.3364}, 2013.

\bibitem{epstein2013ambiguous}
L.G. Epstein and S.~Ji.
\newblock Ambiguous volatility and asset pricing in continuous time.
\newblock {\em Review of Financial Studies}, 26(7):1740--1786, 2013.

\bibitem{epstein2014ambiguous}
L.G. Epstein and S.~Ji.
\newblock Ambiguous volatility, possibility and utility in continuous time.
\newblock {\em Journal of Mathematical Economics}, 50:269--282, 2014.

\bibitem{farhi2013insurance}
E;~Farhi and I.~Werning.
\newblock Insurance and taxation over the life cycle.
\newblock {\em The Review of Economic Studies}, 80(2):596--635, 2013.

\bibitem{fleming1989existence}
W.H. Fleming and P.E. Souganidis.
\newblock On the existence of value--functions of $2-$player, zero--sum
  stochastic differential--games.
\newblock {\em Indiana University Mathematics Journal}, 38(2):293--314, 1989.

\bibitem{garrett2009dynamic}
D.F. Garrett and A.~Pavan.
\newblock Dynamic managerial compensation: a mechanism design approach.
\newblock Technical report, Northwestern university, 2009.

\bibitem{giat2010investment}
Y.~Giat, S.T. Hackman, and A.~Subramanian.
\newblock Investment under uncertainty, heterogeneous beliefs, and agency
  conflicts.
\newblock {\em Review of Financial Studies}, 23(4):1360--1404, 2010.

\bibitem{giat2013dynamic}
Y.~Giat and A.~Subramanian.
\newblock Dynamic contracting under imperfect public information and asymmetric
  beliefs.
\newblock {\em Journal of Economic Dynamics and Control}, 37(12):2833--2861,
  2013.

\bibitem{golosov2016redistribution}
M.~Golosov, M.~Troshkin, and A.~Tsyvinski.
\newblock Redistribution and social insurance.
\newblock {\em The American Economic Review}, 106(2):359--386, 2016.

\bibitem{hamadene1995backward}
S.~Hamad\`ene and J.-P. Lepeltier.
\newblock Backward equations, stochastic control and zero--sum stochastic
  differential games.
\newblock {\em Stochastics: An International Journal of Probability and
  Stochastic Processes}, 54(3-4):221--231, 1995.

\bibitem{he2009optimal}
Z.~He.
\newblock Optimal executive compensation when firm size follows geometric
  brownian motion.
\newblock {\em Review of Financial Studies}, 22(2):859--892, 2009.

\bibitem{he2010permanent}
Z.~He, B.~Wei, and J.~Yu.
\newblock Permanent risk and dynamic incentives.
\newblock Baruch college, 2010.

\bibitem{he2014optimal}
Z.~He, B.~Wei, and J.~Yu.
\newblock Optimal long--term contracting with learning.
\newblock Technical report, University of Chicago and University of Minnesota,
  2014.

\bibitem{ECTA:ECTA375}
M.F. Hellwig and K.M. Schmidt.
\newblock Discrete--time approximations of the {H}olmstr{\"o}m--{M}ilgrom
  {B}rownian--motion model of intertemporal incentive provision.
\newblock {\em Econometrica}, 70(6):2225--2264, 2002.

\bibitem{holmstrom1987aggregation}
B.~Holmstr{\"o}m and P.~Milgrom.
\newblock Aggregation and linearity in the provision of intertemporal
  incentives.
\newblock {\em Econometrica}, 55(2):303--328, 1987.

\bibitem{karandikar1995pathwise}
R.L. Karandikar.
\newblock On pathwise stochastic integration.
\newblock {\em Stochastic Processes and their Applications}, 57(1):11--18,
  1995.

\bibitem{laffont2009theory}
J.-J. Laffont and D.~Martimort.
\newblock {\em The theory of incentives: the principal--agent model}.
\newblock Princeton University Press, 2009.

\bibitem{leung2014continuous}
R.C.W. Leung.
\newblock Continuous--time principal--agent problem with drift and stochastic
  volatility control: with applications to delegated portfolio management.
\newblock {\em Available at SSRN 2482009}, 2014.

\bibitem{matoussi2015robust}
A.~Matoussi, D.~Possama{\"\i}, and C.~Zhou.
\newblock Robust utility maximization in nondominated models with 2{BSDE}: the
  uncertain volatility model.
\newblock {\em Mathematical Finance}, 25(2):258--287, 2015.

\bibitem{miao2013robust}
J.~Miao and A.~Rivera.
\newblock Robust contracts in continuous time.
\newblock Technical report, Boston university, 2013.

\bibitem{miller2015optimal}
C.W. Miller and I.~Yang.
\newblock Optimal dynamic contracts for a large--scale principal--agent
  hierarchy: a concavity--preserving approach.
\newblock {\em arXiv preprint arXiv:1506.05497}, 2015.

\bibitem{Muller1998276}
H.M. M{\"u}ller.
\newblock The first--best sharing rule in the continuous--time principal--agent
  problem with exponential utility.
\newblock {\em Journal of Economic Theory}, 79(2):276--280, 1998.

\bibitem{muller2000asymptotic}
H.M. M{\"u}ller.
\newblock Asymptotic efficiency in dynamic principal--agent problems.
\newblock {\em Journal of Economic Theory}, 91(2):292--301, 2000.

\bibitem{neufeld2013superreplication}
A.~Neufeld and M.~Nutz.
\newblock Superreplication under volatility uncertainty for measurable claims.
\newblock {\em Electronic Journal of Probability}, 18(48):1--14, 2013.

\bibitem{nutz2012pathwise}
M.~Nutz.
\newblock Pathwise construction of stochastic integrals.
\newblock {\em Electronic Communications in Probability}, 17(24):1--7, 2012.

\bibitem{nutz2013random}
M.~Nutz.
\newblock Random $g-$expectations.
\newblock {\em The Annals of Applied Probability}, 23(5):1755--1777, 2013.

\bibitem{nutz2013constructing}
M.~Nutz and R.~van Handel.
\newblock Constructing sublinear expectations on path space.
\newblock {\em Stochastic Processes and their Applications}, 123(8):3100--3121,
  2013.

\bibitem{ou2003optimal}
H.~Ou-Yang.
\newblock Optimal contracts in a continuous--time delegated portfolio
  management problem.
\newblock {\em Review of Financial Studies}, 16(1):173--208, 2003.

\bibitem{ou2005equilibrium}
H.~Ou-Yang.
\newblock An equilibrium model of asset pricing and moral hazard.
\newblock {\em Review of Financial Studies}, 18(4):1253--1303, 2005.

\bibitem{pages2012bank}
H.~Pag{\`e}s.
\newblock Bank monitoring incentives and optimal {ABS}.
\newblock {\em Journal of Financial Intermediation}, 22(1):30--54, 2013.

\bibitem{pages2014mathematical}
H.~Pag{\`e}s and D.~Possama{\"\i}.
\newblock A mathematical treatment of bank monitoring incentives.
\newblock {\em Finance and Stochastics}, 18(1):39--73, 2014.

\bibitem{pavan2009dynamic}
A.~Pavan, I.R. Segal, and J.~Toikka.
\newblock Dynamic mechanism design: incentive compatibility, profit
  maximization and information disclosure.
\newblock Technical report, Northwestern university and Stanford university,
  2009.

\bibitem{peng2014complete}
S.~Peng, Y.~Song, and J.~Zhang.
\newblock A complete representation theorem for $g-$martingales.
\newblock {\em Stochastics: An International Journal of Probability and
  Stochastic Processes}, 86(4):609--631, 2014.

\bibitem{pham2014two}
T.~Pham and J.~Zhang.
\newblock Two person zero--sum game in weak formulation and path dependent
  {B}ellman--{I}saacs equation.
\newblock {\em SIAM Journal on Control and Optimization}, 52(4):2090--2121,
  2014.

\bibitem{piskorski2010optimal}
T.~Piskorski and A.~Tchistyi.
\newblock Optimal mortgage design.
\newblock {\em Review of Financial Studies}, 23(8):3098--3140, 2010.

\bibitem{possamai2013second1}
D.~Possama{\"\i}.
\newblock Second order backward stochastic differential equations under a
  monotonicity condition.
\newblock {\em Stochastic Processes and their Applications}, 123(5):1521--1545,
  2013.

\bibitem{possamai2015stochastic}
D.~Possama{\"\i}, X.~Tan, and C.~Zhou.
\newblock Stochastic control for a class of nonlinear kernels and applications.
\newblock {\em arXiv preprint arXiv:1510.08439}, 2015.

\bibitem{possamai2013second}
D.~Possama{\"\i} and C.~Zhou.
\newblock Second order backward stochastic differential equations with
  quadratic growth.
\newblock {\em Stochastic Processes and their Applications},
  123(10):3770--3799, 2013.

\bibitem{prat2014dynamic}
J.~Prat and B.~Jovanovic.
\newblock Dynamic contracts when the agent's quality is unknown.
\newblock {\em Theoretical Economics}, 9(3):865--914, 2014.

\bibitem{rao1991theory}
M.M. Rao and Z.D. Ren.
\newblock {\em Theory of {O}rlicz spaces}, volume 146 of {\em Monographs and
  textbooks in pure and applied mathematics}.
\newblock Marcel Dekker Inc., New York, 1991.

\bibitem{rao2002applications}
M.M. Rao and Z.D. Ren.
\newblock {\em Applications of {O}rlicz spaces,}.
\newblock Monographs and textbooks in pure and applied mathematics. Marcel
  Dekker Inc., New York, 2002.

\bibitem{ren2015convergence}
Z.~Ren and X.~Tan.
\newblock On the convergence of monotone schemes for path--dependent {PDE}.
\newblock {\em arXiv preprint arXiv:1504.01872}, 2015.

\bibitem{ren2014comparison}
Z.~Ren, N.~Touzi, and J.~Zhang.
\newblock Comparison of viscosity solutions of semi--linear path--dependent
  {PDE}s.
\newblock {\em arXiv preprint arXiv:1410.7281}, 2014.

\bibitem{Ren2014}
Z.~Ren, N.~Touzi, and J.~Zhang.
\newblock An overview of viscosity solutions of path--dependent {PDE}s.
\newblock In D.~Crisan, B.~Hambly, and T.~Zariphopoulou, editors, {\em
  Stochastic analysis and applications 2014: in honour of Terry Lyons}, volume
  100 of {\em Springer proceedings in mathematics and statistics}, pages
  397--453. Springer, 2014.

\bibitem{ren2015comparison}
Z.~Ren, N.~Touzi, and J.~Zhang.
\newblock Comparison of viscosity solutions of fully nonlinear degenerate
  parabolic path--dependent {PDE}s.
\newblock {\em arXiv preprint arXiv:1511.05910}, 2015.

\bibitem{ross1973economic}
S.A. Ross.
\newblock The economic theory of agency: the principal's problem.
\newblock {\em The American Economic Review}, 63(2):134--139, 1973.

\bibitem{salanie2005economics}
B.~Salani{\'e}.
\newblock {\em The economics of contracts: a primer}.
\newblock MIT press, 2005.

\bibitem{sannikov2007agency}
Y.~Sannikov.
\newblock Agency problems, screening and increasing credit lines.
\newblock Princeton university, 2007.

\bibitem{sannikov2008continuous}
Y.~Sannikov.
\newblock A continuous--time version of the principal--agent problem.
\newblock {\em The Review of Economic Studies}, 75(3):957--984, 2008.

\bibitem{sannikov2012contracts}
Y.~Sannikov.
\newblock Contracts: the theory of dynamic principal--agent relationships and
  the continuous--time approach.
\newblock In D.~Acemoglu, M.~Arellano, and E.~Dekel, editors, {\em Advances in
  economics and econometrics, 10th world congress of the Econometric Society,
  volume 1, economic theory}, number~49 in Econometric Society Monographs,
  pages 89--124. Cambridge University Press, 2013.

\bibitem{schattler1993first}
H.~Sch{\"a}ttler and J.~Sung.
\newblock The first--order approach to the continuous--time principal--agent
  problem with exponential utility.
\newblock {\em Journal of Economic Theory}, 61(2):331--371, 1993.

\bibitem{schattler1997optimal}
H.~Sch{\"a}ttler and J.~Sung.
\newblock On optimal sharing rules in discrete--and continuous--time
  principal--agent problems with exponential utility.
\newblock {\em Journal of Economic Dynamics and Control}, 21(2):551--574, 1997.

\bibitem{sirbu2014stochastic}
M.~S{\^\i}rbu.
\newblock Stochastic {P}erron's method and elementary strategies for zero--sum
  differential games.
\newblock {\em SIAM Journal on Control and Optimization}, 52(3):1693--1711,
  2014.

\bibitem{soner2011quasi}
H.M. Soner, N.~Touzi, and J.~Zhang.
\newblock Quasi--sure stochastic analysis through aggregation.
\newblock {\em Electronic Journal of Probability}, 16(2):1844--1879, 2011.

\bibitem{soner2012wellposedness}
H.M. Soner, N.~Touzi, and J.~Zhang.
\newblock Wellposedness of second order backward {SDE}s.
\newblock {\em Probability Theory and Related Fields}, 153(1-2):149--190, 2012.

\bibitem{spear1987repeated}
S.E. Spear and S.~Srivastava.
\newblock On repeated moral hazard with discounting.
\newblock {\em The Review of Economic Studies}, 54(4):599--617, 1987.

\bibitem{stroock2007multidimensional}
D.W. Stroock and S.R.S. Varadhan.
\newblock {\em Multidimensional diffusion processes}.
\newblock Springer, 2007.

\bibitem{sung1995linearity}
J.~Sung.
\newblock Linearity with project selection and controllable diffusion rate in
  continuous--time principal--agent problems.
\newblock {\em The RAND Journal of Economics}, 26(4):720--743, 1995.

\bibitem{sung1997corporate}
J.~Sung.
\newblock Corporate insurance and managerial incentives.
\newblock {\em Journal of Economic Theory}, 74(2):297--332, 1997.

\bibitem{sung2001lectures}
J.~Sung.
\newblock Lectures on the theory of contracts in corporate finance: from
  discrete--time to continuous--time models.
\newblock {\em Com2Mac Lecture Note Series}, 4, 2001.

\bibitem{sung2005optimal}
J.~Sung.
\newblock Optimal contracts under adverse selection and moral hazard: a
  continuous--time approach.
\newblock {\em Review of Financial Studies}, 18(3):1021--1073, 2005.

\bibitem{sung2015optimal}
J.~Sung.
\newblock Optimal contracting under mean--volatility ambiguity uncertainties.
\newblock {\em Available at {SSRN} 2601174}, 2015.

\bibitem{szydlowski2012ambiguity}
M.~Szydlowski.
\newblock Ambiguity in dynamic contracts.
\newblock Technical report, University of Minnesota, 2012.

\bibitem{weinschenk2010moral}
P.~Weinschenk.
\newblock Moral hazard and ambiguity.
\newblock Technical Report~39, Bonn graduate school of economics, 2010.

\bibitem{williams2009dynamic}
N.~Williams.
\newblock On dynamic principal--agent problems in continuous time.
\newblock University of Wisconsin, Madison, 2009.

\bibitem{williams2011persistent}
N.~Williams.
\newblock Persistent private information.
\newblock {\em Econometrica}, 79(4):1233--1275, 2011.

\bibitem{wilson1968theory}
R.~Wilson.
\newblock The theory of syndicates.
\newblock {\em Econometrica}, 36(1):119--132, 1968.

\bibitem{zhang2014monotone}
J.~Zhang and J.~Zhuo.
\newblock Monotone schemes for fully nonlinear parabolic path dependent {PDE}s.
\newblock {\em Journal of Financial Engineering}, 1(1):1450005, 2014.

\bibitem{zhang2009dynamic}
Y.~Zhang.
\newblock Dynamic contracting with persistent shocks.
\newblock {\em Journal of Economic Theory}, 144(2):635--675, 2009.

\end{thebibliography}

 \end{document}